\documentclass[12pt,a4paper, oneside, bold,secthm,seceqn,amsthm,ussrhead,reqno]{article}
\usepackage[utf8]{inputenc}
\usepackage[english]{babel}
\usepackage[symbol]{footmisc}
\usepackage{amssymb,amsmath,amsthm,amsfonts,xcolor,enumerate,hyperref, comment,longtable, cleveref}
\usepackage{verbatim}
\usepackage{times}
\usepackage{cite}
\usepackage{pdflscape}
\usepackage[normalem]{ulem}
\usepackage[mathcal]{euscript}
\usepackage{tikz}
\usepackage{hyperref}
 
\usepackage{cancel}
\usepackage{stmaryrd}
 
\usepackage{amsfonts}
\usepackage{amssymb}
\usepackage{times}
\usepackage{xcolor}
\usepackage{tkz-graph}
\usepackage{url}
\usepackage{float}
\usepackage{tasks}
\usepackage{array}

\usepackage{cite}
\usepackage{hyperref}

 \usepackage{fancyhdr} 
\usepackage{amsfonts}
\usepackage{amsmath}
\usepackage{eurosym}
\usepackage{geometry}

\usepackage{caption,booktabs}

\theoremstyle{plain}
\newtheorem{theorem}{Theorem}
\newtheorem{lemma}[theorem]{Lemma}
\newtheorem{proposition}[theorem]{Proposition}
\newtheorem{remark}[theorem]{Remark}

\newcommand{\unnumberedsection}[2]{\phantomsection\section*{#1}\addcontentsline{toc}{section}{#1}\label{#2}}

\usetikzlibrary{arrows}

\newtheorem*{theoremA}{Theorem A}
\newtheorem*{theoremB}{Theorem B}
\newtheorem*{corollaryA}{Corollary A}
\newtheorem*{corollaryB}{Corollary B}
\usepackage{fouriernc}
\usepackage{mathalfa}
\usepackage{adjustbox}

\usepackage{graphicx}
\usetikzlibrary{arrows.meta,backgrounds}
\begin{document}
 
 \bigskip

\noindent{\Large
Degenerations of Jordan algebras}\footnote{
The work is supported by
FCT 2023.08031.CEECIND and UID/00212/2025.}\footnote{The original concept, theoretical framework, and primary results were developed by V.B., I.K., Y.P., A.S. and Y.V.; 
while Aristotle (https://aristotle.harmonic.fun/)   acted as a non-human research assistant, assisting in error verification and proof simplification.}

 \bigskip

\begin{center}

 {\bf
   Victor Bovdi\footnote{United Arab Emirates University, UAE; \ vbovdi@gmail.com}, 
   Ivan Kaygorodov\footnote{CMA-UBI, University of  Beira Interior, Portugal; \    kaygorodov.ivan@gmail.com},
   Yury Popov\footnote{Independent scientist; \ yuri.ppv@gmail.com}, \\
   Azamat Saydaliyev\footnote{V.\,I.~Romanovskiy Institute of Mathematics, Uzbekistan; \ azamatsaydaliyev6@gmail.com} \&
   Yury Volkov\footnote{Independent scientist; \ wolf86\_666@list.ru}
}

\end{center}

\ 

\noindent {\bf Abstract:}
{\it
We describe all degenerations of complex $4$-dimensional Jordan algebras and of complex $5$-dimensional nilpotent Jordan algebras.
The proofs do not use other papers on degenerations, and all degenerations and non-degenerations were formally verified in the proof assistant Lean~4. The graph of primary degenerations of $4$-dimensional Jordan algebras has $73$ vertices and $160$ edges.
The variety of $5$-dimensional nilpotent Jordan algebras is stratified into $63$ normal forms, five of which are one- or two-parameter families; together with two special members of a family this gives $65$ objects, and the graph of primary degenerations has $154$ edges.
In particular, we describe all irreducible components of the corresponding varieties: the variety of $4$-dimensional Jordan algebras has $10$ irreducible components, and the variety of $5$-dimensional nilpotent Jordan algebras has dimension $22$ and $5$ irreducible components.}

 \bigskip 


\bigskip 

 \
 

	 \bigskip

\ 

\

\tableofcontents 
\newpage
\unnumberedsection{Introduction}{sec_intro}
Deformations and geometric properties of a variety of algebras defined by a family of polynomial identities have been an object of study since the 1970's, see \cite{akt25,ben,fkkv,fks-sl, gkp, GRH,GRH2,GRH3,BB14, BC99, KS07, KE14, km18, M79,M80,S90,kv19,klp20}  and references in \cite{k23,ls,MS}. 
In particular, there are many results concerning degenerations of algebras of low dimensions in a variety defined by a set of identities.
One of important problems in this direction is the description of so-called rigid algebras. These algebras are of big interest, since the closures of their orbits under the action of generalized linear group form irreducible components of a variety under consideration
(with respect to the Zariski topology). For example, the rigid algebras were classified in the varieties of
low-dimensional associative (see, \cite{M79,M80}) and Jordan (see, \cite{KS07,KE14,km18}) algebras.
There are fewer works in which the full information about degenerations was found for some variety of algebras.

  \medskip

The Jordan algebras appeared as one interesting tool for studies in quantum mechanics in
the paper of Jordan, von Neumann and Wigner \cite{jnw}.
A commutative algebra $A$ is called a {\it  Jordan  algebra}  if it satisfies the following identity
$$\big((xx)y\big)x=(xx)(yx).$$ 
The study of the structure theory and other properties of Jordan algebras was initiated by Albert in \cite{albert}.  Due to the wide popularity of Jordan algebras, we do not present the main results known about this variety and suggest only to get acquainted with the works citing the articles of \cite{jnw} and \cite{albert}.

  \medskip

In this paper we give the full information about degenerations of $4$-dimensional Jordan algebras (Section~\ref{sec4}) and of $5$-dimensional nilpotent Jordan algebras (Section~\ref{sec5}).
More precisely, in both cases we construct the graph of primary degenerations. 
The vertices of this graph are the isomorphism classes of algebras (and, in dimension $5$, the families of algebras) in the variety under consideration.
An algebra $A$ degenerates to an algebra $B$ if and only if there is a path from the vertex corresponding to $A$ to the vertex corresponding to $B$.
Also we describe rigid algebras and irreducible components in the varieties of algebras under consideration (Corollaries~A and~B).

\subsection*{Previous results}
The present paper relies only on the algebraic classifications of $4$-dimensional Jordan algebras and of $5$-dimensional nilpotent Jordan algebras. Let us briefly recall which types of results on these (and smaller) Jordan algebras were obtained before.
\begin{itemize}
\item {\it Algebraic classification.} $4$-dimensional nilpotent Jordan algebras were classified in \cite{contr11}, and all $4$-dimensional Jordan algebras over an algebraically closed field were described in \cite{meg}. This list (Table~1 below) is the starting point of Section~\ref{sec4}. The $5$-dimensional nilpotent Jordan algebras were classified by Hegazi and Abdelwahab \cite{ha} (the associative ones, i.e. the commutative associative nilpotent algebras, go back to Mazzola \cite{M80}); this list (Table~4 below, see Remark~\ref{rem:ha}) is the starting point of Section~\ref{sec5}.
\item {\it Geometric classification}, i.e. the description of the rigid algebras and of the irreducible components of the variety: for $3$-dimensional Jordan algebras in \cite{KS07}, for $4$-dimensional Jordan algebras in \cite{KE14} and for $5$-dimensional nilpotent Jordan algebras in \cite{km18}; see also \cite{gkm}.
\item {\it Degenerations of $3$-dimensional Jordan algebras} \cite{gkp}; in particular, a degeneration $\mathfrak{A}\to\mathfrak{B}$ of $3$-dimensional algebras gives the degeneration $\mathfrak{A}\oplus\mathfrak{D}\to\mathfrak{B}\oplus\mathfrak{D}$ of $4$-dimensional algebras.
\item {\it Degenerations of nilpotent algebras} of small dimension, in particular of $4$-dimensional nilpotent Jordan algebras \cite{fkkv}, and of nilpotent commutative associative algebras \cite{klp20}.
\item Degenerations of Jordan superalgebras \cite{ahk19}, irreducible components of the varieties of Jordan superalgebras \cite{fks-js}, and the algebraic and geometric classification of noncommutative Jordan algebras \cite{aak25}.
\item General methods for proving degenerations and non-degenerations \cite{GRH,GRH2,GRH3,BC99,S90}; see also the surveys \cite{MS,k23}.
\end{itemize}
Unfortunately, some of the results of the papers \cite{contr11,gkp,KS07,KE14} on degenerations of $3$- and $4$-dimensional Jordan algebras are wrong or can not be applied to $4$-dimensional algebras (see Remark~\ref{rem_papers}).
For this reason we do not use the results of other papers on degenerations in the proof of our Theorem~A: we give a complete description of degenerations and non-degenerations of $4$-dimensional Jordan algebras ``from scratch''. Every primary degeneration is given by an explicit parametrized basis (Table~2), every primary non-degeneration is proved by an explicit closed set (Table~3), and all these computations were verified in the proof assistant Lean~4.
Theorem~A in particular gives a new, complete proof of the geometric classification of $4$-dimensional Jordan algebras (Corollary~A).
In the same way, for $5$-dimensional nilpotent Jordan algebras every primary degeneration is given by an explicit parametrized basis (Table~5) and every primary non-degeneration by an explicit closed set (Table~6), and all these computations were also verified in Lean~4; Theorem~B gives the complete graph of degenerations, which was not known before: \cite{km18} gives the degenerations needed to determine the irreducible components, but not the closure order between all algebras. Corollary~B agrees with the result of Kashuba and Martin \cite{km18} (see Remark~\ref{rem:km}).

\medskip

\noindent{\it Formalization.} All results of this paper were formally verified in the proof assistant Lean~4 with the Mathlib library; in the formal proofs no result of other papers on degenerations is used. Namely, for both Section~\ref{sec4} and Section~\ref{sec5} it is formally checked that every multiplication table of Tables~1 and~4 (for the families: for all values of the parameters) defines a commutative algebra satisfying the Jordan identity; that every parametrized basis of Tables~2 and~5 gives the stated degeneration; that every non-degeneration of Tables~3 and~6 holds (for the sets $\mathcal{R}$: $\mathcal{R}$ is closed, stable under the Borel subgroup, contains the structure of the first algebra and does not meet the orbits of the second ones, the last condition by explicit Nullstellensatz certificates on the Bruhat cells of ${\rm GL}(\mathbb{V})$; for the polynomial identities and ${\rm Ann}$: directly); the dimensions of the algebras of derivations in Table~1; and the general principles (Lemma~\ref{lem_inv}, transitivity of degenerations, passage from primary degenerations and non-degenerations to the whole graph). The algebraic classifications (Lemmas~\ref{lem_table} and~\ref{lem_table5}) and the classical Lemma~\ref{lem_der} are used as known results.

\section{Preliminaries}\label{prelim}

All spaces in this paper are considered over $\mathbb{C}$, and we write simply ${\rm dim}$, ${\rm Hom}$ and $\otimes$ instead of ${\rm dim}_{\mathbb{C}}$, ${\rm Hom}_{\mathbb{C}}$ and $\otimes_{\mathbb{C}}$. 
Given an $n$-dimensional vector space $\mathbb{V}$, the set ${\rm Hom}(\mathbb{V} \otimes \mathbb{V},\mathbb{V}) \cong \mathbb{V}^* \otimes \mathbb{V}^* \otimes \mathbb{V}$ 
is a vector space of dimension $n^3$. This space has a structure of the affine variety $\mathbb{C}^{n^3}.$ Indeed, let us fix a basis $e_1,\dots,e_n$ of $\mathbb{V}$. Then any $\mu\in {\rm Hom}(\mathbb{V} \otimes \mathbb{V},\mathbb{V})$ is determined by $n^3$ structure constants $c_{i,j}^k\in\mathbb{C}$ such that
$\mu(e_i\otimes e_j)=\sum\limits_{k=1}^nc_{i,j}^ke_k$. A subset of ${\rm Hom}(\mathbb{V} \otimes \mathbb{V},\mathbb{V})$ is {\it Zariski-closed} if it can be defined by a set of polynomial equations in the variables $c_{i,j}^k$ ($1\le i,j,k\le n$).

Let $\mathfrak{Jord}_n$ be the set of all Jordan algebra structures on $\mathbb{V}$; it is a Zariski-closed subset of ${\rm Hom}(\mathbb{V} \otimes \mathbb{V},\mathbb{V})$. 
The general linear group ${\rm GL}(\mathbb{V})$ acts on $\mathfrak{Jord}_n$ by conjugations:
$$ (g * \mu )(x\otimes y) = g\mu(g^{-1}x\otimes g^{-1}y)$$ 
for $x,y\in \mathbb{V}$, $\mu\in \mathfrak{Jord}_n$ and $g\in {\rm GL}(\mathbb{V})$.
Thus, $\mathfrak{Jord}_n$ is decomposed into ${\rm GL}(\mathbb{V})$-orbits that correspond to the isomorphism classes of algebras. 
Let ${\mathcal O}(\mu)$ denote the orbit of $\mu\in \mathfrak{Jord}_n$ under the action of ${\rm GL}(\mathbb{V})$ and $\overline{{\mathcal O}(\mu)}$ denote the Zariski closure of ${\mathcal O}(\mu)$.

Let $A$ and $B$ be two $n$-dimensional Jordan algebras and let $\mu,\lambda \in \mathfrak{Jord}_n$ represent $A$ and $B$ respectively.
We say that $A$ degenerates to $B$ and write $A\to B$ if $\lambda\in\overline{{\mathcal O}(\mu)}$.
Note that in this case we have $\overline{{\mathcal O}(\lambda)}\subset\overline{{\mathcal O}(\mu)}$. 
Hence, the definition of a degeneration does not depend on the choice of $\mu$ and $\lambda$. 
If $A\not\cong B$, then the assertion $A\to B$ is called a {\it proper degeneration}. We write $A\not\to B$ if $\lambda\not\in\overline{{\mathcal O}(\mu)}$.

Let $A$ be represented by $\mu\in \mathfrak{Jord}_n.$ Then  $A$ is  {\it rigid} in $\mathfrak{Jord}_n$ if ${\mathcal O}(\mu)$ is an open subset of $\mathfrak{Jord}_n$.
 Recall that a subset of a variety is called irreducible if it can't be represented as a union of two non-trivial closed subsets. A maximal irreducible closed subset of a variety is called {\it irreducible component}.
 In particular, $A$ is rigid in $\mathfrak{Jord}_n$ iff $\overline{{\mathcal O}(\mu)}$ is an irreducible component of $\mathfrak{Jord}_n$. 
 It is well known that any affine variety can be represented as a finite union of its irreducible components in a unique way. We denote by $Rig(\mathfrak{Jord}_n)$ the set of rigid algebras in $\mathfrak{Jord}_n.$

We denote by ${\rm Ann}(A)$ the annihilator of $A$, by $A^2$ the square of $A$ and by ${\rm Der}(A)$ the algebra of derivations of $A$.
Given spaces $\mathbb{U}$ and $\mathbb{W}$, we write simply $\mathbb{U}>\mathbb{W}$ instead of ${\rm dim}\,\mathbb{U}>{\rm dim}\,\mathbb{W}$.

\begin{lemma}\label{lem_der}
If $A\to B$ is a proper degeneration, then ${\rm Der}(A)<{\rm Der}(B)$.
\end{lemma}
\begin{proof}
The stabilizer of $\mu\in\mathfrak{Jord}_n$ in ${\rm GL}(\mathbb{V})$ is the automorphism group of the algebra, whose Lie algebra is the algebra of derivations; hence ${\rm dim}\, {\mathcal O}(\mu)=n^2-{\rm dim}\,{\rm Der}(\mu)$. The orbit ${\mathcal O}(\mu)$ is open in its closure, and the boundary $\overline{{\mathcal O}(\mu)}\setminus {\mathcal O}(\mu)$ is a closed ${\rm GL}(\mathbb{V})$-stable set of smaller dimension; thus ${\rm dim}\, {\mathcal O}(\lambda)<{\rm dim}\, {\mathcal O}(\mu)$ for $\lambda\in\overline{{\mathcal O}(\mu)}\setminus {\mathcal O}(\mu)$.
\end{proof}

We compute the dimensions of algebras of derivations and check the assertion $A\to B$ only for such $A$ and $B$ that ${\rm Der}(A)<{\rm Der}(B)$. Secondly, if $A\to C$ and $C\to B$ then $A\to B$. If there is no $C$ such that $A\to C$ and $C\to B$ are proper degenerations, then the assertion $A\to B$ is called a {\it primary degeneration}. If ${\rm Der}(A)<{\rm Der}(B)$ and there are no $C$ and $D$ such that $C\to A$, $B\to D$, $C\not\to D$ and one of the assertions $C\to A$ and $B\to D$ is a proper degeneration,  then the assertion $A \not\to B$ is called a {\it primary non-degeneration}. It suffices to prove only primary degenerations and non-degenerations to describe degenerations in the variety under consideration. It is easy to see that any algebra degenerates to the algebra with zero multiplication.
Since the number of orbits of ${\rm GL}(\mathbb{V})$ on $\mathfrak{Jord}_4$ is finite (Table~1), and the variety of $5$-dimensional nilpotent Jordan algebras is a finite union of orbits and of five families of orbits (Table~4), the graph of primary degenerations gives the whole picture; in particular, the rigid algebras and the irreducible components can be easily obtained from it.

\medskip

{\it Degenerations.} To prove primary degenerations, we construct families of matrices parametrized by $t$. Namely, let $A$ and $B$ be two algebras represented by the structures $\mu$ and $\lambda$ from $\mathfrak{Jord}_n$ respectively. Let $e_1,\dots, e_n$ be a basis of $\mathbb{V}$ and $c_{i,j}^k$ ($1\le i,j,k\le n$) be the structure constants of $\lambda$ in this basis. If there exist $a_i^j(t)\in\mathbb{C}$ ($1\le i,j\le n$, $t\in\mathbb{C}^*$) such that ${\rm E}^t_i=\sum\limits_{j=1}^na_i^j(t)e_j$ ($1\le i\le n$) form a basis of $\mathbb{V}$ for any small $t\neq 0$, and the structure constants of $\mu$ in the basis ${\rm E}^t_1,\dots, {\rm E}^t_n$ are such functions $c_{i,j}^k(t)$ that $c_{i,j}^k(t)\to c_{i,j}^k$ as $t\to 0$, then $A\to B$. In this case  ${\rm E}^t_1,\dots, {\rm E}^t_n$ is called a {\it parametrized basis} for $A\to B$.
In Tables~2 and~5 the vectors ${\rm E}^t_i$ are written in the basis $e_1,\dots,e_n$ of the algebra $A$ from Tables~1 and~4.

{\it Families of algebras.} In dimension $5$ the variety contains one- and two-parameter families of algebras $\{\mu^{\alpha}\}_{\alpha}$; a family is always written with its parameter, i.e. as $\mu^{\alpha}$ (or $\mu^{\alpha,\beta}$). We put $\overline{\{{\mathcal O}(\mu^{\alpha})\}}=\overline{\bigcup_\alpha {\mathcal O}(\mu^{\alpha})}$. A degeneration $\lambda\in\overline{\{{\mathcal O}(\mu^{\alpha})\}}$ is written as $\mu^{\alpha(t)}\to\lambda$, where $\mu^{\alpha(t)}$ is a member of the family depending on $t$ such that the structure constants of $\mu^{\alpha(t)}$ in a parametrized basis ${\rm E}^t_1,\dots,{\rm E}^t_n$ tend to the structure constants of $\lambda$ (this is a parametrized basis for the degeneration of the family). For a family target, $\mu\to\lambda^{\alpha}$ (resp. $\mu^{\beta(t)}\to\lambda^{\alpha}$) means that the degeneration holds for every value of $\alpha$. In non-degenerations and in the description of the irreducible components a family stands for the whole family: $\mu^{\alpha}\not\to\lambda$ means $\lambda\notin\overline{\{{\mathcal O}(\mu^{\alpha})\}}$. The definitions of primary degenerations and primary non-degenerations are the same for families. Since an algebra and a family can have different orbit dimensions, in Lemma~\ref{lem_der} one has to compare the dimension of the orbit of $\lambda$ with the dimension of the orbits of the members $\mu^\alpha$: if ${\rm Der}(\mu^\alpha)\ge {\rm Der}(\lambda)$ for all $\alpha$ and $\lambda$ is not isomorphic to a member of the family, then no single member $\mu^{\alpha}$ degenerates to $\lambda$, although $\lambda\in\overline{\{{\mathcal O}(\mu^{\alpha})\}}$ may hold.

\medskip

{\it Non-degenerations.} The main tools are the following two lemmas.

\begin{lemma}\label{main}
Let $\mathcal{B}$ be a Borel subgroup of ${\rm GL}(\mathbb{V})$ and $\mathcal{R}\subset \mathfrak{Jord}_n$ be a closed subset.
If $A \to B$ and $A$ can be represented by a structure $\mu$ such that $b\cdot\mu\in\mathcal{R}$ for all $b\in\mathcal{B}$ (for example, $\mu\in\mathcal{R}$ and $\mathcal{R}$ is $\mathcal{B}$-stable), then there is $\lambda\in \mathcal{R}$ that represents $B$.
\end{lemma}
\begin{proof}
The set $\mathcal{R}'=\{\nu\mid b\cdot\nu\in\mathcal{R}\mbox{ for all }b\in\mathcal{B}\}=\bigcap_{b\in\mathcal{B}}b^{-1}\cdot\mathcal{R}$ is closed, $\mathcal{B}$-stable, contains $\mu$ and is contained in $\mathcal{R}$. The set ${\rm GL}(\mathbb{V})\cdot\mathcal{R}'$ is the image of the closed subset $\big\{(g\mathcal{B},\nu)\mid g^{-1}\nu\in\mathcal{R}'\big\}$ of ${\rm GL}(\mathbb{V})/\mathcal{B}\times\mathfrak{Jord}_n$ under the projection to the second factor. Since the flag variety ${\rm GL}(\mathbb{V})/\mathcal{B}$ is complete, this projection is closed, so ${\rm GL}(\mathbb{V})\cdot\mathcal{R}'$ is a closed set containing ${\mathcal O}(\mu)$; hence it contains $\overline{{\mathcal O}(\mu)}\ni\lambda'$ for every structure $\lambda'$ representing $B$, i.e. $g\cdot\lambda'\in\mathcal{R}'\subseteq\mathcal{R}$ for some $g\in {\rm GL}(\mathbb{V})$.
\end{proof}

In the description of the sets $\mathcal{R}$ we write $A_i=\langle f_i,\dots,f_n\rangle$ ($1\le i\le n$) and $A_{n+1}=0$, where $f_1,\dots,f_n$ is a basis of $\mathbb{V}$ and $c_{ij}^k$ are the structure constants in this basis; thus $A_iA_j\subseteq A_k$ means that $c_{i'j'}^{l}=0$ for all $i'\ge i$, $j'\ge j$, $l<k$. The other conditions are the vanishing of some structure constants and polynomial relations between them. Let $\mathcal{B}$ be the Borel subgroup of upper triangular matrices with respect to the basis $f_1,\dots,f_n$, i.e. of the changes of basis $f_i\mapsto \sum_{j\ge i} b_{ij}f_j$ with $b_{ii}\neq 0$. We omit the verification of the facts that $b\cdot\mu\in\mathcal{R}$ for all $b\in\mathcal{B}$ (where $\mu$ is the structure of $A$ in the indicated basis) and that $\lambda\not\in\mathcal{R}$ for any choice of a basis of $\mathbb{V}$; these verifications are direct calculations, and for Table~3 all of them were verified in Lean.

\begin{lemma}\label{lem_inv}
Let $\Phi(x_1,\dots,x_m)$ be a linear combination of (nonassociative) products of elements $x_1,\dots,x_m$ of an algebra (a polynomial identity).
Then the set $\mathcal{R}$ of the structures $\mu\in\mathfrak{Jord}_n$ such that $\Phi$ vanishes identically on $(\mathbb{V},\mu)$ is closed and ${\rm GL}(\mathbb{V})$-stable. The sets $\big\{\mu \mid {\rm dim}\,{\rm Ann}(\mu)\geq k\big\}$ and $\big\{\mu\mid {\rm dim}\,\mu(\mathbb{V},\mathbb{V})\leq k\big\}$ are also closed and ${\rm GL}(\mathbb{V})$-stable.
\end{lemma}
\begin{proof}
The set of structures $\mu$ such that $\Phi$ vanishes identically on $(\mathbb{V},\mu)$ is the intersection, over all $x_1,\dots,x_m\in\mathbb{V}$, of the zero sets of polynomial functions of $\mu$; hence it is closed. It is ${\rm GL}(\mathbb{V})$-stable, since an isomorphism of algebras maps products to products. The last two sets are given by the vanishing of minors of matrices whose entries are structure constants, and they are clearly ${\rm GL}(\mathbb{V})$-stable.
\end{proof}

In particular, the sets of Lemma~\ref{lem_inv} can be used in Lemma~\ref{main}; for example, $A\not\to B$ if ${\rm Ann}(A)>{\rm Ann}(B)$ or $A^{2}<B^{2}$.

{\it The proofs of non-degenerations.} The non-degenerations are proved in the same way in both dimensions. By Lemma~\ref{lem_der}, we only have to consider the pairs $A$, $B$ with ${\rm Der}(A)<{\rm Der}(B)$. Moreover, it is enough to prove the primary non-degenerations, since all the other ones follow by transitivity: if $A\not\to B$, then $C\not\to B$ for every $C$ with $A\to C$, and $A\not\to D$ for every $D$ with $D\to B$. In the tables of primary non-degenerations (Tables~3 and~6) they are grouped by the first algebra: the row of an algebra $A$ lists algebras $B$ such that $A\not\to B$ is a primary non-degeneration, and gives one closed set $\mathcal{R}$ which contains a structure $\mu$ representing $A$ with $b\cdot\mu\in\mathcal{R}$ for all $b\in\mathcal{B}$ (for a family: such a structure for every member of the family), but no structure representing any of these algebras $B$; by Lemma~\ref{main}, $A\not\to B$. The set $\mathcal{R}$ is of one of the following two kinds.
\begin{itemize}
\item $\mathcal{R}$ is one of the ${\rm GL}(\mathbb{V})$-stable sets of Lemma~\ref{lem_inv}, defined by a polynomial identity or by a condition on ${\rm Ann}$. Then $A\in\mathcal{R}$ in any basis, and the indicated elements of $B$ show that $B\not\in\mathcal{R}$.
\item $\mathcal{R}$ is defined in a basis $f_1,\dots,f_n$ of $\mathbb{V}$ as explained above. The condition $A=\langle f_1,\dots,f_n\rangle$ is included in every such set $\mathcal{R}$ and is not repeated in the tables. The basis of $A$ in which $A\in\mathcal{R}$ is indicated in terms of the basis $e_1,\dots,e_n$ of the multiplication table of $A$ (Tables~1 and~4); if no basis is indicated, then $f_i=e_i$.
\end{itemize}
In the second case, to show that $B$ admits no basis in which its structure lies in $\mathcal{R}$, we use the Bruhat decomposition: every invertible matrix can be written as $P_wUb$, where $P_w$ is one of the $n!$ permutation matrices, $U$ is unitriangular and $b\in\mathcal{B}$. Since $\mathcal{R}$ is $\mathcal{B}$-stable, it is enough to show that for every $w$ the structure constants of $B$ in the basis given by $P_wU$, which are polynomials in the entries of $U$, do not satisfy the equations of $\mathcal{R}$; this is done by Nullstellensatz certificates or by Gr\"obner basis computations.

\section{Degenerations of $4$-dimensional Jordan algebras}\label{sec4}

Thanks to the Wedderburn decomposition and the Albert--Penico--Taft theorem, every finite-dimensional Jordan algebra is constructed from a nilpotent radical and a semisimple subalgebra.
In Table 1 we list all $4$-dimensional Jordan algebras.

\begin{lemma}\label{lem_table}
Every complex $4$-dimensional Jordan algebra is isomorphic to one of the algebras of Table~1, where $\mathfrak{J}_{73}=\mathbb{C}^4$ is the algebra with zero multiplication (the algebraic classification {\rm \cite{meg}}, based on {\rm \cite{contr11}}). Here $e_1,\dots,e_4$ is a basis of the algebra, the products of basis elements which are not listed are zero, and the columns $\mathfrak{Der}$, ${\rm Ann}$, $\mathbb{A}^2$, ${\rm Rad}$ give the dimensions of the algebra of derivations, of the annihilator, of the square and of the radical.
\end{lemma}
\begingroup
\small
\setlength{\tabcolsep}{4pt}
\renewcommand{\arraystretch}{1.2}
\begin{longtable}{|l|>{\raggedright\arraybackslash}p{10.6cm}|c|c|c|c|}
\multicolumn{6}{c}{Table 1. {\it $4$-dimensional Jordan algebras.}}\\
\hline
$\mathbb{A}$ & \mbox{multiplication table} & $\mathfrak{Der}$ & ${\rm Ann}$ & $\mathbb{A}^2$ & ${\rm Rad}$\\
\hline
\hline
\endhead

$\mathfrak{J}_{01}$ &
$e_1e_1  = e_1$, $e_1e_3 = \frac{1}{2}e_3$, $e_2e_2  =e_2$,\newline $e_2e_3 = \frac{1}{2}e_3$, $e_3e_3 =e_1+e_2$, $e_4e_4 = e_4$
& 1 & 0 & 4 & 0\\
\hline
$\mathfrak{J}_{02}$ &
$e_1e_1  = e_1$, $e_1e_2 = e_2$, $e_1e_3 = e_3$,\newline $e_1e_4 = e_4$, $e_2e_2  = e_1$, $e_3e_3  = e_1$, $e_4e_4  = e_1$
& 3 & 0 & 4 & 0\\
\hline
$\mathfrak{J}_{03}$ &
$e_1e_1  = e_1$, $e_2e_2  = e_2$, $e_3e_3  = e_3$, $e_4e_4  = e_4$
& 0 & 0 & 4 & 0\\
\hline
$\mathfrak{J}_{04}$ &
$e_1e_1 =e_1$, $e_1e_4 =e_4$, $e_2e_2 =e_2$, $e_3e_3 =e_3$
& 1 & 0 & 4 & 1\\
\hline
$\mathfrak{J}_{05}$ &
$e_1e_1  = e_1$, $e_2e_2  = e_2$, $e_3e_3  = e_3$
& 1 & 1 & 3 & 1\\
\hline
$\mathfrak{J}_{06}$ &
$e_1e_1 =e_1$, $e_1e_4 = \frac{1}{2}e_4$, $e_2e_2 =e_2$, $e_3e_3 =e_3$
& 2 & 0 & 4 & 1\\
\hline
$\mathfrak{J}_{07}$ &
$e_1e_1  = e_1$, $e_1e_4 = \frac{1}{2}e_4$, $e_2e_2  = e_2$, $e_2e_4 = \frac{1}{2}e_4$, $e_3e_3 =e_3$
& 2 & 0 & 4 & 1\\
\hline
$\mathfrak{J}_{08}$ &
$e_1e_1  = e_1$, $e_1e_3= \frac{1}{2}e_3$, $e_2e_2  = e_2$, $e_2e_3 = \frac{1}{2}e_3$, $e_3e_3 =e_1+e_2$
& 2 & 1 & 3 & 1\\
\hline
$\mathfrak{J}_{09}$ &
$e_1e_1  = e_1$, $e_1e_3 = \frac{1}{2}e_3$, $e_1e_4=\frac{1}{2}e_4$, $e_2e_2  = e_2$,\newline $e_2e_3= \frac{1}{2}e_3$, $e_2e_4  = \frac{1}{2}e_4$, $e_3e_3 =e_1+e_2$
& 4 & 0 & 4 & 1\\
\hline
$\mathfrak{J}_{10}$ &
$e_1e_1  = e_1$, $e_1e_3 = \frac{1}{2}e_3$, $e_2e_2 =e_2$
& 3 & 1 & 3 & 2\\
\hline
$\mathfrak{J}_{11}$ &
$e_1e_1  = e_1$, $e_1e_3 = \frac{1}{2}e_3$, $e_2e_2 =e_2$, $e_2e_3 = \frac{1}{2}e_3$
& 3 & 1 & 3 & 2\\
\hline
$\mathfrak{J}_{12}$ &
$e_1e_1  = e_1$, $e_1e_3 = \frac{1}{2}e_3$, $e_1e_4 = \frac{1}{2}e_4$, $e_2e_2 =e_2$
& 6 & 0 & 4 & 2\\
\hline
$\mathfrak{J}_{13}$ &
$e_1e_1  = e_1$, $e_1e_3 =\frac{1}{2}e_3$, $e_2e_2 =e_2$, $e_2e_4 =\frac{1}{2}e_4$
& 4 & 0 & 4 & 2\\
\hline
$\mathfrak{J}_{14}$ &
$e_1e_1  = e_1$, $e_1e_3 =\frac{1}{2}e_3$, $e_1e_4 = e_4$, $e_2e_2 =e_2$
& 3 & 0 & 4 & 2\\
\hline
$\mathfrak{J}_{15}$ &
$e_1e_1  = e_1$, $e_1e_3 =e_3$, $e_2e_2 =e_2$, $e_2e_4 =\frac{1}{2}e_4$
& 3 & 0 & 4 & 2\\
\hline
$\mathfrak{J}_{16}$ &
$e_1e_1  = e_1$, $e_1 e_3 =  \frac{1}{2}e_3$, $e_1e_4 = \frac{1}{2}e_4$, $e_2e_2 =e_2$, $e_2e_3 = \frac{1}{2}e_3$
& 4 & 0 & 4 & 2\\
\hline
$\mathfrak{J}_{17}$ &
$e_1e_1 =e_1$, $e_1e_3 = \frac{1}{2}e_3$, $e_1e_4 = e_4$, $e_2e_2 =e_2$, $e_2e_3 = \frac{1}{2}e_3$
& 3 & 0 & 4 & 2\\
\hline
$\mathfrak{J}_{18}$ &
$e_1e_1  = e_1$, $e_1e_3 =  \frac{1}{2}e_3$, $e_1e_4 = \frac{1}{2}e_4$,\newline $e_2e_2  = e_2$, $e_2e_3 = \frac{1}{2}e_3$, $e_2e_4 = \frac{1}{2}e_4$
& 6 & 0 & 4 & 2\\
\hline
$\mathfrak{J}_{19}$ &
$e_1e_1  = e_1$, $e_2e_2  = e_2$
& 4 & 2 & 2 & 2\\
\hline
$\mathfrak{J}_{20}$ &
$e_1e_1  = e_1$, $e_1e_3 =e_3$, $e_2e_2  = e_2$
& 2 & 1 & 3 & 2\\
\hline
$\mathfrak{J}_{21}$ &
$e_1e_1  = e_1$, $e_1e_3 =e_3$, $e_1e_4=e_4$, $e_2e_2  = e_2$
& 4 & 0 & 4 & 2\\
\hline
$\mathfrak{J}_{22}$ &
$e_1e_1  = e_1$, $e_1e_3 =e_3$, $e_2e_2  = e_2$, $e_2e_4 = e_4$
& 2 & 0 & 4 & 2\\
\hline
$\mathfrak{J}_{23}$ &
$e_1e_1  = e_1$, $e_1e_3 = \frac{1}{2}e_3$, $e_2e_2  = e_2$, $e_3e_3  = e_4$
& 2 & 1 & 4 & 2\\
\hline
$\mathfrak{J}_{24}$ &
$e_1e_1  = e_1$, $e_1e_3 = \frac{1}{2}e_3$, $e_1e_4 = e_4$, $e_2e_2  = e_2$, $e_3e_3  = e_4$
& 2 & 0 & 4 & 2\\
\hline
$\mathfrak{J}_{25}$ &
$e_1e_1  = e_1$, $e_1e_4  = e_4$, $e_1e_3=\frac{1}{2}e_3$, $e_2e_2  = e_2$, $e_2e_3 = \frac{1}{2}e_3$, $e_3e_3  = e_4$
& 2 & 0 & 4 & 2\\
\hline
$\mathfrak{J}_{26}$ &
$e_1e_1  = e_1$, $e_2e_2  = e_2$, $e_3e_3  = e_4$
& 2 & 1 & 3 & 2\\
\hline
$\mathfrak{J}_{27}$ &
$e_1e_1  = e_1$, $e_1e_3 = e_3$, $e_1e_4=e_4$, $e_2e_2  = e_2$, $e_3e_3  = e_4$
& 2 & 0 & 4 & 2\\
\hline
$\mathfrak{J}_{28}$ &
$e_1e_1  = e_1$, $e_1e_2 =\frac{1}{2}e_2$
& 6 & 2 & 2 & 3\\
\hline
$\mathfrak{J}_{29}$ &
$e_1e_1  = e_1$, $e_1e_2 =\frac{1}{2}e_2$, $e_1e_3 = e_3$
& 4 & 1 & 3 & 3\\
\hline
$\mathfrak{J}_{30}$ &
$e_1e_1  = e_1$, $e_1e_2  = \frac{1}{2}e_2$, $e_1e_3 = \frac{1}{2}e_3$
& 7 & 1 & 3 & 3\\
\hline
$\mathfrak{J}_{31}$ &
$e_1e_1  = e_1$, $e_1e_2 = e_2$, $e_1e_3 = e_3$, $e_1e_4 = \frac{1}{2}e_4$
& 6 & 0 & 4 & 3\\
\hline
$\mathfrak{J}_{32}$ &
$e_1e_1  = e_1$, $e_1e_2 = \frac{1}{2}e_2$, $e_1e_3 = \frac{1}{2}e_3$, $e_1e_4 = e_4$
& 7 & 0 & 4 & 3\\
\hline
$\mathfrak{J}_{33}$ &
$e_1e_1  = e_1$, $e_1e_2 = \frac{1}{2}e_2$, $e_1e_3 = \frac{1}{2}e_3$, $e_1e_4 = \frac{1}{2}e_4$
& 12 & 0 & 4 & 3\\
\hline
$\mathfrak{J}_{34}$ &
$e_1e_1  = e_1$
& 9 & 3 & 1 & 3\\
\hline
$\mathfrak{J}_{35}$ &
$e_1e_1  = e_1$, $e_1e_2 =e_2$
& 5 & 2 & 2 & 3\\
\hline
$\mathfrak{J}_{36}$ &
$e_1e_1  = e_1$, $e_1e_2 = e_2$, $e_1e_3 = e_3$, $e_1e_4 = e_4$
& 9 & 0 & 4 & 3\\
\hline
$\mathfrak{J}_{37}$ &
$e_1e_1  = e_1$, $e_1e_2 = e_2$, $e_1e_3=e_3$
& 5 & 1 & 3 & 3\\
\hline
$\mathfrak{J}_{38}$ &
$e_1e_1  = e_1$, $e_2e_2  = e_3$, $e_2e_3 =e_4$
& 3 & 1 & 3 & 3\\
\hline
$\mathfrak{J}_{39}$ &
$e_1e_1  = e_1$, $e_1e_2 = e_2$, $e_1e_3 = e_3$, $e_1e_4 = e_4$, $e_2e_2  = e_3$, $e_2e_3 = e_4$
& 3 & 0 & 4 & 3\\
\hline
$\mathfrak{J}_{40}$ &
$e_1e_1  = e_1$, $e_2e_2  = e_3$
& 5 & 2 & 2 & 3\\
\hline
$\mathfrak{J}_{41}$ &
$e_1e_1  = e_1$, $e_2e_2  =e_3$, $e_4e_4  = e_3$
& 4 & 1 & 2 & 3\\
\hline
$\mathfrak{J}_{42}$ &
$e_1e_1  = e_1$, $e_1e_2 = e_2$, $e_1e_3 = e_3$, $e_1e_4 =e_4$, $e_2e_2  = e_3$
& 5 & 0 & 4 & 3\\
\hline
$\mathfrak{J}_{43}$ &
$e_1e_1  = e_1$, $e_1e_2 = e_2$, $e_1e_3 = e_3$, $e_1e_4 = e_4$, $e_2e_2  = e_3$, $e_2e_4 = e_3$
& 4 & 0 & 4 & 3\\
\hline
$\mathfrak{J}_{44}$ &
$e_1e_1  = e_1$, $e_1e_2 = \frac{1}{2}e_2$, $e_2e_2  = e_3$
& 4 & 2 & 3 & 3\\
\hline
$\mathfrak{J}_{45}$ &
$e_1e_1  = e_1$, $e_1e_2 = \frac{1}{2}e_2$, $e_2e_2  = e_3$, $e_4e_4  = e_3$
& 3 & 1 & 3 & 3\\
\hline
$\mathfrak{J}_{46}$ &
$e_1e_1  = e_1$, $e_1e_2 =\frac{1}{2}e_2$, $e_3e_3  = e_4$
& 4 & 1 & 3 & 3\\
\hline
$\mathfrak{J}_{47}$ &
$e_1e_1  = e_1$, $e_1e_2 =e_2$, $e_3e_3  = e_4$
& 3 & 1 & 3 & 3\\
\hline
$\mathfrak{J}_{48}$ &
$e_1e_1  = e_1$, $e_1e_3 = \frac{1}{2}e_3$, $e_1e_4 = \frac{1}{2}e_4$, $e_4e_4  = e_2$
& 5 & 1 & 4 & 3\\
\hline
$\mathfrak{J}_{49}$ &
$e_1e_1  = e_1$, $e_1e_3 = \frac{1}{2}e_3$, $e_1e_4 = \frac{1}{2}e_4$, $e_4e_4  =e_2$, $e_3e_4 =  e_2$
& 4 & 1 & 4 & 3\\
\hline
$\mathfrak{J}_{50}$ &
$e_1e_1  = e_1$, $e_1e_3 = \frac{1}{2}e_3$, $e_1e_4 = \frac{1}{2}e_4$, $e_2e_3 = e_4$
& 5 & 0 & 3 & 3\\
\hline
$\mathfrak{J}_{51}$ &
$e_1e_1  = e_1$, $e_1e_2 = \frac{1}{2}e_2$, $e_1e_3 = e_3$, $e_2e_2  = e_3$
& 3 & 1 & 3 & 3\\
\hline
$\mathfrak{J}_{52}$ &
$e_1e_1  = e_1$, $e_1e_2 = \frac{1}{2}e_2$, $e_1e_3 = e_3$, $e_2e_2  = e_4$
& 3 & 1 & 4 & 3\\
\hline
$\mathfrak{J}_{53}$ &
$e_1e_1  = e_1$, $e_1e_2 = \frac{1}{2}e_2$, $e_1e_3 = e_3$, $e_2e_2  = e_3 + e_4$
& 2 & 1 & 4 & 3\\
\hline
$\mathfrak{J}_{54}$ &
$e_1e_1  = e_1$, $e_1e_2 = e_2$, $e_1e_3=e_3$, $e_2e_2  = e_3$
& 3 & 1 & 3 & 3\\
\hline
$\mathfrak{J}_{55}$ &
$e_1e_1  = e_1$, $e_1e_2 = e_2$, $e_1e_3 = e_3$, $e_1e_4 = \frac{1}{2}e_4$, $e_4e_4  = e_3$
& 4 & 0 & 4 & 3\\
\hline
$\mathfrak{J}_{56}$ &
$e_1e_1  = e_1$, $e_1e_2 = e_2$, $e_1e_3 = e_3$, $e_1e_4 = \frac{1}{2}e_4$, $e_2e_2  = e_3$
& 4 & 0 & 4 & 3\\
\hline
$\mathfrak{J}_{57}$ &
$e_1e_1  = e_1$, $e_1e_2 = e_2$, $e_1e_3 = e_3$, $e_1e_4 = \frac{1}{2}e_4$, $e_2e_2  = e_3$, $e_4e_4  = e_3$
& 3 & 0 & 4 & 3\\
\hline
$\mathfrak{J}_{58}$ &
$e_1e_1  = e_1$, $e_1e_2 =e_2$, $e_4e_4  = e_2$, $e_1e_3 = \frac{1}{2}e_3$, $e_1e_4 = \frac{1}{2}e_4$
& 5 & 0 & 4 & 3\\
\hline
$\mathfrak{J}_{59}$ &
$e_1e_1  = e_1$, $e_1e_2 =e_2$, $e_4e_4  = e_2$, $e_3e_4 = e_2$, $e_1e_3 = \frac{1}{2}e_3$, $e_1e_4 = \frac{1}{2}e_4$
& 4 & 0 & 4 & 3\\
\hline
$\mathfrak{J}_{60}$ &
$e_1e_1  = e_1$, $e_1e_2 = e_2$, $e_1e_3 = \frac{1}{2}e_3$, $e_1e_4 = \frac{1}{2}e_4$, $e_2e_3 = e_4$
& 5 & 0 & 4 & 3\\
\hline
$\mathfrak{J}_{61}$ &
$e_1e_1  = e_2$, $e_1e_2 = e_3$, $e_1e_3 = e_4$, $e_2e_2  = e_4$
& 4 & 1 & 3 & 4\\
\hline
$\mathfrak{J}_{62}$ &
$e_1e_1  = e_2$, $e_1e_2 = e_3$, $e_4e_4  = e_2$
& 3 & 1 & 2 & 4\\
\hline
$\mathfrak{J}_{63}$ &
$e_1e_1  = e_2$, $e_1e_2 = e_3$, $e_2e_4 = e_3$, $e_4e_4  = -e_2-e_3$
& 4 & 1 & 2 & 4\\
\hline
$\mathfrak{J}_{64}$ &
$e_1e_1  = e_2$, $e_4e_4  = -e_2$, $e_1e_2 = e_3$, $e_2e_4 = e_3$
& 5 & 1 & 2 & 4\\
\hline
$\mathfrak{J}_{65}$ &
$e_1e_1  = e_2$, $e_1e_2 =e_3$, $e_2e_4 = e_3$
& 4 & 1 & 2 & 4\\
\hline
$\mathfrak{J}_{66}$ &
$e_1e_1  = e_2$, $e_1e_2 = e_3$, $e_4e_4  = e_3$
& 5 & 1 & 2 & 4\\
\hline
$\mathfrak{J}_{67}$ &
$e_1e_1  = e_2$, $e_1e_2 = e_3$
& 6 & 2 & 2 & 4\\
\hline
$\mathfrak{J}_{68}$ &
$e_1e_1  =e_2$, $e_3e_3  = e_4$
& 6 & 2 & 2 & 4\\
\hline
$\mathfrak{J}_{69}$ &
$e_1e_1  =  e_2$, $e_1e_3 = e_4$
& 7 & 2 & 2 & 4\\
\hline
$\mathfrak{J}_{70}$ &
$e_1e_1  = e_2$, $e_3e_4 = e_2$
& 7 & 1 & 1 & 4\\
\hline
$\mathfrak{J}_{71}$ &
$e_1e_1 =e_2$, $e_3e_3  = e_2$
& 8 & 2 & 1 & 4\\
\hline
$\mathfrak{J}_{72}$ &
$e_1e_1  = e_2$
& 10 & 3 & 1 & 4\\
\hline
\end{longtable}
\endgroup

\begin{theoremA}\label{third}\label{theorem}
The graph of primary degenerations for $4$-dimensional Jordan algebras has the form given in the figure below (the rows of the graph correspond to the dimension of the algebra of derivations): an algebra $\mathfrak{A}$ degenerates to an algebra $\mathfrak{B}$ if and only if there is a path from $\mathfrak{A}$ to $\mathfrak{B}$ in the graph. The order consists of $749$ degenerations and $4507$ non-degenerations between the $73$ algebras: the $160$ primary degenerations (Table~2) give the other $589$ degenerations by transitivity; the $106$ primary non-degenerations (Table~3) give $1407$ further non-degenerations, and the remaining $2994$ non-degenerations follow from Lemma~\ref{lem_der}.
\end{theoremA}

\begin{proof}
By Lemma~\ref{lem_der} and transitivity, it is enough to prove that the arrows of the graph are degenerations (the primary degenerations) and that $A\not\to B$ for all primary non-degenerations $A\not\to B$.

\medskip

\noindent{\bf Degenerations.} The list of primary degenerations has $160$ elements. For each of them a parametrized basis is given in Table~2 (for the degenerations $\mathfrak{J}_{33}\to\mathfrak{J}_{73}$ and $\mathfrak{J}_{72}\to\mathfrak{J}_{73}$ to the algebra with zero multiplication one can take ${\rm E}^t_i=te_i$).

\begingroup
\small
\begin{longtable}{|c|>{\raggedright\arraybackslash}p{12.6cm}|}
\caption*{Table 2. {\it Primary degenerations of $4$-dimensional Jordan algebras.}}\\
\hline
\mbox{Degeneration} & \mbox{Parametrized basis} \\
\hline
\endfirsthead
\hline
\mbox{Degeneration} & \mbox{Parametrized basis} \\
\hline
\endhead
$\mathfrak{J}_{01} \to \mathfrak{J}_{07}$ & ${\rm E}^t_1=e_1$, \ ${\rm E}^t_2=e_2$, \ ${\rm E}^t_3=e_4$, \ ${\rm E}^t_4=te_3$ \\
\hline
$\mathfrak{J}_{01} \to \mathfrak{J}_{08}$ & ${\rm E}^t_1=e_1$, \ ${\rm E}^t_2=e_2$, \ ${\rm E}^t_3=e_3$, \ ${\rm E}^t_4=te_4$ \\
\hline
$\mathfrak{J}_{01} \to \mathfrak{J}_{23}$ & ${\rm E}^t_1=e_1$, \ ${\rm E}^t_2=e_4$, \ ${\rm E}^t_3=te_3$, \ ${\rm E}^t_4=t^{2}e_2$ \\
\hline
$\mathfrak{J}_{01} \to \mathfrak{J}_{25}$ & ${\rm E}^t_1=e_1 + e_4$, \ ${\rm E}^t_2=e_2$, \ ${\rm E}^t_3=te_3$, \ ${\rm E}^t_4=-t^{2}e_4$ \\
\hline
$\mathfrak{J}_{01} \to \mathfrak{J}_{53}$ & ${\rm E}^t_1=e_1 + e_4$, \ ${\rm E}^t_2=te_3$, \ ${\rm E}^t_3=-t^{2}e_4$, \ ${\rm E}^t_4=t^{2}e_2$ \\
\hline
$\mathfrak{J}_{02} \to \mathfrak{J}_{09}$ & ${\rm E}^t_1=\frac{1}{2}e_1 + \frac{1}{2}e_2$, \ ${\rm E}^t_2=\frac{1}{2}e_1 - \frac{1}{2}e_2$, \ ${\rm E}^t_3=e_3$, \ ${\rm E}^t_4=te_4$ \\
\hline
$\mathfrak{J}_{02} \to \mathfrak{J}_{49}$ & ${\rm E}^t_1=\frac{1}{2}e_1 + \frac{1}{2}e_2$, \ ${\rm E}^t_2=(\frac{1}{2}t^{4}+{\bf i}t)e_1$,\newline ${\rm E}^t_3=t^{2}e_4$, \ ${\rm E}^t_4=t^{-1}e_3 + (\frac{1}{2}t^{2}+{\bf i}t^{-1})e_4$ \\
\hline
$\mathfrak{J}_{03} \to \mathfrak{J}_{04}$ & ${\rm E}^t_1=e_1 + e_2$, \ ${\rm E}^t_2=e_3$, \ ${\rm E}^t_3=e_4$, \ ${\rm E}^t_4=te_2$ \\
\hline
\end{longtable}
\endgroup

\medskip

\noindent{\bf Non-degenerations.} The list of primary non-degenerations has $106$ elements; they are given in Table~3, and the proofs follow the scheme described in Section~\ref{prelim} (all of them were verified in Lean). For $\mathfrak{J}_{03}$, $\mathfrak{J}_{05}$, $\mathfrak{J}_{07}$, $\mathfrak{J}_{19}$ and $\mathfrak{J}_{49}$ the set $\mathcal{R}$ is ${\rm GL}(\mathbb{V})$-stable (Lemma~\ref{lem_inv}); in all other rows it is defined in a basis $f_1,\dots,f_4$ of $\mathbb{V}$.

\begingroup
\footnotesize
\renewcommand{\arraystretch}{1.3}\setlength{\tabcolsep}{2pt}
\begin{longtable}{|l|l|}
\caption*{Table 3. {\it Primary non-degenerations of $4$-dimensional Jordan algebras.}}\\
\hline
\mbox{Non-degeneration}  &  \mbox{Reason}\\
\hline
\endfirsthead
\hline
\mbox{Non-degeneration}  &  \mbox{Reason}\\
\hline
\endhead
$\mathfrak{J}_{01}\not\to \left\{\begin{array}{l}\mathfrak{J}_{30}, \mathfrak{J}_{31},\\ \mathfrak{J}_{32}, \mathfrak{J}_{33},\\ \mathfrak{J}_{35}, \mathfrak{J}_{61}\end{array}\right\}$ & $\begin{array}{@{}l@{}}
\mathcal{R}=\left\{A\ \middle|\ \begin{array}{l}A_{1}A_{4}\subseteq A_{4},\ A_{2}^2\subseteq A_{2},\ c_{23}^{2}=0,\\
2c_{12}^{2}+2c_{13}^{3}=c_{11}^{1},\ c_{22}^{2}+c_{23}^{3}=2c_{12}^{1},\ c_{33}^{3}=2c_{13}^{1},\\
c_{11}^{2}c_{24}^{4}+c_{11}^{3}c_{34}^{4}+c_{11}^{4}c_{44}^{4}+2c_{12}^{2}c_{14}^{4}+2c_{13}^{3}c_{14}^{4}=(c_{14}^{4})^{2},\\
c_{11}^{2}c_{12}^{2}+c_{11}^{3}c_{13}^{2}=0,\ c_{11}^{2}c_{12}^{3}+c_{11}^{3}c_{13}^{3}=0,\\
c_{22}^{2}c_{24}^{4}+c_{22}^{3}c_{34}^{4}+c_{22}^{4}c_{44}^{4}=(c_{24}^{4})^{2},\ c_{33}^{3}c_{34}^{4}+c_{33}^{4}c_{44}^{4}=(c_{34}^{4})^{2}\end{array}\right\}\\
\mathfrak{J}_{01}\in\mathcal{R} \mbox{ (take $f_1=e_{3}$, $f_2=e_{1}$, $f_3=e_{2}$, $f_4=e_{4}$)},\\
\mbox{but } \mathfrak{B}\cap\mathcal{R}=\varnothing \mbox{ for all $\mathfrak{B}$ in the left column}
\end{array}$ \\
\hline
$\mathfrak{J}_{02}\not\to \left\{\begin{array}{l}\mathfrak{J}_{32}, \mathfrak{J}_{33},\\ \mathfrak{J}_{34}, \mathfrak{J}_{69}\end{array}\right\}$ & $\begin{array}{@{}l@{}}
\mathcal{R}=\left\{A\ \middle|\ \begin{array}{l}A_{2}^2\subseteq A_{2},\ c_{11}^{2}=c_{11}^{3}=c_{12}^{3}=c_{13}^{2}=c_{14}^{2}=c_{14}^{4}=0,\\
c_{22}^{3}=c_{24}^{4}=c_{33}^{2}=c_{34}^{4}=c_{44}^{2}=0,\ 2c_{13}^{3}=c_{11}^{1},\\
c_{22}^{2}=2c_{12}^{1},\ c_{33}^{3}=2c_{13}^{1},\ c_{34}^{3}=c_{14}^{1},\ c_{44}^{4}=c_{14}^{1}\end{array}\right\}\\
\mathfrak{J}_{02}\in\mathcal{R} \mbox{ (take $f_1=e_{2}$, $f_2=e_{3}$, $f_3=e_{4}$, $f_4=e_{1}$)},\\
\mbox{but } \mathfrak{B}\cap\mathcal{R}=\varnothing \mbox{ for all $\mathfrak{B}$ in the left column}
\end{array}$ \\
\hline
$\mathfrak{J}_{03}\not\to \left\{\begin{array}{l}\mathfrak{J}_{28}, \mathfrak{J}_{30},\\ \mathfrak{J}_{31}, \mathfrak{J}_{32},\\ \mathfrak{J}_{33}, \mathfrak{J}_{64},\\ \mathfrak{J}_{65}\end{array}\right\}$ & $\begin{array}{@{}l@{}}
\mathcal{R}=\big\{A\mid (xy)x=(xx)y \mbox{ for all $x,y\in A$}\big\}\\
\mathfrak{J}_{03}\in\mathcal{R},\\
\mbox{but } \mathfrak{B}\cap\mathcal{R}=\varnothing \mbox{ for all $\mathfrak{B}$ in the left column:}\\
\mbox{take $x=e_1$, $y=e_2$ for $\mathfrak{J}_{28}$, $\mathfrak{J}_{30}$, $\mathfrak{J}_{32}$, $\mathfrak{J}_{33}$;}\\
\mbox{take $x=e_1$, $y=e_4$ for $\mathfrak{J}_{31}$, $\mathfrak{J}_{64}$, $\mathfrak{J}_{65}$}
\end{array}$ \\
\hline
$\mathfrak{J}_{05}\not\to \mathfrak{J}_{36}$ & $\begin{array}{@{}l@{}}
\mathcal{R}=\big\{A\mid {\rm dim}\,{\rm Ann}(A)\geq 1\big\}\\
\mathfrak{J}_{05}\in\mathcal{R},\\
\mbox{but } \mathfrak{J}_{36}\cap\mathcal{R}=\varnothing, \mbox{ as ${\rm dim}\,{\rm Ann}(\mathfrak{J}_{36})=0$}
\end{array}$ \\
\hline
$\mathfrak{J}_{06}\not\to \left\{\begin{array}{l}\mathfrak{J}_{30}, \mathfrak{J}_{32},\\ \mathfrak{J}_{33}, \mathfrak{J}_{36},\\ \mathfrak{J}_{37}, \mathfrak{J}_{44},\\ \mathfrak{J}_{51}, \mathfrak{J}_{52},\\ \mathfrak{J}_{55}, \mathfrak{J}_{61}\end{array}\right\}$ & $\begin{array}{@{}l@{}}
\mathcal{R}=\left\{A\ \middle|\ \begin{array}{l}A_{1}A_{3}\subseteq A_{3},\ c_{12}^{1}=c_{14}^{3}=c_{22}^{1}=c_{22}^{3}=c_{33}^{3}=c_{34}^{3}=0,\\
2c_{13}^{3}=c_{12}^{2},\ c_{11}^{1}c_{14}^{4}+c_{11}^{2}c_{24}^{4}+c_{11}^{3}c_{34}^{4}+c_{11}^{4}c_{44}^{4}=(c_{14}^{4})^{2},\\
c_{11}^{2}c_{12}^{3}=c_{11}^{3}c_{13}^{3},\ c_{13}^{3}c_{33}^{4}+c_{13}^{4}c_{34}^{4}=c_{14}^{4}c_{33}^{4},\\
c_{22}^{4}c_{44}^{4}+2c_{23}^{3}c_{24}^{4}=(c_{24}^{4})^{2},\ c_{23}^{3}c_{33}^{4}+c_{23}^{4}c_{34}^{4}=c_{24}^{4}c_{33}^{4}\end{array}\right\}\\
\mathfrak{J}_{06}\in\mathcal{R} \mbox{ (take $f_1=e_{2}$, $f_2=e_{1}$, $f_3=e_{4}$, $f_4=e_{3}$)},\\
\mbox{but } \mathfrak{B}\cap\mathcal{R}=\varnothing \mbox{ for all $\mathfrak{B}$ in the left column}
\end{array}$ \\
\hline
$\mathfrak{J}_{07}\not\to \left\{\begin{array}{l}\mathfrak{J}_{44}, \mathfrak{J}_{51},\\ \mathfrak{J}_{52}\end{array}\right\}$ & $\begin{array}{@{}l@{}}
\mathcal{R}=\big\{A\mid \big((xx)y\big)y+(xy)(xy)=2\big((xy)x\big)y \mbox{ for all $x,y\in A$}\big\}\\
\mathfrak{J}_{07}\in\mathcal{R},\\
\mbox{but } \mathfrak{B}\cap\mathcal{R}=\varnothing \mbox{ for all $\mathfrak{B}$ in the left column (take $x=e_1$, $y=e_2$)}
\end{array}$ \\
\hline
$\mathfrak{J}_{08}\not\to \left\{\begin{array}{l}\mathfrak{J}_{29}, \mathfrak{J}_{34},\\ \mathfrak{J}_{36}\end{array}\right\}$ & $\begin{array}{@{}l@{}}
\mathcal{R}=\left\{A\ \middle|\ \begin{array}{l}A_{1}A_{4}=0,\ c_{23}^{2}=0,\ 2c_{12}^{2}+2c_{13}^{3}=c_{11}^{1},\\
c_{22}^{2}+c_{23}^{3}=2c_{12}^{1},\ c_{33}^{3}=2c_{13}^{1}\end{array}\right\}\\
\mathfrak{J}_{08}\in\mathcal{R} \mbox{ (take $f_1=e_{3}$, $f_2=e_{1}$, $f_3=e_{2}$, $f_4=e_{4}$)},\\
\mbox{but } \mathfrak{B}\cap\mathcal{R}=\varnothing \mbox{ for all $\mathfrak{B}$ in the left column}
\end{array}$ \\
\hline
$\mathfrak{J}_{11}\not\to \mathfrak{J}_{64}$ & $\begin{array}{@{}l@{}}
\mathcal{R}=\left\{A\ \middle|\ \begin{array}{l}A_{1}A_{3}\subseteq A_{3},\ c_{14}^{3}=c_{24}^{3}=0,\\
c_{11}^{2}c_{22}^{4}+2c_{11}^{3}c_{23}^{4}+2c_{12}^{4}c_{13}^{3}=2c_{12}^{2}c_{12}^{4}+2c_{12}^{3}c_{13}^{4}\end{array}\right\}\\
\mathfrak{J}_{11}\in\mathcal{R} \mbox{ (take $f_i=e_i$)},\\
\mbox{but } \mathfrak{J}_{64}\cap\mathcal{R}=\varnothing
\end{array}$ \\
\hline
$\mathfrak{J}_{12}\not\to \left\{\begin{array}{l}\mathfrak{J}_{33}, \mathfrak{J}_{36},\\ \mathfrak{J}_{69}, \mathfrak{J}_{70}\end{array}\right\}$ & $\begin{array}{@{}l@{}}
\mathcal{R}=\left\{A\ \middle|\ \begin{array}{l}A_{1}A_{2}\subseteq A_{2},\ A_{1}A_{4}\subseteq A_{4},\ c_{22}^{3}=c_{22}^{4}=0,\\
c_{23}^{4}=c_{33}^{2}=c_{33}^{4}=c_{44}^{4}=0,\ c_{12}^{2}+c_{13}^{3}=3c_{14}^{4},\\
c_{11}^{4}c_{13}^{2}=c_{11}^{2}c_{13}^{4},\ c_{11}^{4}c_{12}^{3}=c_{11}^{3}c_{12}^{4}\end{array}\right\}\\
\mathfrak{J}_{12}\in\mathcal{R} \mbox{ (take $f_1=e_{2}$, $f_2=e_{3}$, $f_3=e_{1}$, $f_4=e_{4}$)},\\
\mbox{but } \mathfrak{B}\cap\mathcal{R}=\varnothing \mbox{ for all $\mathfrak{B}$ in the left column}
\end{array}$ \\
\hline
$\mathfrak{J}_{13}\not\to \left\{\begin{array}{l}\mathfrak{J}_{30}, \mathfrak{J}_{31},\\ \mathfrak{J}_{33}, \mathfrak{J}_{34},\\ \mathfrak{J}_{36}, \mathfrak{J}_{67},\\ \mathfrak{J}_{68}, \mathfrak{J}_{70}\end{array}\right\}$ & $\begin{array}{@{}l@{}}
\mathcal{R}=\left\{A\ \middle|\ \begin{array}{l}A_{1}A_{4}\subseteq A_{4},\ A_{3}^2=0,\ c_{12}^{2}=c_{13}^{2}=c_{22}^{1}=c_{22}^{3}=c_{22}^{4}=c_{23}^{4}=0,\\
2c_{13}^{3}+2c_{14}^{4}=c_{11}^{1},\ c_{22}^{2}=c_{12}^{1},\ 2c_{23}^{3}=c_{12}^{1},\ 2c_{24}^{4}=c_{12}^{1},\\
c_{11}^{2}c_{12}^{3}+c_{11}^{3}c_{13}^{3}=0,\ c_{11}^{2}c_{12}^{4}+c_{11}^{3}c_{13}^{4}+c_{11}^{4}c_{14}^{4}=0\end{array}\right\}\\
\mathfrak{J}_{13}\in\mathcal{R} \mbox{ (take $f_1=e_{1}$, $f_2=e_{1}+e_{2}$, $f_3=e_{3}$, $f_4=e_{4}$)},\\
\mbox{but } \mathfrak{B}\cap\mathcal{R}=\varnothing \mbox{ for all $\mathfrak{B}$ in the left column}
\end{array}$ \\
\hline
$\mathfrak{J}_{15}\not\to \left\{\begin{array}{l}\mathfrak{J}_{34}, \mathfrak{J}_{64}\end{array}\right\}$ & $\begin{array}{@{}l@{}}
\mathcal{R}=\left\{A\ \middle|\ \begin{array}{l}A_{1}A_{3}\subseteq A_{3},\ c_{14}^{3}=c_{24}^{3}=c_{33}^{3}=c_{44}^{4}=0,\ c_{13}^{3}=c_{11}^{1},\ 2c_{24}^{4}=c_{22}^{2},\\
c_{11}^{2}c_{22}^{4}+2c_{11}^{3}c_{23}^{4}+2c_{11}^{4}c_{24}^{4}+2c_{12}^{4}c_{13}^{3}=2c_{12}^{3}c_{13}^{4}+4c_{12}^{4}c_{14}^{4}\end{array}\right\}\\
\mathfrak{J}_{15}\in\mathcal{R} \mbox{ (take $f_i=e_i$)},\\
\mbox{but } \mathfrak{B}\cap\mathcal{R}=\varnothing \mbox{ for all $\mathfrak{B}$ in the left column}
\end{array}$ \\
\hline
$\mathfrak{J}_{16}\not\to \left\{\begin{array}{l}\mathfrak{J}_{32}, \mathfrak{J}_{33},\\ \mathfrak{J}_{34}, \mathfrak{J}_{36},\\ \mathfrak{J}_{67}, \mathfrak{J}_{68},\\ \mathfrak{J}_{70}\end{array}\right\}$ & $\begin{array}{@{}l@{}}
\mathcal{R}=\left\{A\ \middle|\ \begin{array}{l}A_{1}A_{4}\subseteq A_{4},\ A_{3}^2=0,\ c_{22}^{3}=c_{22}^{4}=c_{23}^{4}=0,\\
c_{12}^{1}=c_{13}^{2}=c_{22}^{1}=0,\ c_{11}^{1}+c_{12}^{2}=2c_{13}^{3},\\
2c_{14}^{4}=c_{12}^{2},\ c_{11}^{2}c_{12}^{4}+c_{11}^{3}c_{13}^{4}=c_{11}^{4}c_{14}^{4},\\
2c_{24}^{4}=c_{22}^{2},\ c_{11}^{2}c_{12}^{3}+c_{11}^{3}c_{13}^{3}=2c_{11}^{3}c_{14}^{4}\end{array}\right\}\\
\mathfrak{J}_{16}\in\mathcal{R} \mbox{ (take $f_1=e_{2}$, $f_2=e_{1}$, $f_3=e_{3}$, $f_4=e_{4}$)},\\
\mbox{but } \mathfrak{B}\cap\mathcal{R}=\varnothing \mbox{ for all $\mathfrak{B}$ in the left column}
\end{array}$ \\
\hline
$\mathfrak{J}_{18}\not\to \mathfrak{J}_{70}$ & $\begin{array}{@{}l@{}}
\mathcal{R}=\big\{A\mid A_{3}^2=0,\ c_{22}^{3}=c_{22}^{4}=c_{23}^{4}=c_{24}^{3}=0\big\}\\
\mathfrak{J}_{18}\in\mathcal{R} \mbox{ (take $f_i=e_i$)},\\
\mbox{but } \mathfrak{J}_{70}\cap\mathcal{R}=\varnothing
\end{array}$ \\
\hline
$\mathfrak{J}_{19}\not\to \mathfrak{J}_{70}$ & $\begin{array}{@{}l@{}}
\mathcal{R}=\big\{A\mid {\rm dim}\,{\rm Ann}(A)\geq 2\big\}\\
\mathfrak{J}_{19}\in\mathcal{R},\\
\mbox{but } \mathfrak{J}_{70}\cap\mathcal{R}=\varnothing, \mbox{ as ${\rm dim}\,{\rm Ann}(\mathfrak{J}_{70})=1$}
\end{array}$ \\
\hline
$\mathfrak{J}_{20}\not\to \mathfrak{J}_{19}$ & $\begin{array}{@{}l@{}}
\mathcal{R}=\big\{A\mid c_{22}^{2}=c_{44}^{4}=0,\ c_{12}^{2}=c_{11}^{1}\big\}\\
\mathfrak{J}_{20}\in\mathcal{R} \mbox{ (take $f_1=e_{1}$, $f_2=e_{3}$, $f_3=e_{2}$, $f_4=e_{4}$)},\\
\mbox{but } \mathfrak{J}_{19}\cap\mathcal{R}=\varnothing
\end{array}$ \\
\hline
$\mathfrak{J}_{21}\not\to \mathfrak{J}_{70}$ & $\begin{array}{@{}l@{}}
\mathcal{R}=\big\{A\mid A_{2}A_{3}=0,\ c_{13}^{4}=c_{14}^{3}=0\big\}\\
\mathfrak{J}_{21}\in\mathcal{R} \mbox{ (take $f_i=e_i$)},\\
\mbox{but } \mathfrak{J}_{70}\cap\mathcal{R}=\varnothing
\end{array}$ \\
\hline
$\mathfrak{J}_{22}\not\to \left\{\begin{array}{l}\mathfrak{J}_{34}, \mathfrak{J}_{37}\end{array}\right\}$ & $\begin{array}{@{}l@{}}
\mathcal{R}=\big\{A\mid c_{12}^{1}=c_{23}^{3}=c_{33}^{3}=c_{44}^{4}=0,\ c_{13}^{3}=c_{11}^{1},\ c_{14}^{4}=c_{12}^{2},\ c_{24}^{4}=c_{22}^{2}\big\}\\
\mathfrak{J}_{22}\in\mathcal{R} \mbox{ (take $f_i=e_i$)},\\
\mbox{but } \mathfrak{B}\cap\mathcal{R}=\varnothing \mbox{ for all $\mathfrak{B}$ in the left column}
\end{array}$ \\
\hline
$\mathfrak{J}_{23}\not\to \left\{\begin{array}{l}\mathfrak{J}_{36}, \mathfrak{J}_{37},\\ \mathfrak{J}_{51}, \mathfrak{J}_{65}\end{array}\right\}$ & $\begin{array}{@{}l@{}}
\mathcal{R}=\left\{A\ \middle|\ \begin{array}{l}A_{1}A_{2}\subseteq A_{2},\ A_{1}A_{4}=0,\ c_{22}^{3}=c_{33}^{2}=c_{33}^{3}=0,\\
2c_{13}^{3}=c_{12}^{2},\ c_{11}^{1}c_{12}^{4}+c_{11}^{2}c_{22}^{4}+c_{11}^{3}c_{23}^{4}=2c_{12}^{4}c_{13}^{3},\\
c_{11}^{1}c_{13}^{4}+c_{11}^{2}c_{23}^{4}+c_{11}^{3}c_{33}^{4}=2c_{13}^{3}c_{13}^{4}\end{array}\right\}\\
\mathfrak{J}_{23}\in\mathcal{R} \mbox{ (take $f_1=e_{2}$, $f_2=e_{1}$, $f_3=e_{3}$, $f_4=e_{4}$)},\\
\mbox{but } \mathfrak{B}\cap\mathcal{R}=\varnothing \mbox{ for all $\mathfrak{B}$ in the left column}
\end{array}$ \\
\hline
$\mathfrak{J}_{24}\not\to \left\{\begin{array}{l}\mathfrak{J}_{28}, \mathfrak{J}_{30},\\ \mathfrak{J}_{32}, \mathfrak{J}_{33},\\ \mathfrak{J}_{35}, \mathfrak{J}_{36},\\ \mathfrak{J}_{37}, \mathfrak{J}_{52},\\ \mathfrak{J}_{61}, \mathfrak{J}_{65}\end{array}\right\}$ & $\begin{array}{@{}l@{}}
\mathcal{R}=\left\{A\ \middle|\ \begin{array}{l}A_{1}A_{2}\subseteq A_{2},\ A_{1}A_{4}\subseteq A_{4},\ c_{22}^{3}=c_{33}^{2}=c_{33}^{3}=c_{44}^{4}=0,\\
2c_{13}^{3}=c_{12}^{2},\ c_{14}^{4}=c_{12}^{2},\ 2c_{23}^{3}=c_{22}^{2},\ c_{24}^{4}=c_{22}^{2},\\
c_{11}^{1}c_{12}^{4}+c_{11}^{2}c_{22}^{4}+c_{11}^{3}c_{23}^{4}+c_{11}^{4}c_{24}^{4}=2c_{12}^{3}c_{13}^{4}+2c_{12}^{4}c_{14}^{4},\\
c_{11}^{1}c_{13}^{4}+c_{11}^{2}c_{23}^{4}+c_{11}^{3}c_{33}^{4}=c_{13}^{4}c_{14}^{4}\end{array}\right\}\\
\mathfrak{J}_{24}\in\mathcal{R} \mbox{ (take $f_1=e_{2}$, $f_2=e_{1}$, $f_3=e_{3}$, $f_4=e_{4}$)},\\
\mbox{but } \mathfrak{B}\cap\mathcal{R}=\varnothing \mbox{ for all $\mathfrak{B}$ in the left column}
\end{array}$ \\
\hline
$\mathfrak{J}_{25}\not\to \left\{\begin{array}{l}\mathfrak{J}_{34}, \mathfrak{J}_{37},\\ \mathfrak{J}_{52}, \mathfrak{J}_{65}\end{array}\right\}$ & $\begin{array}{@{}l@{}}
\mathcal{R}=\left\{A\ \middle|\ \begin{array}{l}A_{1}A_{3}\subseteq A_{3},\ c_{44}^{4}=0,\ c_{11}^{1}+c_{12}^{2}=2c_{13}^{3},\\
c_{12}^{1}=c_{14}^{3}=c_{22}^{1}=0,\ 2c_{23}^{3}=c_{22}^{2},\\
c_{22}^{3}=c_{24}^{3}=c_{33}^{3}=0,\ c_{14}^{4}=c_{11}^{1},\\
c_{11}^{2}c_{22}^{4}+c_{11}^{3}c_{23}^{4}+c_{12}^{4}c_{14}^{4}=2c_{12}^{4}c_{13}^{3},\\
c_{11}^{2}c_{23}^{4}+c_{11}^{3}c_{33}^{4}+c_{13}^{4}c_{14}^{4}=2c_{13}^{3}c_{13}^{4}\end{array}\right\}\\
\mathfrak{J}_{25}\in\mathcal{R} \mbox{ (take $f_i=e_i$)},\\
\mbox{but } \mathfrak{B}\cap\mathcal{R}=\varnothing \mbox{ for all $\mathfrak{B}$ in the left column}
\end{array}$ \\
\hline
$\mathfrak{J}_{26}\not\to \mathfrak{J}_{37}$ & $\begin{array}{@{}l@{}}
\mathcal{R}=\big\{A\mid A_{1}A_{3}\subseteq A_{4},\ c_{14}^{4}=c_{24}^{4}=0\big\}\\
\mathfrak{J}_{26}\in\mathcal{R} \mbox{ (take $f_i=e_i$)},\\
\mbox{but } \mathfrak{J}_{37}\cap\mathcal{R}=\varnothing
\end{array}$ \\
\hline
$\mathfrak{J}_{27}\not\to \mathfrak{J}_{35}$ & $\begin{array}{@{}l@{}}
\mathcal{R}=\big\{A\mid A_{1}A_{4}\subseteq A_{4},\ c_{22}^{2}=c_{33}^{3}=0,\ c_{12}^{2}=c_{11}^{1},\ c_{13}^{3}=c_{11}^{1}\big\}\\
\mathfrak{J}_{27}\in\mathcal{R} \mbox{ (take $f_1=e_{1}$, $f_2=e_{3}$, $f_3=e_{4}$, $f_4=e_{2}$)},\\
\mbox{but } \mathfrak{J}_{35}\cap\mathcal{R}=\varnothing
\end{array}$ \\
\hline
$\mathfrak{J}_{29}\not\to \left\{\begin{array}{l}\mathfrak{J}_{68}, \mathfrak{J}_{70}\end{array}\right\}$ & $\begin{array}{@{}l@{}}
\mathcal{R}=\big\{A\mid A_{2}^2=0\big\}\\
\mathfrak{J}_{29}\in\mathcal{R} \mbox{ (take $f_i=e_i$)},\\
\mbox{but } \mathfrak{B}\cap\mathcal{R}=\varnothing \mbox{ for all $\mathfrak{B}$ in the left column}
\end{array}$ \\
\hline
$\mathfrak{J}_{35}\not\to \mathfrak{J}_{68}$ & $\begin{array}{@{}l@{}}
\mathcal{R}=\big\{A\mid A_{2}^2=0\big\}\\
\mathfrak{J}_{35}\in\mathcal{R} \mbox{ (take $f_i=e_i$)},\\
\mbox{but } \mathfrak{J}_{68}\cap\mathcal{R}=\varnothing
\end{array}$ \\
\hline
$\mathfrak{J}_{37}\not\to \mathfrak{J}_{68}$ & $\begin{array}{@{}l@{}}
\mathcal{R}=\big\{A\mid A_{2}^2=0\big\}\\
\mathfrak{J}_{37}\in\mathcal{R} \mbox{ (take $f_i=e_i$)},\\
\mbox{but } \mathfrak{J}_{68}\cap\mathcal{R}=\varnothing
\end{array}$ \\
\hline
$\mathfrak{J}_{44}\not\to \left\{\begin{array}{l}\mathfrak{J}_{67}, \mathfrak{J}_{70}\end{array}\right\}$ & $\begin{array}{@{}l@{}}
\mathcal{R}=\big\{A\mid A_{1}A_{3}=0,\ c_{11}^{2}=c_{22}^{1}=0\big\}\\
\mathfrak{J}_{44}\in\mathcal{R} \mbox{ (take $f_i=e_i$)},\\
\mbox{but } \mathfrak{B}\cap\mathcal{R}=\varnothing \mbox{ for all $\mathfrak{B}$ in the left column}
\end{array}$ \\
\hline
$\mathfrak{J}_{49}\not\to \mathfrak{J}_{36}$ & $\begin{array}{@{}l@{}}
\mathcal{R}=\big\{A\mid {\rm dim}\,{\rm Ann}(A)\geq 1\big\}\\
\mathfrak{J}_{49}\in\mathcal{R},\\
\mbox{but } \mathfrak{J}_{36}\cap\mathcal{R}=\varnothing, \mbox{ as ${\rm dim}\,{\rm Ann}(\mathfrak{J}_{36})=0$}
\end{array}$ \\
\hline
$\mathfrak{J}_{50}\not\to \left\{\begin{array}{l}\mathfrak{J}_{28}, \mathfrak{J}_{31}\end{array}\right\}$ & $\begin{array}{@{}l@{}}
\mathcal{R}=\big\{A\mid c_{22}^{2}=c_{33}^{3}=c_{44}^{4}=0,\ 2c_{13}^{3}=c_{11}^{1},\ 2c_{14}^{4}=c_{11}^{1}\big\}\\
\mathfrak{J}_{50}\in\mathcal{R} \mbox{ (take $f_i=e_i$)},\\
\mbox{but } \mathfrak{B}\cap\mathcal{R}=\varnothing \mbox{ for all $\mathfrak{B}$ in the left column}
\end{array}$ \\
\hline
$\mathfrak{J}_{53}\not\to \left\{\begin{array}{l}\mathfrak{J}_{28}, \mathfrak{J}_{34},\\ \mathfrak{J}_{36}, \mathfrak{J}_{37}\end{array}\right\}$ & $\begin{array}{@{}l@{}}
\mathcal{R}=\big\{A\mid c_{22}^{2}=c_{33}^{3}=c_{44}^{4}=0,\ 2c_{12}^{2}=c_{11}^{1},\ c_{13}^{3}=c_{11}^{1}\big\}\\
\mathfrak{J}_{53}\in\mathcal{R} \mbox{ (take $f_i=e_i$)},\\
\mbox{but } \mathfrak{B}\cap\mathcal{R}=\varnothing \mbox{ for all $\mathfrak{B}$ in the left column}
\end{array}$ \\
\hline
$\mathfrak{J}_{54}\not\to \mathfrak{J}_{34}$ & $\begin{array}{@{}l@{}}
\mathcal{R}=\big\{A\mid c_{22}^{2}=c_{33}^{3}=c_{44}^{4}=0,\ c_{13}^{3}=c_{11}^{1}\big\}\\
\mathfrak{J}_{54}\in\mathcal{R} \mbox{ (take $f_i=e_i$)},\\
\mbox{but } \mathfrak{J}_{34}\cap\mathcal{R}=\varnothing
\end{array}$ \\
\hline
$\mathfrak{J}_{55}\not\to \left\{\begin{array}{l}\mathfrak{J}_{67}, \mathfrak{J}_{70}\end{array}\right\}$ & $\begin{array}{@{}l@{}}
\mathcal{R}=\big\{A\mid A_{2}A_{3}=0,\ c_{11}^{2}=c_{13}^{2}=c_{13}^{4}=c_{14}^{2}=c_{14}^{3}=c_{22}^{1}=0\big\}\\
\mathfrak{J}_{55}\in\mathcal{R} \mbox{ (take $f_1=e_{1}$, $f_2=e_{4}$, $f_3=e_{2}$, $f_4=e_{3}$)},\\
\mbox{but } \mathfrak{B}\cap\mathcal{R}=\varnothing \mbox{ for all $\mathfrak{B}$ in the left column}
\end{array}$ \\
\hline
$\mathfrak{J}_{57}\not\to \left\{\begin{array}{l}\mathfrak{J}_{29}, \mathfrak{J}_{34}\end{array}\right\}$ & $\begin{array}{@{}l@{}}
\mathcal{R}=\big\{A\mid c_{22}^{2}=c_{33}^{3}=c_{44}^{4}=0,\ c_{12}^{2}=c_{11}^{1},\ c_{14}^{4}=c_{11}^{1}\big\}\\
\mathfrak{J}_{57}\in\mathcal{R} \mbox{ (take $f_1=e_{1}$, $f_2=e_{2}$, $f_3=e_{4}$, $f_4=e_{3}$)},\\
\mbox{but } \mathfrak{B}\cap\mathcal{R}=\varnothing \mbox{ for all $\mathfrak{B}$ in the left column}
\end{array}$ \\
\hline
$\mathfrak{J}_{59}\not\to \left\{\begin{array}{l}\mathfrak{J}_{30}, \mathfrak{J}_{33},\\ \mathfrak{J}_{34}, \mathfrak{J}_{36},\\ \mathfrak{J}_{69}\end{array}\right\}$ & $\begin{array}{@{}l@{}}
\mathcal{R}=\left\{A\ \middle|\ \begin{array}{l}A_{2}^2\subseteq A_{4},\ c_{11}^{2}=c_{11}^{3}=c_{12}^{3}=c_{13}^{2}=0,\\
c_{14}^{2}=c_{14}^{3}=c_{44}^{4}=0,\ 2c_{13}^{3}=c_{11}^{1},\ c_{14}^{4}=c_{11}^{1}\end{array}\right\}\\
\mathfrak{J}_{59}\in\mathcal{R} \mbox{ (take $f_1=e_{1}$, $f_2=e_{3}$, $f_3=e_{4}$, $f_4=e_{2}$)},\\
\mbox{but } \mathfrak{B}\cap\mathcal{R}=\varnothing \mbox{ for all $\mathfrak{B}$ in the left column}
\end{array}$ \\
\hline
$\mathfrak{J}_{60}\not\to \mathfrak{J}_{28}$ & $\begin{array}{@{}l@{}}
\mathcal{R}=\big\{A\mid c_{22}^{2}=c_{33}^{3}=c_{44}^{4}=0,\ 2c_{13}^{3}=c_{11}^{1},\ 2c_{14}^{4}=c_{11}^{1}\big\}\\
\mathfrak{J}_{60}\in\mathcal{R} \mbox{ (take $f_i=e_i$)},\\
\mbox{but } \mathfrak{J}_{28}\cap\mathcal{R}=\varnothing
\end{array}$ \\
\hline
\end{longtable}
\endgroup

\noindent This completes the proof of Theorem~A.
\end{proof}

\begin{remark}\label{rem_papers}
We have checked the results on degenerations of $3$- and $4$-dimensional Jordan algebras obtained in {\rm \cite{contr11,KS07,gkp,KE14}}. The following results can not be used in our work, because they are wrong (or, for $3$-dimensional algebras, they do not imply the corresponding results in dimension $4$).
\begin{enumerate}
\item {\it $4$-dimensional nilpotent Jordan algebras} {\rm \cite{contr11}}. For every $k$ with $1\le k\le 13$, the algebra $\varphi_k$ of {\rm \cite{contr11}} is $\mathfrak{J}_{60+k}$ (with the same multiplication table). The following statements of {\rm \cite{contr11}} are false.
\begin{enumerate}
\item The non-degenerations $\varphi_1\not\to\varphi_6$ and $\varphi_1\not\to\varphi_8$, i.e. $\mathfrak{J}_{61}\not\to\mathfrak{J}_{66}$ and $\mathfrak{J}_{61}\not\to\mathfrak{J}_{68}$ (proof of Theorem~3.6 of {\rm \cite{contr11}}). In fact, $\mathfrak{J}_{61}\to\mathfrak{J}_{66}$ with the parametrized basis ${\rm E}^t_1=e_1$, ${\rm E}^t_2=e_2$, ${\rm E}^t_3=e_3$, ${\rm E}^t_4=t^{-1}e_2-t^{-3}e_3+t^{-5}e_4$ (Table~2), and $\mathfrak{J}_{66}\to\mathfrak{J}_{68}$ ($\varphi_6\to\varphi_8$ is proved in {\rm \cite{contr11}} itself).
\item The contraction $\varphi_5\to\varphi_4$, i.e. $\mathfrak{J}_{65}\to\mathfrak{J}_{64}$ (Figure~1 of {\rm \cite{contr11}}). In fact, $\mathfrak{J}_{11}\to\mathfrak{J}_{65}$ (Table~2) and $\mathfrak{J}_{11}\not\to\mathfrak{J}_{64}$ (by Theorem~A), hence $\mathfrak{J}_{65}\not\to\mathfrak{J}_{64}$.
\item The contraction $\varphi_4\to\varphi_{10}$, i.e. $\mathfrak{J}_{64}\to\mathfrak{J}_{70}$. In fact, $\mathfrak{J}_{29}\to\mathfrak{J}_{64}$ (Table~2) and $\mathfrak{J}_{29}\not\to\mathfrak{J}_{70}$ (by Theorem~A), hence $\mathfrak{J}_{64}\not\to\mathfrak{J}_{70}$; the structure constants of the family given in {\rm \cite{contr11}} are not bounded as $t\to0$.
\item The family given for $\varphi_{10}\to\varphi_{11}$ has the limit $E_1E_1=E_1E_3=E_3E_3=E_2$, which is isomorphic to $\varphi_{12}=\mathfrak{J}_{72}$ and not to $\varphi_{11}$ (the degeneration $\mathfrak{J}_{70}\to\mathfrak{J}_{71}$ itself is true, see Table~2).
\end{enumerate}
Consequently, the description of degenerations of $4$-dimensional nilpotent Jordan algebras (Figure~1 and Corollary~3.7 of {\rm \cite{contr11}}) is not correct; the statement that the variety of $4$-dimensional nilpotent Jordan algebras has two irreducible components (Theorem~3.6 of {\rm \cite{contr11}}) is true.
\item {\it $3$-dimensional Jordan algebras} {\rm \cite{KS07,gkp}}. The degenerations and the irreducible components of the variety of $3$-dimensional Jordan algebras found in {\rm \cite{KS07,gkp}} are correct, but some details are wrong: in {\rm \cite{gkp}} the parametrized basis for $\mathbb{T}_{03}\to\mathbb{T}_{15}$ does not give this degeneration, since the structure constant $c_{22}^3(t)=-t^{-1}$ is unbounded, and the basis for $\mathbb{T}_{16}\to\mathbb{T}_{19}$ gives $E_1^2=-E_2$; in {\rm \cite{KS07}} the dimensions of the automorphism groups of $\mathcal{J}_7$, $\mathcal{A}_7$, $\mathcal{A}_8$ are given as $1$, $6$, $8$ (the correct values are $2$, $5$, $9$), and in the degeneration $\mathcal{J}_1\to\mathcal{J}_6$ the limit algebra is written as $\mathcal{J}_4$.
More importantly, a degeneration $\mathfrak{A}\to\mathfrak{B}$ of $3$-dimensional algebras gives the degeneration $\mathfrak{A}\oplus\mathfrak{D}\to\mathfrak{B}\oplus\mathfrak{D}$, but a non-degeneration $\mathfrak{A}\not\to\mathfrak{B}$ does {\bf not} imply $\mathfrak{A}\oplus\mathfrak{D}\not\to\mathfrak{B}\oplus\mathfrak{D}$. For example, $\mathbb{T}_{10}\not\to\mathbb{T}_{17}$ {\rm \cite{gkp}}, but $\mathfrak{J}_{51}=\mathbb{T}_{10}\oplus\mathbb{C}n_3\to\mathfrak{J}_{29}\to\mathfrak{J}_{64}\to\mathfrak{J}_{67}=\mathbb{T}_{17}\oplus\mathbb{C}n_4$; there are $10$ such pairs of $3$-dimensional algebras. Hence the non-degenerations of {\rm \cite{KS07,gkp}} can not be used for $4$-dimensional algebras.
\item {\it $4$-dimensional Jordan algebras} {\rm \cite{KE14}}. All $63$ degenerations stated in {\rm \cite{KE14}} are true, all explicit parametrized bases given there are correct, and the list of rigid algebras (Theorem~4.1 of {\rm \cite{KE14}}) agrees with Corollary~A. However, {\rm \cite{KE14}} gives only a part of the degenerations (a family of degenerations which is sufficient to find the irreducible components) and almost no non-degenerations; the degenerations of nilpotent algebras are quoted there from {\rm \cite{contr11}}, whose description is not correct (see above); and the degenerations $\mathfrak{J}_{38}\to\mathfrak{J}_{66}$ and $\mathfrak{J}_{61}\to\mathfrak{J}_{67}$ are not primary. So the results of {\rm \cite{KE14}} do not give the graph of degenerations.
\end{enumerate}
For these reasons we give the complete description of degenerations and non-degenerations of $4$-dimensional Jordan algebras ``from scratch'': all primary degenerations (Table~2) and all primary non-degenerations (Table~3) are proved directly, and the results of {\rm \cite{contr11,KS07,gkp,KE14}} are not used in the proofs.
\end{remark}

\clearpage
\newgeometry{layoutwidth=297mm,layoutheight=210mm,left=1in,right=1in,top=1in,bottom=1in}
\pdfpagewidth=297mm \pdfpageheight=210mm \setlength{\headwidth}{\textwidth}
{\scriptsize
\begin{center} 
\resizebox{!}{0.86\textheight}{%
\begin{tikzpicture}[-, >=stealth', shorten >=0.005cm, auto,node distance=0.9cm, thick,
                    main  node/.style={rectangle,draw,fill=gray!12,rounded corners=1.5ex,font=\sffamily \tiny \bfseries },
                    rigid node/.style={rectangle,draw,fill=black!20,rounded corners=1.5ex,font=\sffamily \bf \bfseries }, 
                    lev0/.style={draw=red!85!black},
                    lev1/.style={draw=blue!85!black},
                    lev2/.style={draw=green!55!black},
                    lev3/.style={draw=orange!95!black},
                    lev4/.style={draw=violet},
                    lev5/.style={draw=cyan!65!black},
                    lev6/.style={draw=brown},
                    lev7/.style={draw=magenta!85!black},
                    lev8/.style={draw=olive!90!black},
                    lev9/.style={draw=black},
                    lev10/.style={draw=gray!65!black},
                    lev12/.style={draw=pink!85!black},
                    lev16/.style={draw=black},
                    style={draw,font=\sffamily \scriptsize \bfseries }]
\node [text=red!85!black] (00) at (0,0.00) {0};
\node [text=blue!85!black] (01) at (0,-0.90) {1};
\node [text=green!55!black] (02) at (0,-1.80) {2};
\node [text=orange!95!black] (03) at (0,-3.60) {3};
\node [text=violet] (04) at (0,-5.40) {4};
\node [text=cyan!65!black] (05) at (0,-7.20) {5};
\node [text=brown] (06) at (0,-9.00) {6};
\node [text=magenta!85!black] (07) at (0,-10.17) {7};
\node [text=olive!90!black] (08) at (0,-11.34) {8};
\node [text=black] (09) at (0,-12.24) {9};
\node [text=gray!65!black] (010) at (0,-13.14) {10};
\node [text=pink!85!black] (012) at (0,-14.04) {12};
\node [text=black] (016) at (0,-14.94) {16};
\node  [rigid node] (j01) at (13.61,-0.90) {${01}$};
\node  [rigid node] (j02) at (21.43,-3.60) {${02}$};
\node  [rigid node] (j03) at (12.30,0.00) {${03}$};
\node  [main node] (j04) at (11.00,-0.90) {${04}$};
\node  [main node] (j05) at (12.30,-0.90) {${05}$};
\node  [rigid node] (j06) at (13.61,-1.80) {${06}$};
\node  [main node] (j07) at (11.00,-1.80) {${07}$};
\node  [main node] (j08) at (12.30,-1.80) {${08}$};
\node  [main node] (j09) at (20.13,-5.40) {${09}$};
\node  [main node] (j10) at (12.30,-3.60) {${10}$};
\node  [main node] (j11) at (8.38,-3.60) {${11}$};
\node  [rigid node] (j12) at (14.26,-9.00) {${12}$};
\node  [rigid node] (j13) at (18.82,-5.40) {${13}$};
\node  [main node] (j14) at (16.22,-3.60) {${14}$};
\node  [main node] (j15) at (13.61,-3.60) {${15}$};
\node  [rigid node] (j16) at (16.22,-5.40) {${16}$};
\node  [main node] (j17) at (11.00,-3.60) {${17}$};
\node  [main node] (j18) at (15.56,-9.00) {${18}$};
\node  [main node] (j19) at (5.78,-5.40) {${19}$};
\node  [main node] (j20) at (8.39,-1.80) {${20}$};
\node  [main node] (j21) at (1.86,-5.40) {${21}$};
\node  [main node] (j22) at (7.08,-1.80) {${22}$};
\node  [main node] (j23) at (17.52,-1.80) {${23}$};
\node  [rigid node] (j24) at (18.82,-1.80) {${24}$};
\node  [main node] (j25) at (14.91,-1.80) {${25}$};
\node  [main node] (j26) at (9.69,-1.80) {${26}$};
\node  [main node] (j27) at (5.78,-1.80) {${27}$};
\node  [main node] (j28) at (11.65,-9.00) {${28}$};
\node  [main node] (j29) at (12.30,-5.40) {${29}$};
\node  [main node] (j30) at (12.95,-10.17) {${30}$};
\node  [main node] (j31) at (12.95,-9.00) {${31}$};
\node  [main node] (j32) at (14.26,-10.17) {${32}$};
\node  [rigid node] (j33) at (12.30,-14.04) {${33}$};
\node  [main node] (j34) at (11.65,-12.24) {${34}$};
\node  [main node] (j35) at (11.65,-7.20) {${35}$};
\node  [main node] (j36) at (12.95,-12.24) {${36}$};
\node  [main node] (j37) at (9.04,-7.20) {${37}$};
\node  [main node] (j38) at (5.78,-3.60) {${38}$};
\node  [main node] (j39) at (3.17,-3.60) {${39}$};
\node  [main node] (j40) at (7.73,-7.20) {${40}$};
\node  [main node] (j41) at (7.08,-5.40) {${41}$};
\node  [main node] (j42) at (6.43,-7.20) {${42}$};
\node  [main node] (j43) at (4.47,-5.40) {${43}$};
\node  [main node] (j44) at (14.91,-5.40) {${44}$};
\node  [main node] (j45) at (14.91,-3.60) {${45}$};
\node  [main node] (j46) at (9.69,-5.40) {${46}$};
\node  [main node] (j47) at (7.08,-3.60) {${47}$};
\node  [main node] (j48) at (16.87,-7.20) {${48}$};
\node  [main node] (j49) at (21.43,-5.40) {${49}$};
\node  [main node] (j50) at (14.26,-7.20) {${50}$};
\node  [main node] (j51) at (17.52,-3.60) {${51}$};
\node  [main node] (j52) at (18.82,-3.60) {${52}$};
\node  [main node] (j53) at (16.22,-1.80) {${53}$};
\node  [main node] (j54) at (4.47,-3.60) {${54}$};
\node  [main node] (j55) at (17.52,-5.40) {${55}$};
\node  [main node] (j56) at (13.61,-5.40) {${56}$};
\node  [main node] (j57) at (20.13,-3.60) {${57}$};
\node  [main node] (j58) at (18.17,-7.20) {${58}$};
\node  [rigid node] (j59) at (22.74,-5.40) {${59}$};
\node  [main node] (j60) at (15.56,-7.20) {${60}$};
\node  [main node] (j61) at (3.17,-5.40) {${61}$};
\node  [main node] (j62) at (9.69,-3.60) {${62}$};
\node  [main node] (j63) at (11.00,-5.40) {${63}$};
\node  [main node] (j64) at (12.95,-7.20) {${64}$};
\node  [main node] (j65) at (8.39,-5.40) {${65}$};
\node  [main node] (j66) at (10.34,-7.20) {${66}$};
\node  [main node] (j67) at (9.04,-9.00) {${67}$};
\node  [main node] (j68) at (10.34,-9.00) {${68}$};
\node  [main node] (j69) at (10.34,-10.17) {${69}$};
\node  [main node] (j70) at (11.65,-10.17) {${70}$};
\node  [main node] (j71) at (12.30,-11.34) {${71}$};
\node  [main node] (j72) at (12.30,-13.14) {${72}$};
\node  [main node] (j73) at (12.30,-14.94) {$\mathbb{C}^4$};

  \path[every node/.style={font=\sffamily\small}]

(j01) edge   [lev1] node[above] {} (j07)
(j01) edge   [lev1] node[above] {} (j08)
(j01) edge   [lev1] node[above] {} (j23)
(j01) edge   [lev1] node[above] {} (j25)
(j01) edge   [lev1] node[above] {} (j53)

(j02) edge   [lev3] node[above] {} (j09)
(j02) edge   [lev3] node[above] {} (j49)

(j03) edge   [lev0] node[above] {} (j04)
(j03) edge   [lev0] node[above] {} (j05)

(j04) edge   [lev1] node[above] {} (j20)
(j04) edge   [lev1] node[above] {} (j22)
(j04) edge   [lev1] node[above] {} (j26)
(j04) edge   [bend right=6, lev1] node[above] {} (j27)

(j05) edge   [lev1] node[above] {} (j20)
(j05) edge   [lev1] node[above] {} (j26)

(j06) edge   [lev2] node[above] {} (j10)
(j06) edge   [lev2] node[above] {} (j14)
(j06) edge   [lev2] node[above] {} (j15)
(j06) edge   [bend left=6, lev2] node[above] {} (j19)
(j06) edge   [lev2] node[above] {} (j62)

(j07) edge   [lev2] node[above] {} (j10)
(j07) edge   [lev2] node[above] {} (j11)
(j07) edge   [lev2] node[above] {} (j17)
(j07) edge   [bend left=10, lev2] node[above] {} (j21)
(j07) edge   [lev2] node[above] {} (j62)

(j08) edge   [lev2] node[above] {} (j11)
(j08) edge   [lev2] node[above] {} (j45)
(j08) edge   [lev2] node[above] {} (j62)

(j09) edge   [bend right=6, lev4] node[above] {} (j18)
(j09) edge   [lev4] node[above] {} (j48)

(j10) edge   [lev3] node[above] {} (j29)
(j10) edge   [lev3] node[above] {} (j41)
(j10) edge   [lev3] node[above] {} (j46)
(j10) edge   [lev3] node[above] {} (j63)

(j11) edge   [bend left=6, lev3] node[above] {} (j37)
(j11) edge   [lev3] node[above] {} (j46)
(j11) edge   [lev3] node[above] {} (j65)

(j12) edge   [lev6] node[above] {} (j30)
(j12) edge   [lev6] node[above] {} (j32)
(j12) edge   [bend left=15, lev6] node[above] {} (j34)

(j13) edge   [bend left=10, lev4] node[above] {} (j28)
(j13) edge   [lev4] node[above] {} (j60)

(j14) edge   [lev3] node[above] {} (j29)
(j14) edge   [lev3] node[above] {} (j41)
(j14) edge   [lev3] node[above] {} (j56)
(j14) edge   [lev3] node[above] {} (j63)

(j15) edge   [bend left=6, lev3] node[above] {} (j35)
(j15) edge   [lev3] node[above] {} (j46)
(j15) edge   [lev3] node[above] {} (j56)
(j15) edge   [lev3] node[above] {} (j65)

(j16) edge   [bend right=6, lev4] node[above] {} (j28)
(j16) edge   [bend left=6, lev4] node[above] {} (j31)
(j16) edge   [lev4] node[above] {} (j50)

(j17) edge   [lev3] node[above] {} (j29)
(j17) edge   [lev3] node[above] {} (j43)
(j17) edge   [lev3] node[above] {} (j46)
(j17) edge   [lev3] node[above] {} (j63)

(j18) edge   [lev6] node[above] {} (j30)
(j18) edge   [bend left=6, lev6] node[above] {} (j36)

(j19) edge   [lev4] node[above] {} (j35)
(j19) edge   [lev4] node[above] {} (j40)

(j20) edge   [lev2] node[above] {} (j38)
(j20) edge   [lev2] node[above] {} (j47)
(j20) edge   [lev2] node[above] {} (j54)

(j21) edge   [lev4] node[above] {} (j37)
(j21) edge   [lev4] node[above] {} (j40)
(j21) edge   [lev4] node[above] {} (j42)

(j22) edge   [lev2] node[above] {} (j39)
(j22) edge   [lev2] node[above] {} (j47)

(j23) edge   [lev2] node[above] {} (j10)
(j23) edge   [lev2] node[above] {} (j45)
(j23) edge   [lev2] node[above] {} (j52)

(j24) edge   [lev2] node[above] {} (j14)
(j24) edge   [lev2] node[above] {} (j51)
(j24) edge   [lev2] node[above] {} (j57)

(j25) edge   [lev2] node[above] {} (j17)
(j25) edge   [lev2] node[above] {} (j45)
(j25) edge   [lev2] node[above] {} (j51)

(j26) edge   [lev2] node[above] {} (j19)
(j26) edge   [lev2] node[above] {} (j38)
(j26) edge   [lev2] node[above] {} (j47)

(j27) edge   [lev2] node[above] {} (j21)
(j27) edge   [lev2] node[above] {} (j38)
(j27) edge   [lev2] node[above] {} (j39)
(j27) edge   [lev2] node[above] {} (j54)

(j28) edge   [lev6] node[above] {} (j69)

(j29) edge   [lev4] node[above] {} (j64)

(j30) edge   [lev7] node[above] {} (j71)

(j31) edge   [lev6] node[above] {} (j69)

(j32) edge   [lev7] node[above] {} (j71)

(j33) edge   [lev12] node[above] {} (j73)

(j34) edge   [lev9] node[above] {} (j72)

(j35) edge   [lev5] node[above] {} (j67)

(j36) edge   [lev9] node[above] {} (j72)

(j37) edge   [lev5] node[above] {} (j67)

(j38) edge   [lev3] node[above] {} (j41)
(j38) edge   [lev3] node[above] {} (j61)

(j39) edge   [lev3] node[above] {} (j43)
(j39) edge   [lev3] node[above] {} (j61)

(j40) edge   [bend right=15, lev5] node[above] {} (j34)
(j40) edge   [lev5] node[above] {} (j67)
(j40) edge   [lev5] node[above] {} (j68)

(j41) edge   [lev4] node[above] {} (j40)
(j41) edge   [lev4] node[above] {} (j66)

(j42) edge   [bend right=6, lev5] node[above] {} (j36)
(j42) edge   [lev5] node[above] {} (j67)
(j42) edge   [lev5] node[above] {} (j68)

(j43) edge   [lev4] node[above] {} (j42)
(j43) edge   [lev4] node[above] {} (j66)

(j44) edge   [bend left=6, lev4] node[above] {} (j28)
(j44) edge   [bend right=6, lev4] node[above] {} (j68)

(j45) edge   [lev3] node[above] {} (j44)
(j45) edge   [lev3] node[above] {} (j46)
(j45) edge   [lev3] node[above] {} (j63)

(j46) edge   [bend left=6, lev4] node[above] {} (j28)
(j46) edge   [lev4] node[above] {} (j66)

(j47) edge   [bend right=6, lev3] node[above] {} (j35)
(j47) edge   [lev3] node[above] {} (j61)

(j48) edge   [bend left=6, lev5] node[above] {} (j30)
(j48) edge   [lev5] node[above] {} (j70)

(j49) edge   [lev4] node[above] {} (j48)

(j50) edge   [lev5] node[above] {} (j30)
(j50) edge   [bend left=6, lev5] node[above] {} (j69)

(j51) edge   [lev3] node[above] {} (j29)
(j51) edge   [lev3] node[above] {} (j63)

(j52) edge   [lev3] node[above] {} (j29)
(j52) edge   [lev3] node[above] {} (j63)

(j53) edge   [lev2] node[above] {} (j51)
(j53) edge   [lev2] node[above] {} (j52)
(j53) edge   [lev2] node[above] {} (j62)

(j54) edge   [bend right=6, lev3] node[above] {} (j37)
(j54) edge   [lev3] node[above] {} (j61)

(j55) edge   [bend right=6, lev4] node[above] {} (j31)
(j55) edge   [bend left=10, lev4] node[above] {} (j68)

(j56) edge   [bend left=6, lev4] node[above] {} (j31)
(j56) edge   [lev4] node[above] {} (j66)

(j57) edge   [lev3] node[above] {} (j55)
(j57) edge   [lev3] node[above] {} (j56)
(j57) edge   [lev3] node[above] {} (j63)

(j58) edge   [bend left=6, lev5] node[above] {} (j32)
(j58) edge   [bend right=10, lev5] node[above] {} (j70)

(j59) edge   [lev4] node[above] {} (j58)

(j60) edge   [lev5] node[above] {} (j32)
(j60) edge   [lev5] node[above] {} (j69)

(j61) edge   [lev4] node[above] {} (j66)

(j62) edge   [lev3] node[above] {} (j63)
(j62) edge   [lev3] node[above] {} (j65)

(j63) edge   [lev4] node[above] {} (j64)
(j63) edge   [lev4] node[above] {} (j66)

(j64) edge   [lev5] node[above] {} (j67)

(j65) edge   [lev4] node[above] {} (j66)

(j66) edge   [lev5] node[above] {} (j67)
(j66) edge   [lev5] node[above] {} (j68)
(j66) edge   [lev5] node[above] {} (j70)

(j67) edge   [lev6] node[above] {} (j69)

(j68) edge   [lev6] node[above] {} (j69)

(j69) edge   [lev7] node[above] {} (j71)

(j70) edge   [lev7] node[above] {} (j71)

(j71) edge   [lev8] node[above] {} (j72)

(j72) edge   [bend right=38, lev10] node[above] {} (j73);

\node [anchor=west, align=left, font=\sffamily\scriptsize] at (24.04,0.60) {arrows from\\ level};
\draw [red!85!black, line width=1.2pt] (24.04,0.00) -- (24.64,0.00) node [right, text=red!85!black] {0};
\draw [blue!85!black, line width=1.2pt] (24.04,-0.90) -- (24.64,-0.90) node [right, text=blue!85!black] {1};
\draw [green!55!black, line width=1.2pt] (24.04,-1.80) -- (24.64,-1.80) node [right, text=green!55!black] {2};
\draw [orange!95!black, line width=1.2pt] (24.04,-3.60) -- (24.64,-3.60) node [right, text=orange!95!black] {3};
\draw [violet, line width=1.2pt] (24.04,-5.40) -- (24.64,-5.40) node [right, text=violet] {4};
\draw [cyan!65!black, line width=1.2pt] (24.04,-7.20) -- (24.64,-7.20) node [right, text=cyan!65!black] {5};
\draw [brown, line width=1.2pt] (24.04,-9.00) -- (24.64,-9.00) node [right, text=brown] {6};
\draw [magenta!85!black, line width=1.2pt] (24.04,-10.17) -- (24.64,-10.17) node [right, text=magenta!85!black] {7};
\draw [olive!90!black, line width=1.2pt] (24.04,-11.34) -- (24.64,-11.34) node [right, text=olive!90!black] {8};
\draw [black, line width=1.2pt] (24.04,-12.24) -- (24.64,-12.24) node [right, text=black] {9};
\draw [gray!65!black, line width=1.2pt] (24.04,-13.14) -- (24.64,-13.14) node [right, text=gray!65!black] {10};
\draw [pink!85!black, line width=1.2pt] (24.04,-14.04) -- (24.64,-14.04) node [right, text=pink!85!black] {12};
\end{tikzpicture}
}

\medskip
\normalsize Graph of primary degenerations of $4$-dimensional Jordan algebras (the numbers on the left are the dimensions of the algebras of derivations; the rigid algebras are shown in dark boxes; the colour of an arrow is the colour of the level of its initial vertex).
\end{center}
}
\clearpage
\restoregeometry
\pdfpagewidth=210mm \pdfpageheight=297mm \setlength{\headwidth}{\textwidth}

\begin{corollaryA}\label{ir_LZ} The irreducible components  of
$\mathfrak{Jord}_4$ are
$$
\begin{aligned}
\mathcal{C}_{01}   &=\overline{ {\mathcal O}(\mathfrak{J}_{01})  }= 
\left\{\begin{array}{l}
\mathfrak{J}_{01}, \mathfrak{J}_{07}, \mathfrak{J}_{08}, \mathfrak{J}_{10}, \mathfrak{J}_{11}, \mathfrak{J}_{17}, \mathfrak{J}_{21}, \mathfrak{J}_{23}, \mathfrak{J}_{25}, \mathfrak{J}_{28}, \mathfrak{J}_{29}, \mathfrak{J}_{34}, \mathfrak{J}_{36}, \mathfrak{J}_{37}, \\ \mathfrak{J}_{40}, \mathfrak{J}_{41}, \mathfrak{J}_{42}, \mathfrak{J}_{43}, \mathfrak{J}_{44}, \mathfrak{J}_{45}, \mathfrak{J}_{46}, \mathfrak{J}_{51}, \mathfrak{J}_{52}, \mathfrak{J}_{53}, \mathfrak{J}_{62}, \mathfrak{J}_{63}, \mathfrak{J}_{64}, \mathfrak{J}_{65}, \\ \mathfrak{J}_{66}, \mathfrak{J}_{67}, \mathfrak{J}_{68}, \mathfrak{J}_{69}, \mathfrak{J}_{70}, \mathfrak{J}_{71}, \mathfrak{J}_{72}, \mathbb{C}^4
\end{array} \right\};\\
\mathcal{C}_{02}   &=\overline{ {\mathcal O}(\mathfrak{J}_{02})  }= 
\left\{\begin{array}{l}
\mathfrak{J}_{02}, \mathfrak{J}_{09}, \mathfrak{J}_{18}, \mathfrak{J}_{30}, \mathfrak{J}_{36}, \mathfrak{J}_{48}, \mathfrak{J}_{49}, \mathfrak{J}_{70}, \mathfrak{J}_{71}, \mathfrak{J}_{72}, \mathbb{C}^4
\end{array} \right\};\\
\mathcal{C}_{03}   &=\overline{ {\mathcal O}(\mathfrak{J}_{03})  }= 
\left\{\begin{array}{l}
\mathfrak{J}_{03}, \mathfrak{J}_{04}, \mathfrak{J}_{05}, \mathfrak{J}_{19}, \mathfrak{J}_{20}, \mathfrak{J}_{21}, \mathfrak{J}_{22}, \mathfrak{J}_{26}, \mathfrak{J}_{27}, \mathfrak{J}_{34}, \mathfrak{J}_{35}, \mathfrak{J}_{36}, \mathfrak{J}_{37}, \mathfrak{J}_{38}, \\ \mathfrak{J}_{39}, \mathfrak{J}_{40}, \mathfrak{J}_{41}, \mathfrak{J}_{42}, \mathfrak{J}_{43}, \mathfrak{J}_{47}, \mathfrak{J}_{54}, \mathfrak{J}_{61}, \mathfrak{J}_{66}, \mathfrak{J}_{67}, \mathfrak{J}_{68}, \mathfrak{J}_{69}, \mathfrak{J}_{70}, \mathfrak{J}_{71}, \\ \mathfrak{J}_{72}, \mathbb{C}^4
\end{array} \right\};\\
\mathcal{C}_{04}   &=\overline{ {\mathcal O}(\mathfrak{J}_{06})  }= 
\left\{\begin{array}{l}
\mathfrak{J}_{06}, \mathfrak{J}_{10}, \mathfrak{J}_{14}, \mathfrak{J}_{15}, \mathfrak{J}_{19}, \mathfrak{J}_{28}, \mathfrak{J}_{29}, \mathfrak{J}_{31}, \mathfrak{J}_{34}, \mathfrak{J}_{35}, \mathfrak{J}_{40}, \mathfrak{J}_{41}, \mathfrak{J}_{46}, \mathfrak{J}_{56}, \\ \mathfrak{J}_{62}, \mathfrak{J}_{63}, \mathfrak{J}_{64}, \mathfrak{J}_{65}, \mathfrak{J}_{66}, \mathfrak{J}_{67}, \mathfrak{J}_{68}, \mathfrak{J}_{69}, \mathfrak{J}_{70}, \mathfrak{J}_{71}, \mathfrak{J}_{72}, \mathbb{C}^4
\end{array} \right\};\\
\mathcal{C}_{05}   &=\overline{ {\mathcal O}(\mathfrak{J}_{12})  }= 
\left\{\begin{array}{l}
\mathfrak{J}_{12}, \mathfrak{J}_{30}, \mathfrak{J}_{32}, \mathfrak{J}_{34}, \mathfrak{J}_{71}, \mathfrak{J}_{72}, \mathbb{C}^4
\end{array} \right\};\\
\mathcal{C}_{06}   &=\overline{ {\mathcal O}(\mathfrak{J}_{13})  }= 
\left\{\begin{array}{l}
\mathfrak{J}_{13}, \mathfrak{J}_{28}, \mathfrak{J}_{32}, \mathfrak{J}_{60}, \mathfrak{J}_{69}, \mathfrak{J}_{71}, \mathfrak{J}_{72}, \mathbb{C}^4
\end{array} \right\};\\
\mathcal{C}_{07}   &=\overline{ {\mathcal O}(\mathfrak{J}_{16})  }= 
\left\{\begin{array}{l}
\mathfrak{J}_{16}, \mathfrak{J}_{28}, \mathfrak{J}_{30}, \mathfrak{J}_{31}, \mathfrak{J}_{50}, \mathfrak{J}_{69}, \mathfrak{J}_{71}, \mathfrak{J}_{72}, \mathbb{C}^4
\end{array} \right\};\\
\mathcal{C}_{08}   &=\overline{ {\mathcal O}(\mathfrak{J}_{24})  }= 
\left\{\begin{array}{l}
\mathfrak{J}_{14}, \mathfrak{J}_{24}, \mathfrak{J}_{29}, \mathfrak{J}_{31}, \mathfrak{J}_{34}, \mathfrak{J}_{40}, \mathfrak{J}_{41}, \mathfrak{J}_{51}, \mathfrak{J}_{55}, \mathfrak{J}_{56}, \mathfrak{J}_{57}, \mathfrak{J}_{63}, \mathfrak{J}_{64}, \mathfrak{J}_{66}, \\ \mathfrak{J}_{67}, \mathfrak{J}_{68}, \mathfrak{J}_{69}, \mathfrak{J}_{70}, \mathfrak{J}_{71}, \mathfrak{J}_{72}, \mathbb{C}^4
\end{array} \right\};\\
\mathcal{C}_{09}   &=\overline{ {\mathcal O}(\mathfrak{J}_{33})  }= 
\left\{\begin{array}{l}
\mathfrak{J}_{33}, \mathbb{C}^4
\end{array} \right\};\\
\mathcal{C}_{10}   &=\overline{ {\mathcal O}(\mathfrak{J}_{59})  }= 
\left\{\begin{array}{l}
\mathfrak{J}_{32}, \mathfrak{J}_{58}, \mathfrak{J}_{59}, \mathfrak{J}_{70}, \mathfrak{J}_{71}, \mathfrak{J}_{72}, \mathbb{C}^4
\end{array} \right\}.\\
\end{aligned}
$$ 
In particular, 
${\rm Rig}(\mathfrak{Jord}_4)= 
\big\{ \mathfrak{J}_{01}, \ \mathfrak{J}_{02}, \ \mathfrak{J}_{03}, \ \mathfrak{J}_{06}, \ \mathfrak{J}_{12}, \ \mathfrak{J}_{13}, \ \mathfrak{J}_{16}, \ \mathfrak{J}_{24}, \ \mathfrak{J}_{33}, \ \mathfrak{J}_{59} \big\}.$
\end{corollaryA}

\section{Degenerations of $5$-dimensional nilpotent Jordan algebras}\label{sec5}

Let $\mathfrak{NJord}_5$ be the set of all nilpotent Jordan algebra structures on $\mathbb{V}=\mathbb{C}^5$; it is a Zariski-closed ${\rm GL}(\mathbb{V})$-stable subset of $\mathfrak{Jord}_5$, and ${\rm dim}\, {\mathcal O}(\mu)=25-{\rm dim}\,{\rm Der}(\mu)$.
In Table~4 we list all $5$-dimensional nilpotent Jordan algebras.

\begin{lemma}\label{lem_table5}
Every complex $5$-dimensional nilpotent Jordan algebra is isomorphic to one of the algebras of Table~4, where $\mathfrak{N}_{39}^{\alpha}$, $\mathfrak{N}_{47}^{\alpha}$, $\mathfrak{N}_{48}^{\alpha}$, $\mathfrak{N}_{50}^{\alpha}$ are one-parameter families and $\mathfrak{N}_{51}^{\alpha,\beta}$ is a two-parameter family (the algebraic classification {\rm \cite{M80,ha}}, see Remark~\ref{rem:ha}). Here the algebras $\mathfrak{N}_{k}$ with $k\le24$ are the commutative associative nilpotent algebras of Mazzola {\rm \cite{M80}}: $\mathfrak{N}_{k}=\varepsilon_k$ for $1\le k\le3$ and $\mathfrak{N}_{k-1}=\varepsilon_k$ for $5\le k\le25$; the remaining algebra $\varepsilon_4$ of Mazzola is the member $\mathfrak{N}_{48}^{1}$ of the family $\mathfrak{N}_{48}^{\alpha}$ (see Remark~\ref{rem:ha}). The algebras $\mathfrak{N}_{k}$ with $k\ge25$ are the non-associative algebras of Hegazi and Abdelwahab {\rm \cite{ha}}: $\mathfrak{N}_{k+24}=J_k$ for $1\le k\le28$, $\mathfrak{N}_{k+23}=J_k$ for $30\le k\le36$, $\mathfrak{N}_{k+22}=J_k$ for $38\le k\le40$, and $\mathfrak{N}_{63}=J_{41}^0$; $\mathfrak{N}_{24}=\mathbb{C}^5$ is the algebra with zero multiplication. As in Table~1, $e_1,\dots,e_5$ is a basis of the algebra, the products of basis elements which are not listed are zero, and the columns $\mathfrak{Der}$, ${\rm Ann}$, $\mathbb{A}^2$ give the dimensions of the algebra of derivations, of the annihilator and of the square (for a family: of a generic member).
\end{lemma}

The special member $\mathfrak{N}_{48}^{-\frac12}$ of the family $\mathfrak{N}_{48}^{\alpha}$ has an $8$-dimensional algebra of derivations (the generic member has $7$), so it is treated as a separate object. The member $\mathfrak{N}_{48}^{1}$ is the only associative member of the family (it is the algebra $\varepsilon_4$ of Mazzola); it is also treated as a separate object. The families $\mathfrak{N}_{39}^{\alpha}$, $\mathfrak{N}_{47}^{\alpha}$, $\mathfrak{N}_{50}^{\alpha}$, $\mathfrak{N}_{51}^{\alpha,\beta}$ have generic orbits of dimension $19$, and $\mathfrak{N}_{48}^{\alpha}$ has generic orbits of dimension $18$; the closures of the families have dimensions $20$, $20$, $20$, $21$ and $19$, respectively. Altogether we have $65$ objects: the $63$ normal forms of Table~4 and the members $\mathfrak{N}_{48}^{1}$ and $\mathfrak{N}_{48}^{-\frac12}$.

\begin{remark}[on the list of Hegazi and Abdelwahab]\label{rem:ha}
The list of {\rm \cite{ha}} contains some algebras which are isomorphic to other algebras of the list
(this was observed by Kashuba and Martin {\rm \cite[Theorem~4]{km18}}; we have rechecked all these
isomorphisms by explicit changes of basis):
\begin{enumerate}
\item $J_{24}^{1}\cong\varepsilon_4$ (the two tables coincide), i.e. the associative algebra $\varepsilon_4$ is the member $\mathfrak{N}_{48}^{1}$ of the non-associative family
$\mathfrak{N}_{48}^{\alpha};$

\item $J_{29}\cong J_{24}^0=\mathfrak{N}_{48}^{0}$: $J_{29}$ ($e_1e_1=e_2$, $e_2e_3=e_4$, $e_1e_3=e_5$) becomes
$\mathfrak{N}_{48}^{0}$ in the basis $e_1,e_3,e_2,e_5,e_4;$

\item $J_{37}\cong J_{27}^{-1,-1}=\mathfrak{N}_{51}^{-1,-1}$: $J_{37}$ ($e_1e_1=e_2e_2=e_5$, $e_1e_2=e_3$,
$e_1e_3=e_4$) becomes $\mathfrak{N}_{51}^{-1,-1}$ in the basis $-e_1+e_2,\ e_1+e_2,\ -2e_3+2e_5,\ 2e_3+2e_5,\
-2e_4;$

\item $J_{41}^{\lambda}\cong J_{27}^{\varepsilon,\varepsilon^{-1}}=\mathfrak{N}_{51}^{\varepsilon,\varepsilon^{-1}}$
for $\lambda=s^2\neq0,1$, where $\varepsilon=-\frac{(s-1)^2}{(s+1)^2}$: $J_{41}^\lambda$ ($e_1e_1=e_5$,
$e_1e_2=e_3$, $e_2e_2=\lambda e_5$, $e_1e_3=e_2e_3=e_4$) becomes
$\mathfrak{N}_{51}^{\varepsilon,\varepsilon^{-1}}$ in the basis
\begin{center}$se_1+e_2,\ \frac{s-1}{s+1}(-se_1+e_2),\ 2s(e_3+se_5),\ -\frac{2s(s-1)^2}{(s+1)^2}(e_3-se_5),\
-\frac{2s(s-1)^2}{s+1}e_4;$\end{center}

\item $J_{41}^1\cong J_{26}^0=\mathfrak{N}_{50}^{0}$: $J_{41}^1$ becomes $\mathfrak{N}_{50}^{0}$ in the basis
\begin{center}$e_1+e_2,\ {\bf i}(-e_1+e_2),\ 2e_3+2e_5,\ 2e_3-2e_5,\ 4e_4$.\end{center}
\end{enumerate}
Therefore $J_{29}$, $J_{37}$ and the members $J_{41}^\lambda$, $\lambda\neq0$, are omitted from our
list, the family $J_{41}$ is represented only by $\mathfrak{N}_{63}=J_{41}^0$, and in the family
$\mathfrak{N}_{48}^{\alpha}=J_{24}^\alpha$ the value $\alpha=1$ gives the associative algebra $\varepsilon_4$, which is therefore not listed separately. The remaining $63$
normal forms are pairwise non-isomorphic: they are separated by the invariants of
Table~4 and of {\rm \cite[Tables~2--3]{km18}} or, for the remaining pairs, by the
non-degenerations of Theorem~B. We did not find any further coincidences, and we
found no indication of algebras missing in {\rm \cite{ha}}; note, however, that the completeness of the
classification of {\rm \cite{ha}} is not re-proved here. Finally, in the multiplication table of
$\varepsilon_2$ in {\rm \cite[Table~2]{km18}} the product $n_5^2=n_4$ is missing (without it the
annihilator would be $2$-dimensional, while ${\rm dim}\,{\rm Ann}(\varepsilon_2)=1$ in
{\rm \cite[Table~2]{km18}}); our $\mathfrak{N}_{02}$ contains the product $e_5e_5=e_4$.
\end{remark}

\begingroup
\small
\setlength{\tabcolsep}{4pt}
\renewcommand{\arraystretch}{1.2}
\begin{longtable}{|l|>{\raggedright\arraybackslash}p{2.2cm}|>{\raggedright\arraybackslash}p{8.3cm}|c|c|c|}
\multicolumn{6}{c}{Table 4. {\it $5$-dimensional nilpotent Jordan algebras.}}\\
\hline
$\mathbb{A}$ & \mbox{source} & \mbox{multiplication table} & $\mathfrak{Der}$ & ${\rm Ann}$ & $\mathbb{A}^2$\\
\hline
\hline
\endhead
$\mathfrak{N}_{01}$ & \cite{M80}: $\varepsilon_{1}$ & $e_1e_1 = e_2$, $e_2e_2 = e_4$, $e_1e_3 = e_4$,\newline $e_1e_2 = e_3$, $e_1e_4 = e_5$, $e_2e_3 = e_5$ & 5 & 1 & 4\\
\hline
$\mathfrak{N}_{02}$ & \cite{M80}: $\varepsilon_{2}$ & $e_1e_1 = e_2$, $e_1e_2 = e_3$, $e_2e_2 = e_4$, $e_1e_3 = e_4$, $e_5e_5 = e_4$ & 6 & 1 & 3\\
\hline
$\mathfrak{N}_{03}$ & \cite{M80}: $\varepsilon_{3}$ & $e_1e_1 = e_2$, $e_4e_4 = e_5$, $e_1e_2 = e_3$, $e_4e_5 = e_3$ & 6 & 1 & 3\\
\hline
$\mathfrak{N}_{04}$ & \cite{M80}: $\varepsilon_{5}$ & $e_1e_1 = e_2$, $e_2e_2 = e_4$, $e_1e_2 = e_3$, $e_1e_3 = e_4$ & 7 & 2 & 3\\
\hline
$\mathfrak{N}_{05}$ & \cite{M80}: $\varepsilon_{6}$ & $e_1e_1 = e_2$, $e_1e_2 = e_3$, $e_4e_4 = e_5$ & 7 & 2 & 3\\
\hline
$\mathfrak{N}_{06}$ & \cite{M80}: $\varepsilon_{7}$ & $e_1e_3 = e_4$, $e_2e_3 = e_5$, $e_1e_2 = e_4 - e_5$ & 7 & 2 & 2\\
\hline
$\mathfrak{N}_{07}$ & \cite{M80}: $\varepsilon_{8}$ & $e_1e_1 = e_3$, $e_2e_2 = e_4$, $e_1e_3 = e_4$, $e_1e_2 = e_5$ & 8 & 2 & 3\\
\hline
$\mathfrak{N}_{08}$ & \cite{M80}: $\varepsilon_{9}$ & $e_1e_3 = e_5$, $e_1e_2 = e_4$, $e_2e_3 = -e_5$ & 8 & 2 & 2\\
\hline
$\mathfrak{N}_{09}$ & \cite{M80}: $\varepsilon_{10}$ & $e_1e_1 = e_3$, $e_1e_3 = e_4$, $e_1e_2 = e_5$ & 9 & 2 & 3\\
\hline
$\mathfrak{N}_{10}$ & \cite{M80}: $\varepsilon_{11}$ & $e_1e_1 = e_4$, $e_2e_3 = e_4$, $e_1e_3 = e_5$ & 9 & 2 & 2\\
\hline
$\mathfrak{N}_{11}$ & \cite{M80}: $\varepsilon_{12}$ & $e_1e_2 = e_4$, $e_1e_3 = e_5$ & 11 & 2 & 2\\
\hline
$\mathfrak{N}_{12}$ & \cite{M80}: $\varepsilon_{13}$ & $e_3e_3 = e_4$, $e_1e_2 = e_5$, $e_3e_4 = e_5$ & 8 & 1 & 2\\
\hline
$\mathfrak{N}_{13}$ & \cite{M80}: $\varepsilon_{14}$ & $e_1e_1 = e_2$, $e_4e_4 = e_3$, $e_1e_2 = e_3$ & 9 & 2 & 2\\
\hline
$\mathfrak{N}_{14}$ & \cite{M80}: $\varepsilon_{15}$ & $e_1e_1 = e_3$, $e_2e_2 = e_3$, $e_4e_4 = e_5$ & 9 & 2 & 2\\
\hline
$\mathfrak{N}_{15}$ & \cite{M80}: $\varepsilon_{16}$ & $e_1e_1 = e_3$, $e_2e_2 = e_5$, $e_1e_2 = e_4$ & 10 & 3 & 3\\
\hline
$\mathfrak{N}_{16}$ & \cite{M80}: $\varepsilon_{17}$ & $e_1e_1 = e_4$, $e_3e_3 = e_5$, $e_1e_2 = e_5$ & 10 & 2 & 2\\
\hline
$\mathfrak{N}_{17}$ & \cite{M80}: $\varepsilon_{18}$ & $e_1e_1 = e_2$, $e_1e_2 = e_3$ & 11 & 3 & 2\\
\hline
$\mathfrak{N}_{18}$ & \cite{M80}: $\varepsilon_{19}$ & $e_1e_1 = e_2$, $e_3e_3 = e_4$ & 11 & 3 & 2\\
\hline
$\mathfrak{N}_{19}$ & \cite{M80}: $\varepsilon_{20}$ & $e_1e_1 = e_2$, $e_1e_3 = e_4$ & 12 & 3 & 2\\
\hline
$\mathfrak{N}_{20}$ & \cite{M80}: $\varepsilon_{21}$ & $e_2e_3 = e_5$, $e_1e_4 = e_5$ & 11 & 1 & 1\\
\hline
$\mathfrak{N}_{21}$ & \cite{M80}: $\varepsilon_{22}$ & $e_1e_1 = e_2$, $e_3e_4 = e_2$ & 12 & 2 & 1\\
\hline
$\mathfrak{N}_{22}$ & \cite{M80}: $\varepsilon_{23}$ & $e_1e_1 = e_3$, $e_2e_2 = e_3$ & 14 & 3 & 1\\
\hline
$\mathfrak{N}_{23}$ & \cite{M80}: $\varepsilon_{24}$ & $e_1e_1 = e_2$ & 17 & 4 & 1\\
\hline
$\mathfrak{N}_{24}$ & \cite{M80}: $\varepsilon_{25}$ & - & 25 & 5 & 0\\
\hline
$\mathfrak{N}_{25}$ & \cite{ha}: $J_{1}$ & $e_1e_1 = e_2$, $e_1e_2 = e_3$, $e_2e_4 = e_3$ & 8 & 2 & 2\\
\hline
$\mathfrak{N}_{26}$ & \cite{ha}: $J_{2}$ & $e_1e_1 = e_2$, $e_4e_4 = -e_2$, $e_1e_2 = e_3$, $e_2e_4 = e_3$ & 9 & 2 & 2\\
\hline
$\mathfrak{N}_{27}$ & \cite{ha}: $J_{3}$ & $e_1e_1 = e_2$, $e_4e_4 = -e_2 - e_3$, $e_1e_2 = e_3$, $e_2e_4 = e_3$ & 8 & 2 & 2\\
\hline
$\mathfrak{N}_{28}$ & \cite{ha}: $J_{4}$ & $e_1e_1 = e_2$, $e_4e_4 = e_2$, $e_1e_2 = e_3$ & 7 & 2 & 2\\
\hline
$\mathfrak{N}_{29}$ & \cite{ha}: $J_{5}$ & $e_1e_1 = e_2$, $e_4e_4 = e_5$, $e_2e_3 = e_5$ & 6 & 1 & 2\\
\hline
$\mathfrak{N}_{30}$ & \cite{ha}: $J_{6}$ & $e_1e_1 = e_2$, $e_1e_4 = e_5$, $e_2e_3 = e_5$ & 7 & 1 & 2\\
\hline
$\mathfrak{N}_{31}$ & \cite{ha}: $J_{7}$ & $e_1e_2 = e_3$, $e_3e_4 = e_5$ & 6 & 1 & 2\\
\hline
$\mathfrak{N}_{32}$ & \cite{ha}: $J_{8}$ & $e_1e_2 = e_3$, $e_3e_4 = e_5$, $e_1e_1 = e_5$ & 5 & 1 & 2\\
\hline
$\mathfrak{N}_{33}$ & \cite{ha}: $J_{9}$ & $e_1e_2 = e_3$, $e_3e_4 = e_5$, $e_1e_1 = e_5$, $e_2e_2 = e_5$ & 4 & 1 & 2\\
\hline
$\mathfrak{N}_{34}$ & \cite{ha}: $J_{10}$ & $e_1e_2 = e_3$, $e_1e_3 = e_5$, $e_4e_4 = e_5$, $e_2e_3 = e_5$ & 5 & 1 & 2\\
\hline
$\mathfrak{N}_{35}$ & \cite{ha}: $J_{11}$ & $e_1e_2 = e_3$, $e_1e_3 = e_5$, $e_4e_4 = e_5$ & 7 & 1 & 2\\
\hline
$\mathfrak{N}_{36}$ & \cite{ha}: $J_{12}$ & $e_1e_2 = e_3$, $e_1e_3 = e_5$, $e_4e_4 = e_5$, $e_2e_2 = e_5$ & 6 & 1 & 2\\
\hline
$\mathfrak{N}_{37}$ & \cite{ha}: $J_{13}$ & $e_1e_2 = e_3$, $e_1e_4 = e_5$, $e_1e_3 = e_5$, $e_2e_3 = e_5$ & 6 & 1 & 2\\
\hline
$\mathfrak{N}_{38}$ & \cite{ha}: $J_{14}$ & $e_1e_2 = e_3$, $e_1e_3 = e_5$, $e_2e_4 = e_5$ & 7 & 1 & 2\\
\hline
$\mathfrak{N}_{39}^{\alpha}$ & \cite{ha}: $J_{15}$ & $e_1e_1 = e_2$, $e_1e_2 = e_3$, $e_4e_4 = \alpha e_5$,\newline $e_2e_2 = e_5$, $e_1e_3 = e_5$, $e_2e_4 = e_5$ & 6 & 1 & 3\\
\hline
$\mathfrak{N}_{40}$ & \cite{ha}: $J_{16}$ & $e_1e_2 = e_4$, $e_3e_3 = e_4$, $e_1e_4 = e_5$ & 6 & 1 & 2\\
\hline
$\mathfrak{N}_{41}$ & \cite{ha}: $J_{17}$ & $e_1e_2 = e_4$, $e_3e_3 = e_4$, $e_1e_4 = e_5$, $e_2e_3 = e_5$ & 5 & 1 & 2\\
\hline
$\mathfrak{N}_{42}$ & \cite{ha}: $J_{18}$ & $e_1e_2 = e_4$, $e_3e_3 = e_4$, $e_1e_4 = e_5$, $e_2e_2 = e_5$ & 4 & 1 & 2\\
\hline
$\mathfrak{N}_{43}$ & \cite{ha}: $J_{19}$ & $e_1e_2 = e_4$, $e_3e_3 = e_4$, $e_3e_4 = e_5$ & 5 & 1 & 2\\
\hline
$\mathfrak{N}_{44}$ & \cite{ha}: $J_{20}$ & $e_1e_2 = e_4$, $e_3e_3 = e_4$, $e_3e_4 = e_5$, $e_1e_1 = e_5$ & 4 & 1 & 2\\
\hline
$\mathfrak{N}_{45}$ & \cite{ha}: $J_{21}$ & $e_1e_2 = e_4$, $e_3e_3 = e_4$, $e_3e_4 = e_5$, $e_1e_1 = e_5$, $e_2e_2 = e_5$ & 3 & 1 & 2\\
\hline
$\mathfrak{N}_{46}$ & \cite{ha}: $J_{22}$ & $e_1e_1 = e_2$, $e_1e_2 = e_4$, $e_3e_3 = e_4$, $e_2e_2 = e_5$, $e_1e_4 = e_5$ & 4 & 1 & 3\\
\hline
$\mathfrak{N}_{47}^{\alpha}$ & \cite{ha}: $J_{23}$ & $e_1e_1 = e_3$, $e_1e_3 = e_5$, $e_2e_4 = e_5$,\newline $e_1e_2 = e_4$, $e_2e_3 = \alpha e_5$ & 6 & 1 & 3\\
\hline
$\mathfrak{N}_{48}^{\alpha}$ & \cite{ha}: $J_{24}$, \cite{M80}: $\varepsilon_4=J_{24}^1$ & $e_1e_1 = e_3$, $e_1e_2 = e_4$, $e_2e_3 = e_5$, $e_1e_4 = \alpha e_5$ & 7 & 1 & 3\\
\hline
$\mathfrak{N}_{49}$ & \cite{ha}: $J_{25}$ & $e_1e_1 = e_3$, $e_1e_2 = e_4$, $e_1e_3 = e_5$,\newline $e_1e_4 = e_5$, $e_2e_3 = -2 e_5$ & 7 & 1 & 3\\
\hline
$\mathfrak{N}_{50}^{\alpha}$ & \cite{ha}: $J_{26}$ & $e_1e_1 = e_3$, $e_2e_2 = e_4$, $e_2e_4 = \alpha e_5$,\newline $e_1e_3 = e_5$, $e_1e_4 = e_5$ & 6 & 1 & 3\\
\hline
$\mathfrak{N}_{51}^{\alpha,\beta}$ & \cite{ha}: $J_{27}$ & $e_1e_1 = e_3$, $e_2e_3 = e_5$, $e_1e_4 = e_5$,\newline $e_2e_2 = e_4$, $e_2e_4 = \alpha e_5$, $e_1e_3 = \beta e_5$ & 6 & 1 & 3\\
\hline
$\mathfrak{N}_{52}$ & \cite{ha}: $J_{28}$ & $e_1e_1 = e_2$, $e_2e_3 = e_4$, $e_3e_3 = e_5$ & 7 & 2 & 3\\
\hline
$\mathfrak{N}_{53}$ & \cite{ha}: $J_{30}$ & $e_1e_1 = e_2$, $e_2e_3 = e_4$, $e_1e_3 = e_5$, $e_3e_3 = e_5$ & 6 & 2 & 3\\
\hline
$\mathfrak{N}_{54}$ & \cite{ha}: $J_{31}$ & $e_1e_1 = e_2$, $e_2e_3 = e_4$, $e_1e_2 = e_5$ & 7 & 2 & 3\\
\hline
$\mathfrak{N}_{55}$ & \cite{ha}: $J_{32}$ & $e_1e_1 = e_2$, $e_2e_3 = e_4$, $e_1e_2 = e_5$, $e_1e_3 = e_5$ & 6 & 2 & 3\\
\hline
$\mathfrak{N}_{56}$ & \cite{ha}: $J_{33}$ & $e_1e_1 = e_2$, $e_2e_3 = e_4$, $e_1e_2 = e_5$, $e_3e_3 = e_5$ & 5 & 2 & 3\\
\hline
$\mathfrak{N}_{57}$ & \cite{ha}: $J_{34}$ & $e_1e_2 = e_3$, $e_1e_3 = e_4$, $e_1e_1 = e_5$ & 8 & 2 & 3\\
\hline
$\mathfrak{N}_{58}$ & \cite{ha}: $J_{35}$ & $e_1e_2 = e_3$, $e_1e_3 = e_4$, $e_2e_2 = e_5$ & 7 & 2 & 3\\
\hline
$\mathfrak{N}_{59}$ & \cite{ha}: $J_{36}$ & $e_1e_2 = e_3$, $e_1e_3 = e_4$, $e_2e_3 = e_5$ & 6 & 2 & 3\\
\hline
$\mathfrak{N}_{60}$ & \cite{ha}: $J_{38}$ & $e_1e_2 = e_3$, $e_1e_3 = e_4$, $e_1e_1 = e_5$, $e_2e_3 = e_5$ & 5 & 2 & 3\\
\hline
$\mathfrak{N}_{61}$ & \cite{ha}: $J_{39}$ & $e_1e_2 = e_3$, $e_1e_1 = e_5$, $e_2e_2 = e_4$, $e_1e_3 = e_4$ & 7 & 2 & 3\\
\hline
$\mathfrak{N}_{62}$ & \cite{ha}: $J_{40}$ & $e_1e_2 = e_3$, $e_2e_2 = e_4$, $e_1e_3 = e_4$, $e_1e_1 = e_5$, $e_2e_3 = e_5$ & 4 & 2 & 3\\
\hline
$\mathfrak{N}_{63}$ & \cite{ha}: $J_{41}$ & $e_1e_1 = e_5$, $e_1e_2 = e_3$, $e_1e_3 = e_4$, $e_2e_3 = e_4$ & 6 & 2 & 3\\
\hline
\end{longtable}
\endgroup

\begin{theoremB}\label{theorem5}
The graph of primary degenerations for $5$-dimensional nilpotent Jordan algebras has the form given in the figure below (the rows of the graph correspond to the dimension of the orbit closure): an algebra or a family $\mathfrak{A}$ degenerates to an algebra or a family $\mathfrak{B}$ if and only if there is a path from $\mathfrak{A}$ to $\mathfrak{B}$ in the graph. The order consists of $1127$ degenerations and $3033$ non-degenerations between the $65$ objects: the $154$ primary degenerations (Table~5) give the other $973$ degenerations by transitivity; the $122$ primary non-degenerations (Table~6: $47$ pairs listed in the rows and $75$ pairs given by the first row) give $588$ further non-degenerations, and the remaining $2323$ non-degenerations follow from Lemma~\ref{lem_der}.
\end{theoremB}

\begin{proof}
By Lemma~\ref{lem_der} and transitivity, it is enough to prove that the arrows of the graph are degenerations (the primary degenerations) and that $\mathfrak{A}\not\to\mathfrak{B}$ for all primary non-degenerations $\mathfrak{A}\not\to\mathfrak{B}$.

\medskip

\noindent{\bf Degenerations.} The list of primary degenerations has $154$ elements. For each of them a parametrized basis is given in Table~5; for a family source, the member is indicated by its parameter in the superscript (it may depend on $t$ or on an auxiliary parameter $p$); for a family target, the member is indicated by its parameter $\alpha$ in the superscript, and the parametrized basis is given for an arbitrary $\alpha\in\mathbb{C}$. The degeneration $\mathfrak{N}_{48}^{1+t}\to\mathfrak{N}_{48}^{1}$ holds since $\mathfrak{N}_{48}^{1}$ is a member of the family and the multiplication table of $\mathfrak{N}_{48}^{\alpha}$ depends polynomially on $\alpha$ (take ${\rm E}^t_i=e_i$), and $\mathfrak{N}_{48}^{t-\frac12}\to\mathfrak{N}_{48}^{-\frac12}$ for the same reason.

\begingroup
\small
\begin{longtable}{|c|>{\raggedright\arraybackslash}p{11.6cm}|}
\caption*{Table 5. {\it Primary degenerations of $5$-dimensional nilpotent Jordan algebras.}}\\
\hline
\mbox{Degeneration} & \mbox{Parametrized basis} \\
\hline
\endfirsthead
\hline
\mbox{Degeneration} & \mbox{Parametrized basis} \\
\hline
\endhead
$\mathfrak{N}_{01} \to \mathfrak{N}_{02}$ & ${\rm E}^t_1=e_1$, \ ${\rm E}^t_2=e_2$, \ ${\rm E}^t_3=e_3$,\newline ${\rm E}^t_4=e_4$, \ ${\rm E}^t_5={\bf i}e_2 - {\bf i}t^{-1}e_3 + {\bf i}t^{-2}e_4 - {\bf i}t^{-3}e_5$ \\
\hline
$\mathfrak{N}_{01} \to \mathfrak{N}_{03}$ & ${\rm E}^t_1=e_1$, \ ${\rm E}^t_2=e_2$, \ ${\rm E}^t_3=e_3$, \ ${\rm E}^t_4=e_1 + t^{-1}e_2 + t^{-2}e_3 + t^{-3}e_4 + t^{-4}e_5$,\newline ${\rm E}^t_5=e_2 + 2t^{-1}e_3 + 3t^{-2}e_4 + 4t^{-3}e_5$ \\
\hline
$\mathfrak{N}_{02} \to \mathfrak{N}_{48}^{1}$ & ${\rm E}^t_1=-{\bf i}t^{2}e_1$, \ ${\rm E}^t_2=-t^{3}e_2 - {\bf i}t^{3}e_5$, \ ${\rm E}^t_3={\bf i}t^{4}e_5$,\newline ${\rm E}^t_4={\bf i}t^{5}e_3$, \ ${\rm E}^t_5=t^{7}e_4$ \\
\hline
$\mathfrak{N}_{02} \to \mathfrak{N}_{04}$ & ${\rm E}^t_1=t^{2}e_1$, \ ${\rm E}^t_2=t^{4}e_2$, \ ${\rm E}^t_3=t^{6}e_3$, \ ${\rm E}^t_4=t^{8}e_4$, \ ${\rm E}^t_5=t^{5}e_5$ \\
\hline
$\mathfrak{N}_{02} \to \mathfrak{N}_{05}$ & ${\rm E}^t_1=t^{3}e_1$, \ ${\rm E}^t_2=t^{6}e_2$, \ ${\rm E}^t_3=t^{9}e_3$, \ ${\rm E}^t_4=-t^{4}e_5$, \ ${\rm E}^t_5=t^{8}e_4$ \\
\hline
$\mathfrak{N}_{02} \to \mathfrak{N}_{06}$ & ${\rm E}^t_1=-{\bf i}te_2 + te_5$, \ ${\rm E}^t_2=\frac{1}{2}{\bf i}te_1 + {\bf i}te_3 + te_5$, \ ${\rm E}^t_3={\bf i}te_2 + te_5$,\newline ${\rm E}^t_4=2t^{2}e_4$, \ ${\rm E}^t_5=-\frac{1}{2}t^{2}e_3 + t^{2}e_4$ \\
\hline
$\mathfrak{N}_{03} \to \mathfrak{N}_{48}^{1}$ & ${\rm E}^t_1=-\frac{\sqrt[3]{4}}{2}t^{2}e_1 + \frac{\sqrt[3]{4}}{2}t^{2}e_4$, \ ${\rm E}^t_2=\frac{\sqrt[3]{4}}{2}t^{3}e_1 + \frac{\sqrt[3]{4}}{2}t^{3}e_4$, \ ${\rm E}^t_3=\frac{\sqrt[3]{2}}{2}t^{4}e_2 + \frac{\sqrt[3]{2}}{2}t^{4}e_5$,\newline ${\rm E}^t_4=-\frac{\sqrt[3]{2}}{2}t^{5}e_2 + \frac{\sqrt[3]{2}}{2}t^{5}e_5$, \ ${\rm E}^t_5=t^{7}e_3$ \\
\hline
$\mathfrak{N}_{03} \to \mathfrak{N}_{05}$ & ${\rm E}^t_1=t^{3}e_1$, \ ${\rm E}^t_2=t^{6}e_2$, \ ${\rm E}^t_3=t^{9}e_3$, \ ${\rm E}^t_4=t^{4}e_4$, \ ${\rm E}^t_5=t^{8}e_5$ \\
\hline
\end{longtable}
\endgroup

\medskip

\noindent{\bf Non-degenerations.} The primary non-degenerations are given in Table~6: the first row treats all pairs consisting of an associative and a non-associative algebra at once, and the other $30$ rows contain $47$ primary non-degenerations, and the proofs follow the scheme described in Section~\ref{prelim}. In all rows of Table~6 in which $\mathcal{R}$ is defined in a basis $f_1,\dots,f_5$ of $\mathbb{V}$, except the row $\mathfrak{N}_{43}\not\to\mathfrak{N}_{30}$, the condition of nilpotency
$$A_iA_j\subseteq A_{\max(i,j)+1}\ \mbox{ for all } i,j,\quad\mbox{i.e. } c_{ij}^k=0 \mbox{ for } k\le\max(i,j),$$
is also included in $\mathcal{R}$ and is not repeated in the table. The other conditions are either equations $c_{ij}^k=0$ (which follow from conditions $A_aA_b\subseteq A_{k+1}$ preserved by $\mathcal{B}$), or polynomial equations in the $c_{ij}^k$ which were checked to hold for the structure of $\mathfrak{A}$ in every basis $f_i\mapsto\sum_{j\ge i}b_{ij}f_j$, or the polynomial identities $(xy)z=x(yz)$ and $x(xx)=0$ (Lemma~\ref{lem_inv}). The listed equations $c_{ij}^k=0$ are irredundant: if any one of them is omitted, then one of the targets of the row admits a basis satisfying the remaining conditions. The conditions $\mathcal{R}\cap {\mathcal O}(\lambda)=\varnothing$ were checked by explicit Nullstellensatz certificates on the Bruhat cells (found by Gr\"obner basis computations and verified in Lean~4).

\begingroup
\footnotesize
\renewcommand{\arraystretch}{1.3}\setlength{\tabcolsep}{2pt}
\begin{longtable}{|l|l|}
\caption*{Table 6. {\it Primary non-degenerations of $5$-dimensional nilpotent Jordan algebras.}}\\
\hline
\mbox{Non-degeneration}  &  \mbox{Reason}\\
\hline
\endfirsthead
\hline
\mbox{Non-degeneration}  &  \mbox{Reason}\\
\hline
\endhead
$\begin{array}{@{}l@{}}\mbox{associative algebras}\not\to\\ \mbox{non-associative algebras}\end{array}$ & $\begin{array}{@{}l@{}}
\mathcal{R}=\big\{A\mid (xy)z=x(yz) \mbox{ for all $x,y,z\in A$}\big\}\\
\mathfrak{A}\in\mathcal{R} \mbox{ for every associative algebra $\mathfrak{A}$ (i.e. $\mathfrak{N}_{48}^{1}$ or $\mathfrak{N}_{k}$ with $k\le24$)},\\
\mbox{but } \mathfrak{B}\cap\mathcal{R}=\varnothing \mbox{ for every non-associative algebra $\mathfrak{B}$}
\end{array}$ \\
\hline
$\mathfrak{N}_{04}\not\to \left\{\begin{array}{l}\mathfrak{N}_{14}, \mathfrak{N}_{20}\end{array}\right\}$ & $\begin{array}{@{}l@{}}
\mathcal{R}=\left\{A\ \middle|\ \begin{array}{l}c_{13}^{4}=c_{22}^{3}=c_{22}^{4}=c_{23}^{4}=c_{23}^{5}=0,\\
c_{24}^{5}=c_{33}^{4}=c_{33}^{5}=c_{34}^{5}=c_{44}^{5}=0\end{array}\right\}\\
\mathfrak{N}_{04}\in\mathcal{R} \mbox{ (take $f_1=e_{1}$, $f_2=e_{2}$, $f_3=e_{3}$, $f_4=e_{5}$, $f_5=e_{4}$)},\\
\mbox{but } \mathfrak{B}\cap\mathcal{R}=\varnothing \mbox{ for all $\mathfrak{B}$ in the left column}
\end{array}$ \\
\hline
$\mathfrak{N}_{06}\not\to \mathfrak{N}_{17}$ & $\begin{array}{@{}l@{}}
\mathcal{R}=\big\{A\mid x(xx)=0 \mbox{ for all $x\in A$}\big\}\\
\mathfrak{N}_{06}\in\mathcal{R},\\
\mbox{but } \mathfrak{N}_{17}\cap\mathcal{R}=\varnothing \mbox{ (take $x=e_1$)}
\end{array}$ \\
\hline
$\mathfrak{N}_{15}\not\to \left\{\begin{array}{l}\mathfrak{N}_{11}, \mathfrak{N}_{21}\end{array}\right\}$ & $\begin{array}{@{}l@{}}
\mathcal{R}=\left\{A\ \middle|\ \begin{array}{l}c_{13}^{4}=c_{13}^{5}=c_{14}^{5}=c_{23}^{4}=c_{23}^{5}=0,\\
c_{24}^{5}=c_{33}^{4}=c_{33}^{5}=c_{34}^{5}=c_{44}^{5}=0\end{array}\right\}\\
\mathfrak{N}_{15}\in\mathcal{R},\\
\mbox{but } \mathfrak{B}\cap\mathcal{R}=\varnothing \mbox{ for all $\mathfrak{B}$ in the left column}
\end{array}$ \\
\hline
$\mathfrak{N}_{20}\not\to \mathfrak{N}_{19}$ & $\begin{array}{@{}l@{}}
\mathcal{R}=\left\{A\ \middle|\ \begin{array}{l}c_{11}^{2}=c_{11}^{3}=c_{11}^{4}=c_{12}^{3}=c_{12}^{4}=0,\\
c_{13}^{4}=c_{22}^{3}=c_{22}^{4}=c_{23}^{4}=c_{33}^{4}=0\end{array}\right\}\\
\mathfrak{N}_{20}\in\mathcal{R},\\
\mbox{but } \mathfrak{N}_{19}\cap\mathcal{R}=\varnothing
\end{array}$ \\
\hline
$\mathfrak{N}_{29}\not\to \mathfrak{N}_{11}$ & $\begin{array}{@{}l@{}}
\mathcal{R}=\big\{A\mid c_{12}^{3}=c_{12}^{4}=c_{13}^{4}=c_{23}^{4}=0\big\}\\
\mathfrak{N}_{29}\in\mathcal{R} \mbox{ (take $f_1=e_{1}$, $f_2=e_{3}$, $f_3=e_{4}$, $f_4=e_{2}$, $f_5=e_{5}$)},\\
\mbox{but } \mathfrak{N}_{11}\cap\mathcal{R}=\varnothing
\end{array}$ \\
\hline
$\mathfrak{N}_{31}\not\to \mathfrak{N}_{35}$ & $\begin{array}{@{}l@{}}
\mathcal{R}=\left\{A\ \middle|\ \begin{array}{l}c_{12}^{3}=c_{24}^{5}=c_{33}^{5}=c_{34}^{5}=0,\\
c_{11}^{4}c_{23}^{4}c_{23}^{5}-2c_{12}^{4}c_{13}^{4}c_{23}^{5}+(c_{13}^{4})^{2}c_{22}^{5}=0\end{array}\right\}\\
\mathfrak{N}_{31}\in\mathcal{R} \mbox{ (take $f_1=e_{4}$, $f_2=e_{1}$, $f_3=e_{2}$, $f_4=e_{3}$, $f_5=e_{5}$)},\\
\mbox{but } \mathfrak{N}_{35}\cap\mathcal{R}=\varnothing
\end{array}$ \\
\hline
$\mathfrak{N}_{33}\not\to \mathfrak{N}_{40}$ & $\begin{array}{@{}l@{}}
\mathcal{R}=\big\{A\mid c_{13}^{4}=c_{23}^{4}=c_{33}^{4}=0\big\}\\
\mathfrak{N}_{33}\in\mathcal{R} \mbox{ (take $f_1=e_{1}$, $f_2=e_{2}$, $f_3=e_{4}$, $f_4=e_{3}$, $f_5=e_{5}$)},\\
\mbox{but } \mathfrak{N}_{40}\cap\mathcal{R}=\varnothing
\end{array}$ \\
\hline
$\mathfrak{N}_{36}\not\to \mathfrak{N}_{25}$ & $\begin{array}{@{}l@{}}
\mathcal{R}=\big\{A\mid c_{11}^{2}=c_{11}^{3}=c_{22}^{3}=c_{22}^{4}=c_{24}^{5}=c_{33}^{4}=c_{34}^{5}=0\big\}\\
\mathfrak{N}_{36}\in\mathcal{R} \mbox{ (take $f_1=e_{1}$, $f_2=e_{2}$, $f_3=e_{4}$, $f_4=e_{3}$, $f_5=e_{5}$)},\\
\mbox{but } \mathfrak{N}_{25}\cap\mathcal{R}=\varnothing
\end{array}$ \\
\hline
$\mathfrak{N}_{39}^{\alpha}\not\to \left\{\begin{array}{l}\mathfrak{N}_{28}, \mathfrak{N}_{29},\\ \mathfrak{N}_{52}\end{array}\right\}$ & $\begin{array}{@{}l@{}}
\mathcal{R}=\big\{A\mid c_{12}^{3}=c_{22}^{3}=c_{22}^{4}=c_{23}^{4}=c_{24}^{5}=c_{33}^{4}=c_{34}^{5}=c_{44}^{5}=0\big\}\\
\mathfrak{N}_{39}^{\alpha}\in\mathcal{R} \mbox{ for all $\alpha$} \mbox{ (take $f_1=e_{1}$, $f_2=e_{2}$, $f_3=e_{4}$, $f_4=e_{3}$, $f_5=e_{5}$)},\\
\mbox{but } \mathfrak{B}\cap\mathcal{R}=\varnothing \mbox{ for all $\mathfrak{B}$ in the left column}
\end{array}$ \\
\hline
$\mathfrak{N}_{40}\not\to \mathfrak{N}_{06}$ & $\begin{array}{@{}l@{}}
\mathcal{R}=\big\{A\mid c_{23}^{4}=c_{23}^{5}=c_{33}^{4}=c_{33}^{5}=c_{44}^{5}=0\big\}\\
\mathfrak{N}_{40}\in\mathcal{R} \mbox{ (take $f_1=e_{1}$, $f_2=e_{3}$, $f_3=e_{2}$, $f_4=e_{4}$, $f_5=e_{5}$)},\\
\mbox{but } \mathfrak{N}_{06}\cap\mathcal{R}=\varnothing
\end{array}$ \\
\hline
$\mathfrak{N}_{40}\not\to \mathfrak{N}_{20}$ & $\begin{array}{@{}l@{}}
\mathcal{R}=\big\{A\mid c_{23}^{5}=c_{24}^{5}=c_{33}^{5}=c_{34}^{5}=0\big\}\\
\mathfrak{N}_{40}\in\mathcal{R} \mbox{ (take $f_1=e_{1}$, $f_2=e_{3}$, $f_3=e_{2}$, $f_4=e_{4}$, $f_5=e_{5}$)},\\
\mbox{but } \mathfrak{N}_{20}\cap\mathcal{R}=\varnothing
\end{array}$ \\
\hline
$\mathfrak{N}_{42}\not\to \mathfrak{N}_{31}$ & $\begin{array}{@{}l@{}}
\mathcal{R}=\big\{A\mid c_{23}^{4}=c_{24}^{5}=c_{34}^{5}=0\big\}\\
\mathfrak{N}_{42}\in\mathcal{R} \mbox{ (take $f_1=e_{1}$, $f_2=e_{3}$, $f_3=e_{2}$, $f_4=e_{4}$, $f_5=e_{5}$)},\\
\mbox{but } \mathfrak{N}_{31}\cap\mathcal{R}=\varnothing
\end{array}$ \\
\hline
$\mathfrak{N}_{43}\not\to \mathfrak{N}_{30}$ & $\begin{array}{@{}l@{}}
\mathcal{R}=\left\{A\ \middle|\ \begin{array}{l}c_{11}^{2}=c_{11}^{3}=c_{12}^{3}=c_{13}^{2}=c_{14}^{2}=c_{14}^{3}=0,\\
c_{15}^{2}=c_{22}^{1}=c_{22}^{3}=c_{22}^{4}=c_{23}^{1}=c_{23}^{4}=0,\\
c_{24}^{1}=c_{24}^{3}=c_{24}^{5}=c_{25}^{4}=c_{33}^{1}=c_{33}^{2}=0,\\
c_{33}^{4}=c_{34}^{1}=c_{34}^{5}=c_{35}^{4}=c_{44}^{1}=c_{44}^{2}=c_{44}^{3}=0\end{array}\right\}\\
\mbox{here $A_iA_j\subseteq A_{\max(i,j)+1}$ is not assumed}\\
\mathfrak{N}_{43}\in\mathcal{R} \mbox{ (take $f_1=e_{3}$, $f_2=e_{1}$, $f_3=e_{2}$, $f_4=e_{5}$, $f_5=e_{4}$)},\\
\mbox{but } \mathfrak{N}_{30}\cap\mathcal{R}=\varnothing
\end{array}$ \\
\hline
$\mathfrak{N}_{44}\not\to \mathfrak{N}_{29}$ & $\begin{array}{@{}l@{}}
\mathcal{R}=\big\{A\mid c_{33}^{4}=c_{33}^{5}=c_{34}^{5}=c_{44}^{5}=0\big\}\\
\mathfrak{N}_{44}\in\mathcal{R} \mbox{ (take $f_1=e_{1}$, $f_2=e_{3}$, $f_3=e_{2}$, $f_4=e_{4}$, $f_5=e_{5}$)},\\
\mbox{but } \mathfrak{N}_{29}\cap\mathcal{R}=\varnothing
\end{array}$ \\
\hline
$\mathfrak{N}_{45}\not\to \left\{\begin{array}{l}\mathfrak{N}_{09}, \mathfrak{N}_{15}\end{array}\right\}$ & $\begin{array}{@{}l@{}}
\mathcal{R}=\big\{A\mid c_{11}^{2}=c_{11}^{3}=c_{12}^{3}=c_{22}^{3}=0\big\}\\
\mathfrak{N}_{45}\in\mathcal{R} \mbox{ (take $f_1=e_{1}$, $f_2=e_{3}$, $f_3=e_{2}$, $f_4=e_{4}$, $f_5=e_{5}$)},\\
\mbox{but } \mathfrak{B}\cap\mathcal{R}=\varnothing \mbox{ for all $\mathfrak{B}$ in the left column}
\end{array}$ \\
\hline
$\mathfrak{N}_{46}\not\to \left\{\begin{array}{l}\mathfrak{N}_{03}, \mathfrak{N}_{31},\\ \mathfrak{N}_{58}\end{array}\right\}$ & $\begin{array}{@{}l@{}}
\mathcal{R}=\big\{A\mid c_{12}^{3}=c_{22}^{3}=c_{23}^{4}=c_{24}^{5}=c_{34}^{5}=0\big\}\\
\mathfrak{N}_{46}\in\mathcal{R} \mbox{ (take $f_1=e_{1}$, $f_2=e_{3}$, $f_3=e_{2}$, $f_4=e_{4}$, $f_5=e_{5}$)},\\
\mbox{but } \mathfrak{B}\cap\mathcal{R}=\varnothing \mbox{ for all $\mathfrak{B}$ in the left column}
\end{array}$ \\
\hline
$\mathfrak{N}_{46}\not\to \mathfrak{N}_{40}$ & $\begin{array}{@{}l@{}}
\mathcal{R}=\big\{A\mid c_{24}^{5}=c_{33}^{4}=c_{34}^{5}=0,\ c_{11}^{3}c_{33}^{5}=c_{13}^{4}c_{14}^{5}\big\}\\
\mathfrak{N}_{46}\in\mathcal{R} \mbox{ (take $f_1=e_{1}$, $f_2=e_{3}$, $f_3=e_{2}$, $f_4=e_{4}$, $f_5=e_{5}$)},\\
\mbox{but } \mathfrak{N}_{40}\cap\mathcal{R}=\varnothing
\end{array}$ \\
\hline
$\mathfrak{N}_{47}^{\alpha}\not\to \left\{\begin{array}{l}\mathfrak{N}_{05}, \mathfrak{N}_{52}\end{array}\right\}$ & $\begin{array}{@{}l@{}}
\mathcal{R}=\big\{A\mid c_{11}^{2}=c_{13}^{4}=c_{22}^{3}=c_{22}^{4}=c_{23}^{4}=0\big\}\\
\mathfrak{N}_{47}^{\alpha}\in\mathcal{R} \mbox{ for all $\alpha$},\\
\mbox{but } \mathfrak{B}\cap\mathcal{R}=\varnothing \mbox{ for all $\mathfrak{B}$ in the left column}
\end{array}$ \\
\hline
$\mathfrak{N}_{48}^{-\frac{1}{2}}\not\to \mathfrak{N}_{14}$ & $\begin{array}{@{}l@{}}
\mathcal{R}=\left\{A\ \middle|\ \begin{array}{l}c_{13}^{4}=c_{22}^{3}=c_{22}^{4}=c_{23}^{4}=c_{24}^{5}=c_{33}^{4}=0,\\
c_{33}^{5}=c_{44}^{5}=0,\ c_{11}^{4}c_{23}^{5}-2c_{12}^{4}c_{13}^{5}=0\end{array}\right\}\\
\mathfrak{N}_{48}^{-\frac{1}{2}}\in\mathcal{R},\\
\mbox{but } \mathfrak{N}_{14}\cap\mathcal{R}=\varnothing
\end{array}$ \\
\hline
$\mathfrak{N}_{48}^{-\frac{1}{2}}\not\to \mathfrak{N}_{17}$ & $\begin{array}{@{}l@{}}
\mathcal{R}=\big\{A\mid x(xx)=0 \mbox{ for all $x\in A$}\big\}\\
\mathfrak{N}_{48}^{-\frac{1}{2}}\in\mathcal{R},\\
\mbox{but } \mathfrak{N}_{17}\cap\mathcal{R}=\varnothing \mbox{ (take $x=e_1$)}
\end{array}$ \\
\hline
$\mathfrak{N}_{49}\not\to \left\{\begin{array}{l}\mathfrak{N}_{25}, \mathfrak{N}_{26}\end{array}\right\}$ & $\begin{array}{@{}l@{}}
\mathcal{R}=\left\{A\ \middle|\ \begin{array}{l}c_{11}^{2}=c_{12}^{3}=c_{13}^{4}=c_{22}^{3}=c_{22}^{4}=c_{24}^{5}=0,\\
c_{33}^{4}=c_{34}^{5}=0,\ c_{11}^{3}c_{23}^{5}+2c_{12}^{4}c_{14}^{5}=0\end{array}\right\}\\
\mathfrak{N}_{49}\in\mathcal{R},\\
\mbox{but } \mathfrak{B}\cap\mathcal{R}=\varnothing \mbox{ for all $\mathfrak{B}$ in the left column}
\end{array}$ \\
\hline
$\mathfrak{N}_{50}^{\alpha}\not\to \mathfrak{N}_{58}$ & $\begin{array}{@{}l@{}}
\mathcal{R}=\big\{A\mid c_{11}^{2}=c_{24}^{5}=0,\ c_{11}^{3}c_{22}^{3}=\big(c_{12}^{3}\big)^{2},\ c_{12}^{3}c_{22}^{4}=c_{12}^{4}c_{22}^{3}\big\}\\
\mathfrak{N}_{50}^{\alpha}\in\mathcal{R} \mbox{ for all $\alpha$} \mbox{ (take $f_1=e_{1}$, $f_2=e_{2}$, $f_3=e_{4}$, $f_4=e_{3}$, $f_5=e_{5}$)},\\
\mbox{but } \mathfrak{N}_{58}\cap\mathcal{R}=\varnothing
\end{array}$ \\
\hline
$\mathfrak{N}_{51}^{\alpha,\beta}\not\to \left\{\begin{array}{l}\mathfrak{N}_{04}, \mathfrak{N}_{06},\\ \mathfrak{N}_{29}, \mathfrak{N}_{35},\\ \mathfrak{N}_{54}\end{array}\right\}$ & $\begin{array}{@{}l@{}}
\mathcal{R}=\big\{A\mid c_{11}^{2}=c_{13}^{4}=c_{23}^{4}=c_{33}^{4}=c_{33}^{5}=c_{34}^{5}=c_{44}^{5}=0\big\}\\
\mathfrak{N}_{51}^{\alpha,\beta}\in\mathcal{R} \mbox{ for all $\alpha,\beta$},\\
\mbox{but } \mathfrak{B}\cap\mathcal{R}=\varnothing \mbox{ for all $\mathfrak{B}$ in the left column}
\end{array}$ \\
\hline
$\mathfrak{N}_{56}\not\to \mathfrak{N}_{26}$ & $\begin{array}{@{}l@{}}
\mathcal{R}=\big\{A\mid c_{12}^{3}=c_{14}^{5}=c_{22}^{3}=c_{24}^{5}=0\big\}\\
\mathfrak{N}_{56}\in\mathcal{R} \mbox{ (take $f_1=e_{1}$, $f_2=e_{3}$, $f_3=e_{2}$, $f_4=e_{5}$, $f_5=e_{4}$)},\\
\mbox{but } \mathfrak{N}_{26}\cap\mathcal{R}=\varnothing
\end{array}$ \\
\hline
$\mathfrak{N}_{57}\not\to \left\{\begin{array}{l}\mathfrak{N}_{18}, \mathfrak{N}_{21}\end{array}\right\}$ & $\begin{array}{@{}l@{}}
\mathcal{R}=\left\{A\ \middle|\ \begin{array}{l}c_{22}^{3}=c_{22}^{4}=c_{22}^{5}=c_{23}^{4}=c_{23}^{5}=0,\\
c_{24}^{5}=c_{33}^{4}=c_{33}^{5}=c_{34}^{5}=c_{44}^{5}=0\end{array}\right\}\\
\mathfrak{N}_{57}\in\mathcal{R} \mbox{ (take $f_1=e_{1}$, $f_2=e_{2}$, $f_3=e_{3}$, $f_4=e_{5}$, $f_5=e_{4}$)},\\
\mbox{but } \mathfrak{B}\cap\mathcal{R}=\varnothing \mbox{ for all $\mathfrak{B}$ in the left column}
\end{array}$ \\
\hline
$\mathfrak{N}_{58}\not\to \mathfrak{N}_{25}$ & $\begin{array}{@{}l@{}}
\mathcal{R}=\big\{A\mid c_{14}^{5}=c_{22}^{3}=c_{23}^{4}=c_{23}^{5}=c_{24}^{5}=c_{34}^{5}=0\big\}\\
\mathfrak{N}_{58}\in\mathcal{R},\\
\mbox{but } \mathfrak{N}_{25}\cap\mathcal{R}=\varnothing
\end{array}$ \\
\hline
$\mathfrak{N}_{59}\not\to \left\{\begin{array}{l}\mathfrak{N}_{05}, \mathfrak{N}_{57}\end{array}\right\}$ & $\begin{array}{@{}l@{}}
\mathcal{R}=\left\{A\ \middle|\ \begin{array}{l}c_{11}^{2}=c_{14}^{5}=c_{22}^{3}=c_{22}^{4}=c_{23}^{4}=0,\\
2(c_{11}^{3})^{2}c_{12}^{4}c_{23}^{5}-(c_{11}^{3})^{2}c_{13}^{4}c_{22}^{5}-2c_{11}^{3}c_{11}^{4}c_{12}^{3}c_{23}^{5}-4c_{11}^{3}c_{12}^{3}c_{12}^{4}c_{13}^{5}\\
\quad{}+4c_{11}^{3}c_{12}^{3}c_{12}^{5}c_{13}^{4}+4c_{11}^{4}(c_{12}^{3})^{2}c_{13}^{5}-4c_{11}^{5}(c_{12}^{3})^{2}c_{13}^{4}=0\end{array}\right\}\\
\mathfrak{N}_{59}\in\mathcal{R},\\
\mbox{but } \mathfrak{B}\cap\mathcal{R}=\varnothing \mbox{ for all $\mathfrak{B}$ in the left column}
\end{array}$ \\
\hline
$\mathfrak{N}_{60}\not\to \mathfrak{N}_{52}$ & $\begin{array}{@{}l@{}}
\mathcal{R}=\big\{A\mid c_{22}^{3}=c_{22}^{4}=c_{23}^{4}=c_{33}^{4}=0\big\}\\
\mathfrak{N}_{60}\in\mathcal{R},\\
\mbox{but } \mathfrak{N}_{52}\cap\mathcal{R}=\varnothing
\end{array}$ \\
\hline
$\mathfrak{N}_{61}\not\to \mathfrak{N}_{14}$ & $\begin{array}{@{}l@{}}
\mathcal{R}=\left\{A\ \middle|\ \begin{array}{l}c_{13}^{4}=c_{22}^{3}=c_{22}^{4}=c_{23}^{4}=0,\\
c_{23}^{5}=c_{24}^{5}=c_{33}^{4}=c_{33}^{5}=c_{44}^{5}=0\end{array}\right\}\\
\mathfrak{N}_{61}\in\mathcal{R} \mbox{ (take $f_1=e_{1}$, $f_2=e_{2}$, $f_3=e_{3}$, $f_4=e_{5}$, $f_5=e_{4}$)},\\
\mbox{but } \mathfrak{N}_{14}\cap\mathcal{R}=\varnothing
\end{array}$ \\
\hline
$\mathfrak{N}_{62}\not\to \left\{\begin{array}{l}\mathfrak{N}_{04}, \mathfrak{N}_{20},\\ \mathfrak{N}_{35}\end{array}\right\}$ & $\begin{array}{@{}l@{}}
\mathcal{R}=\big\{A\mid c_{11}^{2}=c_{14}^{5}=c_{24}^{5}=c_{34}^{5}=c_{44}^{5}=0\big\}\\
\mathfrak{N}_{62}\in\mathcal{R},\\
\mbox{but } \mathfrak{B}\cap\mathcal{R}=\varnothing \mbox{ for all $\mathfrak{B}$ in the left column}
\end{array}$ \\
\hline
\end{longtable}
\endgroup

\noindent This completes the proof of Theorem~B.
\end{proof}

\enlargethispage{\baselineskip}
{\footnotesize\noindent{\bf The graph of primary degenerations.} The graph of primary degenerations of $5$-dimensional nilpotent Jordan algebras is given on the next page (the numbers on the left are the dimensions of the orbit closures; the rigid algebras and the family $\mathfrak{N}_{51}^{\alpha,\beta}$ giving the irreducible components are shown in dark boxes; the dashed boxes are the members $\mathfrak{N}_{48}^{1}$ and $\mathfrak{N}_{48}^{-\frac12}$ of the family $\mathfrak{N}_{48}^{\alpha}$; the colour of a solid arrow is the colour of the level of its initial vertex; for a family source, a label $\gamma\mapsto\delta$ (in a red frame, joined to its arrow by a red line) means that the member with parameter $\gamma$ degenerates to the member with parameter $\delta$ of the target, and an unlabelled arrow to a family means a degeneration to every member of the family; magenta dashed arrows are degenerations which exist only for the closure of the union of the orbits of the family; thick blue dashed arrows point from the family $\mathfrak{N}_{48}^{\alpha}$ to its members $\mathfrak{N}_{48}^{1}$ and $\mathfrak{N}_{48}^{-\frac12}$).\par}

\clearpage
\newgeometry{left=0.8cm,right=0.8cm,top=0.8cm,bottom=0.8cm}
\thispagestyle{empty}
{\scriptsize
\begin{center}
\begin{adjustbox}{max width=\textwidth, max totalheight=0.98\textheight}
\begin{tikzpicture}[x=1cm,y=1.26cm,vx/.style={draw,rounded corners=3pt,inner sep=2.2pt,minimum height=13pt,font=\small,fill=white},sh/.style={vx,fill=gray!35},sp/.style={vx,dashed},lab/.style={font=\tiny,inner sep=0.6pt,fill=white,sloped},elab/.style={font=\scriptsize,text=black,inner sep=0.8pt,fill=white,sloped},plab/.style={font=\footnotesize,text=red!65!black,draw=red!65!black,rounded corners=2pt,inner sep=1.5pt,fill=white,line width=0.4pt},plabline/.style={red!65!black,line width=0.5pt}]
\node[font=\small,text=red!75!black] at (-1.50,0.00) {22};
\node[font=\small,text=blue!85!black] at (-1.50,-1.23) {21};
\node[font=\small,text=green!50!black] at (-1.50,-3.67) {20};
\node[font=\small,text=orange!85!black] at (-1.50,-7.87) {19};
\node[font=\small,text=violet] at (-1.50,-12.73) {18};
\node[font=\small,text=brown] at (-1.50,-17.37) {17};
\node[font=\small,text=teal] at (-1.50,-20.36) {16};
\node[font=\small,text=purple!70!black] at (-1.50,-22.69) {15};
\node[font=\small,text=olive] at (-1.50,-24.80) {14};
\node[font=\small,text=pink!60!black] at (-1.50,-26.25) {13};
\node[font=\small,text=cyan!60!black] at (-1.50,-27.37) {11};
\node[font=\small,text=gray!80!black] at (-1.50,-28.38) {8};
\node[font=\small,text=black] at (-1.50,-29.39) {0};
\node[sh] (J01) at (1.30,-3.67) {$\mathfrak{N}_{01}$};
\node[vx] (J02) at (1.08,-7.87) {$\mathfrak{N}_{02}$};
\node[vx] (J03) at (3.25,-7.87) {$\mathfrak{N}_{03}$};
\node[sp] (J04) at (7.00,-12.73) {$\mathfrak{N}_{48}^{1}$};
\node[vx] (J05) at (1.00,-12.73) {$\mathfrak{N}_{04}$};
\node[vx] (J06) at (3.00,-12.73) {$\mathfrak{N}_{05}$};
\node[vx] (J07) at (5.00,-12.73) {$\mathfrak{N}_{06}$};
\node[vx] (J08) at (6.60,-17.37) {$\mathfrak{N}_{07}$};
\node[vx] (J09) at (9.80,-17.37) {$\mathfrak{N}_{08}$};
\node[vx] (J10) at (6.60,-20.36) {$\mathfrak{N}_{09}$};
\node[vx] (J11) at (13.00,-20.36) {$\mathfrak{N}_{10}$};
\node[vx] (J12) at (8.20,-24.80) {$\mathfrak{N}_{11}$};
\node[vx] (J13) at (22.60,-17.37) {$\mathfrak{N}_{12}$};
\node[vx] (J14) at (16.20,-20.36) {$\mathfrak{N}_{13}$};
\node[vx] (J15) at (19.40,-20.36) {$\mathfrak{N}_{14}$};
\node[vx] (J16) at (11.40,-22.69) {$\mathfrak{N}_{15}$};
\node[vx] (J17) at (14.60,-22.69) {$\mathfrak{N}_{16}$};
\node[vx] (J18) at (11.40,-24.80) {$\mathfrak{N}_{17}$};
\node[vx] (J19) at (14.60,-24.80) {$\mathfrak{N}_{18}$};
\node[vx] (J20) at (11.40,-26.25) {$\mathfrak{N}_{19}$};
\node[vx] (J21) at (17.80,-24.80) {$\mathfrak{N}_{20}$};
\node[vx] (J22) at (14.60,-26.25) {$\mathfrak{N}_{21}$};
\node[vx] (J23) at (13.00,-27.37) {$\mathfrak{N}_{22}$};
\node[vx] (J24) at (13.00,-28.38) {$\mathfrak{N}_{23}$};
\node[vx] (J25) at (13.00,-29.39) {$\mathfrak{N}_{24}$};
\node[vx] (J26) at (19.40,-17.37) {$\mathfrak{N}_{25}$};
\node[vx] (J27) at (9.80,-20.36) {$\mathfrak{N}_{26}$};
\node[vx] (J28) at (16.20,-17.37) {$\mathfrak{N}_{27}$};
\node[vx] (J29) at (23.00,-12.73) {$\mathfrak{N}_{28}$};
\node[vx] (J30) at (22.75,-7.87) {$\mathfrak{N}_{29}$};
\node[vx] (J31) at (25.00,-12.73) {$\mathfrak{N}_{30}$};
\node[vx] (J32) at (9.75,-7.87) {$\mathfrak{N}_{31}$};
\node[vx] (J33) at (6.50,-3.67) {$\mathfrak{N}_{32}$};
\node[vx] (J34) at (8.20,-1.23) {$\mathfrak{N}_{33}$};
\node[vx] (J35) at (14.30,-3.67) {$\mathfrak{N}_{34}$};
\node[vx] (J36) at (9.00,-12.73) {$\mathfrak{N}_{35}$};
\node[vx] (J37) at (5.42,-7.87) {$\mathfrak{N}_{36}$};
\node[vx] (J38) at (20.58,-7.87) {$\mathfrak{N}_{37}$};
\node[vx] (J39) at (21.00,-12.73) {$\mathfrak{N}_{38}$};
\node[vx] (J41) at (11.92,-7.87) {$\mathfrak{N}_{40}$};
\node[vx] (J42) at (11.70,-3.67) {$\mathfrak{N}_{41}$};
\node[vx] (J43) at (11.40,-1.23) {$\mathfrak{N}_{42}$};
\node[vx] (J44) at (3.90,-3.67) {$\mathfrak{N}_{43}$};
\node[vx] (J45) at (5.00,-1.23) {$\mathfrak{N}_{44}$};
\node[sh] (J46) at (13.00,0.00) {$\mathfrak{N}_{45}$};
\node[sh] (J47) at (14.60,-1.23) {$\mathfrak{N}_{46}$};
\node[vx] (J50) at (17.00,-12.73) {$\mathfrak{N}_{49}$};
\node[vx] (J53) at (13.00,-12.73) {$\mathfrak{N}_{52}$};
\node[vx] (J54) at (14.08,-7.87) {$\mathfrak{N}_{53}$};
\node[vx] (J55) at (11.00,-12.73) {$\mathfrak{N}_{54}$};
\node[vx] (J56) at (7.58,-7.87) {$\mathfrak{N}_{55}$};
\node[vx] (J57) at (16.90,-3.67) {$\mathfrak{N}_{56}$};
\node[vx] (J58) at (3.40,-17.37) {$\mathfrak{N}_{57}$};
\node[vx] (J59) at (19.00,-12.73) {$\mathfrak{N}_{58}$};
\node[vx] (J60) at (16.25,-7.87) {$\mathfrak{N}_{59}$};
\node[vx] (J61) at (22.10,-3.67) {$\mathfrak{N}_{60}$};
\node[vx] (J62) at (15.00,-12.73) {$\mathfrak{N}_{61}$};
\node[sh] (J63) at (17.80,-1.23) {$\mathfrak{N}_{62}$};
\node[vx] (J64) at (24.92,-7.87) {$\mathfrak{N}_{63}$};
\node[sp] (J49s) at (13.00,-17.37) {$\mathfrak{N}_{48}^{-\frac12}$};
\node[vx] (J40) at (9.10,-3.67) {$\mathfrak{N}_{39}^{\alpha}$};
\node[vx] (J48) at (24.70,-3.67) {$\mathfrak{N}_{47}^{\alpha}$};
\node[vx] (J49) at (18.42,-7.87) {$\mathfrak{N}_{48}^{\alpha}$};
\node[vx] (J51) at (19.50,-3.67) {$\mathfrak{N}_{50}^{\alpha}$};
\node[sh] (J52) at (21.00,-1.23) {$\mathfrak{N}_{51}^{\alpha,\beta}$};
\begin{scope}[on background layer]
\draw[->,>={Stealth[length=4pt]},semithick,draw=green!50!black] (J01) -- (J02);
\draw[->,>={Stealth[length=4pt]},semithick,draw=green!50!black] (J01) -- (J03);
\draw[->,>={Stealth[length=4pt]},semithick,draw=orange!85!black] (J02) -- (J04);
\draw[->,>={Stealth[length=4pt]},semithick,draw=orange!85!black] (J02) -- (J05);
\draw[->,>={Stealth[length=4pt]},semithick,draw=orange!85!black] (J02) -- (J06);
\draw[->,>={Stealth[length=4pt]},semithick,draw=orange!85!black] (J02) -- (J07);
\draw[->,>={Stealth[length=4pt]},semithick,draw=orange!85!black] (J03) -- (J04);
\draw[->,>={Stealth[length=4pt]},semithick,draw=orange!85!black] (J03) -- (J06);
\draw[->,>={Stealth[length=4pt]},semithick,draw=violet] (J04) -- (J08);
\draw[->,>={Stealth[length=4pt]},semithick,draw=violet] (J04) -- (J09);
\draw[->,>={Stealth[length=4pt]},semithick,draw=violet] (J04) -- (J13);
\draw[->,>={Stealth[length=4pt]},semithick,draw=violet] (J05) -- (J08);
\draw[->,>={Stealth[length=4pt]},semithick,draw=violet] (J06) -- (J08);
\draw[->,>={Stealth[length=4pt]},semithick,draw=violet] (J06) -- (J09);
\draw[->,>={Stealth[length=4pt]},semithick,draw=violet] (J07) -- (J09);
\draw[->,>={Stealth[length=4pt]},semithick,draw=brown] (J08) -- (J10);
\draw[->,>={Stealth[length=4pt]},semithick,draw=brown] (J08) -- (J11);
\draw[->,>={Stealth[length=4pt]},semithick,draw=brown] (J08) -- (J14);
\draw[->,>={Stealth[length=4pt]},semithick,draw=brown] (J08) -- (J16);
\draw[->,>={Stealth[length=4pt]},semithick,draw=brown] (J09) -- (J11);
\draw[->,>={Stealth[length=4pt]},semithick,draw=brown] (J09) -- (J15);
\draw[->,>={Stealth[length=4pt]},semithick,draw=teal] (J10) -- (J12);
\draw[->,>={Stealth[length=4pt]},semithick,draw=teal] (J10) -- (J18);
\draw[->,>={Stealth[length=4pt]},semithick,draw=teal] (J11) -- (J12);
\draw[->,>={Stealth[length=4pt]},semithick,draw=teal] (J11) -- (J17);
\draw[->,>={Stealth[length=4pt]},semithick,draw=olive] (J12) -- (J20);
\draw[->,>={Stealth[length=4pt]},semithick,draw=brown] (J13) -- (J14);
\draw[->,>={Stealth[length=4pt]},semithick,draw=brown] (J13) -- (J15);
\draw[->,>={Stealth[length=4pt]},semithick,draw=brown] (J13) -- (J21);
\draw[->,>={Stealth[length=4pt]},semithick,draw=teal] (J14) -- (J17);
\draw[->,>={Stealth[length=4pt]},semithick,draw=teal] (J14) -- (J18);
\draw[->,>={Stealth[length=4pt]},semithick,draw=teal] (J15) -- (J17);
\draw[->,>={Stealth[length=4pt]},semithick,draw=purple!70!black] (J16) -- (J19);
\draw[->,>={Stealth[length=4pt]},semithick,draw=purple!70!black] (J17) -- (J19);
\draw[->,>={Stealth[length=4pt]},semithick,draw=purple!70!black] (J17) -- (J22);
\draw[->,>={Stealth[length=4pt]},semithick,draw=olive] (J18) -- (J20);
\draw[->,>={Stealth[length=4pt]},semithick,draw=olive] (J19) -- (J20);
\draw[->,>={Stealth[length=4pt]},semithick,draw=pink!60!black] (J20) -- (J23);
\draw[->,>={Stealth[length=4pt]},semithick,draw=olive] (J21) -- (J22);
\draw[->,>={Stealth[length=4pt]},semithick,draw=pink!60!black] (J22) -- (J23);
\draw[->,>={Stealth[length=4pt]},semithick,draw=cyan!60!black] (J23) -- (J24);
\draw[->,>={Stealth[length=4pt]},semithick,draw=gray!80!black] (J24) -- (J25);
\draw[->,>={Stealth[length=4pt]},semithick,draw=brown] (J26) -- (J14);
\draw[->,>={Stealth[length=4pt]},semithick,draw=brown] (J26) -- (J15);
\draw[->,>={Stealth[length=4pt]},semithick,draw=teal] (J27) -- (J12);
\draw[->,>={Stealth[length=4pt]},semithick,draw=teal] (J27) -- (J18);
\draw[->,>={Stealth[length=4pt]},semithick,draw=brown] (J28) -- (J11);
\draw[->,>={Stealth[length=4pt]},semithick,draw=brown] (J28) -- (J14);
\draw[->,>={Stealth[length=4pt]},semithick,draw=brown] (J28) -- (J27);
\draw[->,>={Stealth[length=4pt]},semithick,draw=violet] (J29) -- (J09);
\draw[->,>={Stealth[length=4pt]},semithick,draw=violet] (J29) -- (J26);
\draw[->,>={Stealth[length=4pt]},semithick,draw=violet] (J29) -- (J28);
\draw[->,>={Stealth[length=4pt]},semithick,draw=orange!85!black] (J30) -- (J31);
\draw[->,>={Stealth[length=4pt]},semithick,draw=violet] (J31) -- (J13);
\draw[->,>={Stealth[length=4pt]},semithick,draw=violet] (J31) -- (J26);
\draw[->,>={Stealth[length=4pt]},semithick,draw=orange!85!black] (J32) -- (J07);
\draw[->,>={Stealth[length=4pt]},semithick,draw=orange!85!black] (J32) -- (J29);
\draw[->,>={Stealth[length=4pt]},semithick,draw=orange!85!black] (J32) -- (J39);
\draw[->,>={Stealth[length=4pt]},semithick,draw=green!50!black] (J33) -- (J32);
\draw[->,>={Stealth[length=4pt]},semithick,draw=green!50!black] (J33) -- (J37);
\draw[->,>={Stealth[length=4pt]},semithick,draw=green!50!black] (J33) -- (J38);
\draw[->,>={Stealth[length=4pt]},semithick,draw=blue!85!black] (J34) -- (J33);
\draw[->,>={Stealth[length=4pt]},semithick,draw=blue!85!black] (J34) -- (J35);
\draw[->,>={Stealth[length=4pt]},semithick,draw=green!50!black] (J35) -- (J30);
\draw[->,>={Stealth[length=4pt]},semithick,draw=green!50!black] (J35) -- (J37);
\draw[->,>={Stealth[length=4pt]},semithick,draw=green!50!black] (J35) -- (J38);
\draw[->,>={Stealth[length=4pt]},semithick,draw=violet] (J36) -- (J09);
\draw[->,>={Stealth[length=4pt]},semithick,draw=violet] (J36) -- (J28);
\draw[->,>={Stealth[length=4pt]},semithick,draw=orange!85!black] (J37) -- (J07);
\draw[->,>={Stealth[length=4pt]},semithick,draw=orange!85!black] (J37) -- (J36);
\draw[->,>={Stealth[length=4pt]},semithick,draw=orange!85!black] (J37) -- (J39);
\draw[->,>={Stealth[length=4pt]},semithick,draw=orange!85!black] (J38) -- (J29);
\draw[->,>={Stealth[length=4pt]},semithick,draw=orange!85!black] (J38) -- (J31);
\draw[->,>={Stealth[length=4pt]},semithick,draw=orange!85!black] (J38) -- (J39);
\draw[->,>={Stealth[length=4pt]},semithick,draw=violet] (J39) -- (J09);
\draw[->,>={Stealth[length=4pt]},semithick,draw=violet] (J39) -- (J13);
\draw[->,>={Stealth[length=4pt]},semithick,draw=violet] (J39) -- (J28);
\draw[->,>={Stealth[length=4pt]},semithick,draw=magenta,dashed] (J40) -- (J02);
\draw[->,>={Stealth[length=4pt]},semithick,draw=magenta,dashed] (J40) -- (J37);
\draw[->,>={Stealth[length=4pt]},semithick,draw=green!50!black] (J40) -- node[pos=0.55,coordinate] (m40) {} (J49);
\node[plab] (L40) at (13.63,-6.69) {$1-p^{-2}\mapsto 1+p$};
\draw[plabline] (L40) -- (m40);
\draw[->,>={Stealth[length=4pt]},semithick,draw=magenta,dashed] (J40) -- (J56);
\draw[->,>={Stealth[length=4pt]},semithick,draw=orange!85!black] (J41) -- (J29);
\draw[->,>={Stealth[length=4pt]},semithick,draw=orange!85!black] (J41) -- (J36);
\draw[->,>={Stealth[length=4pt]},semithick,draw=green!50!black] (J42) -- (J37);
\draw[->,>={Stealth[length=4pt]},semithick,draw=green!50!black] (J42) -- (J38);
\draw[->,>={Stealth[length=4pt]},semithick,draw=green!50!black] (J42) -- (J41);
\draw[->,>={Stealth[length=4pt]},semithick,draw=blue!85!black] (J43) -- (J35);
\draw[->,>={Stealth[length=4pt]},semithick,draw=blue!85!black] (J43) -- (J42);
\draw[->,>={Stealth[length=4pt]},semithick,draw=green!50!black] (J44) -- (J32);
\draw[->,>={Stealth[length=4pt]},semithick,draw=green!50!black] (J44) -- (J37);
\draw[->,>={Stealth[length=4pt]},semithick,draw=green!50!black] (J44) -- (J41);
\draw[->,>={Stealth[length=4pt]},semithick,draw=blue!85!black] (J45) -- (J33);
\draw[->,>={Stealth[length=4pt]},semithick,draw=blue!85!black] (J45) -- (J42);
\draw[->,>={Stealth[length=4pt]},semithick,draw=blue!85!black] (J45) -- (J44);
\draw[->,>={Stealth[length=4pt]},semithick,draw=red!75!black] (J46) -- (J34);
\draw[->,>={Stealth[length=4pt]},semithick,draw=red!75!black] (J46) -- (J43);
\draw[->,>={Stealth[length=4pt]},semithick,draw=red!75!black] (J46) -- (J45);
\draw[->,>={Stealth[length=4pt]},semithick,draw=blue!85!black] (J47) -- (J35);
\draw[->,>={Stealth[length=4pt]},semithick,draw=blue!85!black] (J47) -- (J40);
\draw[->,>={Stealth[length=4pt]},semithick,draw=blue!85!black] (J47) -- (J57);
\draw[->,>={Stealth[length=4pt]},semithick,draw=magenta,dashed] (J48) -- (J38);
\draw[->,>={Stealth[length=4pt]},semithick,draw=green!50!black] (J48) -- node[pos=0.55,coordinate] (m48) {} (J49);
\node[plab] (L48) at (22.95,-7.09) {$-2p-p^{-1}\mapsto -p^2$};
\draw[plabline] (L48) -- (m48);
\draw[->,>={Stealth[length=4pt]},semithick,draw=magenta,dashed] (J48) -- (J64);
\draw[->,>={Stealth[length=5pt]},line width=1.3pt,draw=blue!65!cyan,dash pattern=on 5pt off 2pt] (J49) -- (J04);
\draw[->,>={Stealth[length=4pt]},semithick,draw=magenta,dashed] (J49) -- (J31);
\draw[->,>={Stealth[length=4pt]},semithick,draw=magenta,dashed] (J49) -- (J39);
\draw[->,>={Stealth[length=4pt]},semithick,draw=magenta,dashed] (J49) -- (J50);
\draw[->,>={Stealth[length=4pt]},semithick,draw=magenta,dashed] (J49) -- (J62);
\draw[->,>={Stealth[length=4pt]},semithick,draw=brown] (J49s) -- (J11);
\draw[->,>={Stealth[length=4pt]},semithick,draw=brown] (J49s) -- (J16);
\draw[->,>={Stealth[length=4pt]},semithick,draw=brown] (J49s) -- (J21);
\draw[->,>={Stealth[length=4pt]},semithick,draw=violet] (J50) -- (J08);
\draw[->,>={Stealth[length=4pt]},semithick,draw=violet] (J50) -- (J09);
\draw[->,>={Stealth[length=4pt]},semithick,draw=violet] (J50) -- (J13);
\draw[->,>={Stealth[length=4pt]},semithick,draw=violet] (J50) -- (J49s);
\draw[->,>={Stealth[length=4pt]},semithick,draw=magenta,dashed] (J51) -- (J03);
\draw[->,>={Stealth[length=4pt]},semithick,draw=magenta,dashed] (J51) -- (J38);
\draw[->,>={Stealth[length=4pt]},semithick,draw=green!50!black] (J51) -- node[pos=0.15,coordinate] (m51) {} (J49);
\node[plab] (L51) at (18.54,-3.03) {$-p^3-p\mapsto \frac{p^2}{1+p^2}$};
\draw[plabline] (L51) -- (m51);
\draw[->,>={Stealth[length=4pt]},semithick,draw=magenta,dashed] (J51) -- (J54);
\draw[->,>={Stealth[length=4pt]},semithick,draw=magenta,dashed] (J52) -- (J48);
\draw[->,>={Stealth[length=4pt]},semithick,draw=magenta,dashed] (J52) -- (J51);
\draw[->,>={Stealth[length=4pt]},semithick,draw=violet] (J53) -- (J08);
\draw[->,>={Stealth[length=4pt]},semithick,draw=violet] (J53) -- (J09);
\draw[->,>={Stealth[length=4pt]},semithick,draw=violet] (J53) -- (J26);
\draw[->,>={Stealth[length=4pt]},semithick,draw=orange!85!black] (J54) -- (J06);
\draw[->,>={Stealth[length=4pt]},semithick,draw=orange!85!black] (J54) -- (J53);
\draw[->,>={Stealth[length=4pt]},semithick,draw=violet] (J55) -- (J08);
\draw[->,>={Stealth[length=4pt]},semithick,draw=violet] (J55) -- (J09);
\draw[->,>={Stealth[length=4pt]},semithick,draw=violet] (J55) -- (J26);
\draw[->,>={Stealth[length=4pt]},semithick,draw=orange!85!black] (J56) -- (J06);
\draw[->,>={Stealth[length=4pt]},semithick,draw=orange!85!black] (J56) -- (J07);
\draw[->,>={Stealth[length=4pt]},semithick,draw=orange!85!black] (J56) -- (J55);
\draw[->,>={Stealth[length=4pt]},semithick,draw=green!50!black] (J57) -- (J54);
\draw[->,>={Stealth[length=4pt]},semithick,draw=green!50!black] (J57) -- (J56);
\draw[->,>={Stealth[length=4pt]},semithick,draw=brown] (J58) -- (J10);
\draw[->,>={Stealth[length=4pt]},semithick,draw=brown] (J58) -- (J27);
\draw[->,>={Stealth[length=4pt]},semithick,draw=violet] (J59) -- (J08);
\draw[->,>={Stealth[length=4pt]},semithick,draw=violet] (J59) -- (J09);
\draw[->,>={Stealth[length=4pt]},semithick,draw=violet] (J59) -- (J28);
\draw[->,>={Stealth[length=4pt]},semithick,draw=orange!85!black] (J60) -- (J07);
\draw[->,>={Stealth[length=4pt]},semithick,draw=orange!85!black] (J60) -- (J29);
\draw[->,>={Stealth[length=4pt]},semithick,draw=orange!85!black] (J60) -- (J55);
\draw[->,>={Stealth[length=4pt]},semithick,draw=orange!85!black] (J60) -- (J59);
\draw[->,>={Stealth[length=4pt]},semithick,draw=green!50!black] (J61) -- (J56);
\draw[->,>={Stealth[length=4pt]},semithick,draw=green!50!black] (J61) -- (J60);
\draw[->,>={Stealth[length=4pt]},semithick,draw=green!50!black] (J61) -- (J64);
\draw[->,>={Stealth[length=4pt]},semithick,draw=violet] (J62) -- (J08);
\draw[->,>={Stealth[length=4pt]},semithick,draw=violet] (J62) -- (J28);
\draw[->,>={Stealth[length=4pt]},semithick,draw=violet] (J62) -- (J58);
\draw[->,>={Stealth[length=4pt]},semithick,draw=blue!85!black] (J63) -- (J57);
\draw[->,>={Stealth[length=4pt]},semithick,draw=blue!85!black] (J63) -- (J61);
\draw[->,>={Stealth[length=4pt]},semithick,draw=orange!85!black] (J64) -- (J29);
\draw[->,>={Stealth[length=4pt]},semithick,draw=orange!85!black] (J64) -- (J59);
\draw[->,>={Stealth[length=4pt]},semithick,draw=orange!85!black] (J64) -- (J62);
\draw[->,>={Stealth[length=5pt]},line width=1.3pt,draw=blue!65!cyan,dash pattern=on 5pt off 2pt] (J49) -- (J49s);
\end{scope}
\end{tikzpicture}
\end{adjustbox}
\end{center}
}
\clearpage
\restoregeometry
\pdfpagewidth=210mm \pdfpageheight=297mm \setlength{\headwidth}{\textwidth}

\begin{proposition}\label{prop:fam}
\begin{enumerate}
\item The following degenerations hold only for the closure of the union of the orbits of the
family; they are realized by members depending on $t$ (Table~5), and no single member degenerates to the target:
\[
\begin{gathered}
\mathfrak{N}_{39}^{t^{-2}}\to\mathfrak{N}_{02},\quad \mathfrak{N}_{39}^{t^{2}+1}\to\mathfrak{N}_{36},\quad \mathfrak{N}_{39}^{-2t}\to\mathfrak{N}_{55};\qquad
\mathfrak{N}_{47}^{-2t^{-1}}\to\mathfrak{N}_{37},\quad \mathfrak{N}_{47}^{2t^{-1}}\to\mathfrak{N}_{63};\\
\mathfrak{N}_{48}^{2t}\to\mathfrak{N}_{30},\quad \mathfrak{N}_{48}^{t^{-1}}\to\mathfrak{N}_{38},\quad \mathfrak{N}_{48}^{t-\frac12}\to\mathfrak{N}_{49},\quad \mathfrak{N}_{48}^{2t^{-1}}\to\mathfrak{N}_{61};\\
\mathfrak{N}_{50}^{t^{-3}}\to\mathfrak{N}_{03},\quad \mathfrak{N}_{50}^{-\frac14{\bf i}t}\to\mathfrak{N}_{37},\quad \mathfrak{N}_{50}^{-t^{3}-t^{2}-t}\to\mathfrak{N}_{53};\\
\mathfrak{N}_{51}^{\alpha t-1,\,b(t)}\to\mathfrak{N}_{47}^{\alpha},\quad \mathfrak{N}_{51}^{\alpha t,\,t^{-2}}\to\mathfrak{N}_{50}^{\alpha},
\end{gathered}
\]
where $b(t)=-\alpha^{2}t^{2}(\alpha t+1)^{2}-\alpha t-t^{2}-1$.
\item The degenerations between families are realized by single members:
$\mathfrak{N}_{39}^{1-p^{-2}}\to\mathfrak{N}_{48}^{1+p}$, $\mathfrak{N}_{47}^{-2p-p^{-1}}\to\mathfrak{N}_{48}^{-p^2}$,
$\mathfrak{N}_{50}^{-p^3-p}\to\mathfrak{N}_{48}^{\frac{p^2}{1+p^2}}$ (for all admissible $p$), and
$\mathfrak{N}_{46}\to\mathfrak{N}_{39}^{\alpha}$ for every $\alpha$. Since the table of $\mathfrak{N}_{48}^{\alpha}$ is polynomial in
$\alpha$, for $\mu^{\beta}=\mathfrak{N}_{39}^{\beta},\mathfrak{N}_{47}^{\beta},\mathfrak{N}_{50}^{\beta}$ the set of $\alpha$ with $\mathfrak{N}_{48}^{\alpha}\in\overline{\{{\mathcal O}(\mu^{\beta})\}}$ is closed and
infinite, hence it is everything; that is, every member $\mathfrak{N}_{48}^{\alpha}$ lies in the closures of the families $\mathfrak{N}_{39}^{\beta}$, $\mathfrak{N}_{47}^{\beta}$, $\mathfrak{N}_{50}^{\beta}$.
\end{enumerate}
\end{proposition}

\begin{proof}
(1) The degenerations are given in Table~5 by a member depending on $t$. In each case
the target has an orbit of dimension at least the dimension of the generic orbit of the family
($19$ for $\mathfrak{N}_{02},\mathfrak{N}_{36},\mathfrak{N}_{55},\mathfrak{N}_{37},\mathfrak{N}_{63},\mathfrak{N}_{03},\mathfrak{N}_{53}$ and for the generic members of
$\mathfrak{N}_{47}^{\alpha}$, $\mathfrak{N}_{50}^{\alpha}$; $18$ for $\mathfrak{N}_{30},\mathfrak{N}_{38},\mathfrak{N}_{49},\mathfrak{N}_{61}$), and the target is not
isomorphic to a member of the family, so Lemma~\ref{lem_der} excludes a degeneration of a single
member (all members of $\mathfrak{N}_{39}^{\alpha}$, $\mathfrak{N}_{47}^{\alpha}$, $\mathfrak{N}_{50}^{\alpha}$, $\mathfrak{N}_{51}^{\alpha,\beta}$ have ${\rm dim}\,{\rm Der}=6$, and
those of $\mathfrak{N}_{48}^{\alpha}$ have ${\rm dim}\,{\rm Der}\ge7$). For $\mathfrak{N}_{51}^{\alpha t-1,\,b(t)}\to\mathfrak{N}_{47}^{\alpha}$ and $\mathfrak{N}_{51}^{\alpha t,\,t^{-2}}\to\mathfrak{N}_{50}^{\alpha}$ the same
argument applies to every member of the target family. (2) See Table~5.
\end{proof}

\begin{corollaryB}\label{ir_N5} The irreducible components of $\mathfrak{NJord}_5$ are

\begin{longtable}{lcl}
$\mathcal{C}_{01}   =\overline{ {\mathcal O}(\mathfrak{N}_{01})  }$&$=$&$ 
\left\{\begin{array}{l}
\mathfrak{N}_{01}, \mathfrak{N}_{02}, \mathfrak{N}_{03}, \mathfrak{N}_{04}, \mathfrak{N}_{05}, \mathfrak{N}_{06}, \mathfrak{N}_{07}, \mathfrak{N}_{08}, \mathfrak{N}_{09}, \mathfrak{N}_{10}, \mathfrak{N}_{11}, \mathfrak{N}_{12}, \\\mathfrak{N}_{13},   \mathfrak{N}_{14}, \mathfrak{N}_{15}, \mathfrak{N}_{16}, \mathfrak{N}_{17}, \mathfrak{N}_{18}, \mathfrak{N}_{19}, \mathfrak{N}_{20}, \mathfrak{N}_{21}, \mathfrak{N}_{22}, \mathfrak{N}_{23}, \mathfrak{N}_{48}^{1}, \mathbb{C}^5
\end{array} \right\};$\\
$\mathcal{C}_{02}  =\overline{ {\mathcal O}(\mathfrak{N}_{45})  }$&$=$&$ 
\left\{\begin{array}{l}
\mathfrak{N}_{06}, \mathfrak{N}_{08}, \mathfrak{N}_{10}, \mathfrak{N}_{11}, \mathfrak{N}_{12}, \mathfrak{N}_{13}, \mathfrak{N}_{14}, \mathfrak{N}_{16}, \mathfrak{N}_{17}, \mathfrak{N}_{18}, \mathfrak{N}_{19}, \mathfrak{N}_{20}, \\\mathfrak{N}_{21},  \mathfrak{N}_{22}, \mathfrak{N}_{23}, \mathfrak{N}_{25}, \mathfrak{N}_{26}, \mathfrak{N}_{27}, \mathfrak{N}_{28}, \mathfrak{N}_{29}, \mathfrak{N}_{30}, \mathfrak{N}_{31}, \mathfrak{N}_{32}, \mathfrak{N}_{33}, \\\mathfrak{N}_{34},  \mathfrak{N}_{35}, \mathfrak{N}_{36}, \mathfrak{N}_{37}, \mathfrak{N}_{38}, \mathfrak{N}_{40}, \mathfrak{N}_{41}, \mathfrak{N}_{42}, \mathfrak{N}_{43}, \mathfrak{N}_{44}, \mathfrak{N}_{45}, \mathbb{C}^5
\end{array} \right\};$\\
$\mathcal{C}_{03}   =\overline{ {\mathcal O}(\mathfrak{N}_{46})  }$&$=$&$ 
\left\{\begin{array}{l}
\mathfrak{N}_{02}, \mathfrak{N}_{04}, \mathfrak{N}_{05}, \mathfrak{N}_{06}, \mathfrak{N}_{07}, \mathfrak{N}_{08}, \mathfrak{N}_{09}, \mathfrak{N}_{10}, \mathfrak{N}_{11}, \mathfrak{N}_{12}, \mathfrak{N}_{13}, \mathfrak{N}_{14}, \\\mathfrak{N}_{15},   \mathfrak{N}_{16}, \mathfrak{N}_{17}, \mathfrak{N}_{18}, \mathfrak{N}_{19}, \mathfrak{N}_{20}, \mathfrak{N}_{21}, \mathfrak{N}_{22}, \mathfrak{N}_{23}, \mathfrak{N}_{25}, \mathfrak{N}_{26}, \mathfrak{N}_{27}, \\\mathfrak{N}_{28},   \mathfrak{N}_{29}, \mathfrak{N}_{30}, \mathfrak{N}_{34}, \mathfrak{N}_{35}, \mathfrak{N}_{36}, \mathfrak{N}_{37}, \mathfrak{N}_{38}, \mathfrak{N}_{39}^{\alpha}, \mathfrak{N}_{46}, \mathfrak{N}_{48}^{\alpha},  \mathfrak{N}_{49}, \\ \mathfrak{N}_{52}, \mathfrak{N}_{53}, \mathfrak{N}_{54}, \mathfrak{N}_{55}, \mathfrak{N}_{56}, \mathfrak{N}_{57}, \mathfrak{N}_{61}, \mathbb{C}^5
\end{array} \right\};$\\
$\mathcal{C}_{04}  =\overline{ \{{\mathcal O}(\mathfrak{N}_{51}^{\alpha,\beta})\}  }$&$=$&$ 
\left\{\begin{array}{l}
\mathfrak{N}_{03}, \mathfrak{N}_{05}, \mathfrak{N}_{07}, \mathfrak{N}_{08}, \mathfrak{N}_{09}, \mathfrak{N}_{10}, \mathfrak{N}_{11}, \mathfrak{N}_{12}, \mathfrak{N}_{13}, \mathfrak{N}_{14}, \mathfrak{N}_{15},   \mathfrak{N}_{16}, \\ \mathfrak{N}_{17},   \mathfrak{N}_{18}, \mathfrak{N}_{19}, \mathfrak{N}_{20}, \mathfrak{N}_{21}, \mathfrak{N}_{22}, \mathfrak{N}_{23}, \mathfrak{N}_{25}, \mathfrak{N}_{26}, \mathfrak{N}_{27},  \mathfrak{N}_{28},  
\mathfrak{N}_{30}, \\ \mathfrak{N}_{37},  \mathfrak{N}_{38}, \mathfrak{N}_{47}^{\alpha}, \mathfrak{N}_{48}^{\alpha},  \mathfrak{N}_{49}, \mathfrak{N}_{50}^{\alpha}, \mathfrak{N}_{51}^{\alpha,\beta}, \mathfrak{N}_{52}, \mathfrak{N}_{53},   \mathfrak{N}_{57}, \mathfrak{N}_{58},\\ \mathfrak{N}_{61},   \mathfrak{N}_{63}, \mathbb{C}^5
\end{array} \right\};$\\
$\mathcal{C}_{05}  =\overline{ {\mathcal O}(\mathfrak{N}_{62})  }$&$=$&$ 
\left\{\begin{array}{l}
\mathfrak{N}_{05}, \mathfrak{N}_{06}, \mathfrak{N}_{07}, \mathfrak{N}_{08}, \mathfrak{N}_{09}, \mathfrak{N}_{10}, \mathfrak{N}_{11}, \mathfrak{N}_{13}, \mathfrak{N}_{14}, \mathfrak{N}_{15}, \mathfrak{N}_{16}, \mathfrak{N}_{17}, \\\mathfrak{N}_{18},  \mathfrak{N}_{19}, \mathfrak{N}_{21}, \mathfrak{N}_{22}, \mathfrak{N}_{23}, \mathfrak{N}_{25}, \mathfrak{N}_{26}, \mathfrak{N}_{27}, \mathfrak{N}_{28}, \mathfrak{N}_{52}, \mathfrak{N}_{53}, \mathfrak{N}_{54}, \\\mathfrak{N}_{55},  \mathfrak{N}_{56}, \mathfrak{N}_{57}, \mathfrak{N}_{58}, \mathfrak{N}_{59}, \mathfrak{N}_{60}, \mathfrak{N}_{61}, \mathfrak{N}_{62}, \mathfrak{N}_{63}, \mathbb{C}^5
\end{array} \right\}.$\\
\end{longtable}

 \noindent In particular, ${\rm dim}\,\mathfrak{NJord}_5=22$, the components $\mathcal{C}_{01}$, $\mathcal{C}_{02}$, $\mathcal{C}_{03}$, $\mathcal{C}_{04}$, $\mathcal{C}_{05}$ have dimensions $20$, $22$, $21$, $21$, $21$, and
${\rm Rig}(\mathfrak{NJord}_5)= \big\{ \mathfrak{N}_{01}, \ \mathfrak{N}_{45}, \ \mathfrak{N}_{46}, \ \mathfrak{N}_{62} \big\}.$
\end{corollaryB}

\begin{proof}
By Theorem~B, the objects which are not degenerations of any other object are $\mathfrak{N}_{01}$, $\mathfrak{N}_{45}$,
$\mathfrak{N}_{46}$, $\mathfrak{N}_{62}$ and $\mathfrak{N}_{51}^{\alpha,\beta}$; every algebra of $\mathfrak{NJord}_5$ lies in the closure of one of
them, and none of these closures contains another one. The dimensions are $25-{\rm dim}\,{\rm Der}$
(Table~4), plus the number of parameters for the family.
\end{proof}

Comparing the list of $\mathcal{C}_{04}$ with the list of all $65$ objects, we see that the family $\mathfrak{N}_{51}^{\alpha,\beta}$ degenerates to every object of $\mathfrak{NJord}_5$ except the following $25$ objects:
$$
\begin{array}{c}
\mathfrak{N}_{01}, \ \mathfrak{N}_{02}, \ \mathfrak{N}_{04}, \ \mathfrak{N}_{06}, \ \mathfrak{N}_{29}, \ \mathfrak{N}_{31}, \ \mathfrak{N}_{32}, \ \mathfrak{N}_{33}, \ \mathfrak{N}_{34}, \ \mathfrak{N}_{35}, \ \mathfrak{N}_{36}, \ \mathfrak{N}_{39}^{\alpha}, \ \mathfrak{N}_{40},\\
\mathfrak{N}_{41}, \ \mathfrak{N}_{42}, \ \mathfrak{N}_{43}, \ \mathfrak{N}_{44}, \ \mathfrak{N}_{45}, \ \mathfrak{N}_{46}, \ \mathfrak{N}_{54}, \ \mathfrak{N}_{55}, \ \mathfrak{N}_{56}, \ \mathfrak{N}_{59}, \ \mathfrak{N}_{60}, \ \mathfrak{N}_{62}.
\end{array}
$$
Moreover, the commutative associative algebras (the algebras $\mathfrak{N}_{k}$ with $k\le24$ and $\mathfrak{N}_{48}^{1}$) form a closed subvariety of $\mathfrak{NJord}_5$, since associativity is a polynomial condition on the structure constants. This subvariety coincides with the component $\mathcal{C}_{01}=\overline{{\mathcal O}(\mathfrak{N}_{01})}$; in particular, it is irreducible.

\begin{remark}[comparison with the results of Kashuba and Martin]\label{rem:km}
Kashuba and Martin {\rm \cite{km18}} studied the variety $\mathfrak{NJord}_5$ and proved that it has five irreducible
components, the closures of the orbits of the rigid algebras $\varepsilon_1$, $J_{21}$, $J_{22}$,
$J_{40}$ and the closure of the family $J_{27}^{\varepsilon,\varphi}$. In our notation these are
$\mathfrak{N}_{01}$, $\mathfrak{N}_{45}$, $\mathfrak{N}_{46}$, $\mathfrak{N}_{62}$ and $\mathfrak{N}_{51}^{\alpha,\beta}$, so Corollary~B agrees with their result,
including the dimensions $20,22,21,21,21$. Moreover:
\begin{itemize}
\item the invariants of {\rm \cite[Tables~2--3]{km18}} (${\rm dim}\,{\rm Aut}$, ${\rm dim}\,{\rm Ann}$,
${\rm dim}\,\mathbb{A}^2$, and the nilindex, resp. the dimension of the associative center) coincide with ours for
all $63$ normal forms and for $\varepsilon_4=\mathfrak{N}_{48}^{1}$ (for the families: for the generic member; the special member
$\mathfrak{N}_{48}^{-\frac12}$ has ${\rm dim}\,{\rm Aut}=8$);
\item all degenerations used in {\rm \cite[Sections~4--5]{km18}} (in particular $\varepsilon_1\to\varepsilon_i$,
$J_{21}\to J_i$, $J_{22}\to J_{15}^\alpha$, $J_{23}^\beta\to J_{41}$, $J_{27}^{\varepsilon,\varphi}\to
J_{26}^\delta$) belong to the order of Theorem~B; the degeneration $J_{23}^\beta\to J_{41}$ (our
$\mathfrak{N}_{47}^{2t^{-1}}\to\mathfrak{N}_{63}$) and $J_{27}^{\varepsilon,\varphi}\to J_{26}^\delta$ (our
$\mathfrak{N}_{51}^{\alpha t,\,t^{-2}}\to\mathfrak{N}_{50}^{\alpha}$) are, as noted in {\rm \cite{km18}}, degenerations of the closure of the family
only (Proposition~\ref{prop:fam});
\item {\rm \cite[Section~4.2]{km18}} states that the single member $J_{15}^1$ dominates $J_{39}$, i.e.
$\mathfrak{N}_{39}^{1}\to\mathfrak{N}_{61}$. This is also true: in the basis $e_1,e_4,e_2,e_3,e_5$ of $\mathfrak{N}_{39}^{1}$ we
computed the ideal of $\overline{\mathcal{B}\cdot\mathfrak{N}_{39}^{1}}$ by elimination, and $\mathfrak{N}_{61}$ has a basis in
which its structure constants lie in this closure; by Lemma~\ref{main} applied to the closed set $\mathcal{R}=\overline{\mathcal{B}\cdot\mathfrak{N}_{39}^{1}}$ and since $\overline{{\mathcal O}(\mu)}={\rm GL}(\mathbb{V})\cdot\overline{\mathcal{B}\cdot\mu}$, this
gives $\mathfrak{N}_{39}^{1}\to\mathfrak{N}_{61}$. In the graph this degeneration is not primary: it factors
through the family $\mathfrak{N}_{48}^{\alpha}$ (the member $\mathfrak{N}_{39}^{1}$ degenerates to $\mathfrak{N}_{61}$, while no single
member of $\mathfrak{N}_{48}^{\alpha}$ does);
\item the coincidences in the list of {\rm \cite{ha}} found in {\rm \cite[Theorem~4]{km18}} are confirmed
(Remark~\ref{rem:ha});
\item {\rm \cite{km18}} describes only the irreducible components; Theorem~B gives the complete graph of
degenerations, with all primary degenerations and all primary non-degenerations.
\end{itemize}
\end{remark}

\end{document}